\documentclass[12pt]{article}
\usepackage[utf8]{inputenc}
\usepackage{amsthm,amsmath,amssymb,graphicx,enumerate,bbm,dsfont,subcaption,authblk}
\usepackage{multirow}
\usepackage{natbib}
\usepackage[colorlinks=true,linkcolor=violet,citecolor=violet]{hyperref}\usepackage{tikz,graphicx,multicol}
\usepackage[ruled,vlined]{algorithm2e}
\usepackage{geometry}
\usepackage{appendix}
\allowdisplaybreaks

\title{Canonical thresholding for non-sparse \\high-dimensional linear regression\thanks{Research supported by ONR grant N00014-19-1-2120, NSF grant DMS-1662139,  and NIH grant 2R01-GM072611-14. E-mail: {isilin@princeton.edu}, jqfan@princeton.edu}
}

\date{}

\newcommand{\eU}{\mathbf{\widehat{U}}}
\newcommand{\eV}{\mathbf{\widehat{V}}}
\newcommand{\eL}{\mathbf{\widehat{\Lambda}}}
\newcommand{\ebeta}{\boldsymbol{\widehat{\beta}}}
\newcommand{\eff}{\mathsf{D}^{\mathsf{eff}}}

\newcommand{\snr}{\mathsf{SNR}}
\newcommand{\reff}{\mathsf{r^{eff}}}

\newtheorem{theorem}{Theorem}[section]
\newtheorem{corollary}[theorem]{Corollary}
\newtheorem{lemma}[theorem]{Lemma}
\newtheorem{proposition}[theorem]{Proposition}
\newtheorem{assumption}{Assumption}[section]
\newtheorem{definition}{Definition}[section]
\newtheorem{remark}[]{Remark}

\numberwithin{equation}{section}      
\numberwithin{remark}{section}
\numberwithin{example}{section}

\DeclareMathOperator{\eqdef}{\;\stackrel{\operatorname{def}}{=}\;}

\DeclareMathOperator{\Prob}{\mathbb{P}}
\DeclareMathOperator{\E}{\mathbb{E}}

\DeclareMathOperator{\R}{\mathbb{R}}

\newcommand{\Id}{\mathbb{I}}
\newcommand{\Oo}{\mathbb{O}}
\DeclareMathOperator{\Tr}{\mathsf{Tr}}
\DeclareMathOperator{\rank}{\mathsf{rank}}

\DeclareMathOperator{\Fr}{\mathsf{F}}

\newcommand{\T}{\top}

\newcommand{\eps}{\varepsilon}

\newcommand{\XX}{\pmb{\mathbb{X}}}
\newcommand{\YY}{\pmb{\mathbb{Y}}}
\newcommand{\ZZ}{\pmb{\mathbb{Z}}}

\newcommand{\St}{\mathbf{\Sigma}}
\newcommand{\Se}{\mathbf{\widehat{\Sigma}}}

\newcommand{\bA}{\mathbf{A}}

\newcommand{\bC}{\mathbf{C}}

\newcommand{\bH}{\mathbf{H}}

\newcommand{\bU}{\mathbf{U}}
\newcommand{\bV}{\mathbf{V}}

\newcommand{\ba}{\mathbf{a}}
\newcommand{\bbb}{\mathbf{b}} 

\newcommand{\be}{\mathbf{e}}

\newcommand{\bh}{\mathbf{h}}

\newcommand{\br}{\mathbf{r}}

\newcommand{\bu}{\mathbf{u}}

\newcommand{\bw}{\mathbf{w}}
\newcommand{\bx}{\mathbf{x}}

\newcommand{\bz}{\mathbf{z}}

\newcommand{\bbeta}{\boldsymbol{\beta}}
\newcommand{\bgamma}{\boldsymbol{\gamma}}

\newcommand{\beps}{\boldsymbol{\varepsilon}}

\newcommand{\btheta}{\boldsymbol{\theta}}

\newcommand{\bxi}{\boldsymbol{\xi}}

\newcommand{\bpsi}{\boldsymbol{\psi}}
\newcommand{\bomega}{\boldsymbol{\omega}}

\newcommand{\bLambda}{\mathbf{\Lambda}}

\author[]{Igor Silin}
\author[]{Jianqing Fan}
\affil[]{Department of Operations Research and Financial Engineering\\ Princeton University}

\begin{document}

\maketitle
\vspace{-0.5cm}
\begin{abstract}
     We consider a high-dimensional linear regression problem. Unlike many papers on the topic, we do not require sparsity of the regression coefficients; instead, our main structural assumption is a decay of eigenvalues of the covariance matrix of the data. We propose a new family of estimators, called the canonical thresholding estimators, which pick largest regression coefficients in the canonical form. The estimators admit an explicit form and can be linked to LASSO and Principal Component Regression (PCR). A theoretical analysis for both fixed design and random design settings is provided. Obtained bounds on the mean squared error and the prediction error of a specific estimator from the family allow to clearly state sufficient conditions on the decay of eigenvalues to ensure convergence. In addition, we promote the use of the relative errors, strongly linked with the out-of-sample $R^2$. The study of these relative errors leads to a new concept of joint effective dimension, which incorporates the covariance of the data and the regression coefficients simultaneously, and describes the complexity of a linear regression problem.  Some minimax lower bounds are established to showcase the optimality of our procedure. Numerical simulations confirm good performance of the proposed estimators compared to the previously developed methods.
\end{abstract}


\newpage

\section{Introduction and Setup}
Consider the standard linear regression model
\begin{equation}
	\begin{aligned}
		y = \bx^\T \bbeta + \eps,
	\nonumber
	\end{aligned}
\end{equation}
where $\,\bx \in \R^d\,$ is a vector of covariates, $\,\bbeta \in \R^d\,$ is a vector of coefficients, $\,\eps \in \R\,$ is a noise term, and $\,y\in \R\,$ is a response.
Suppose we observe $\,n\,$ pairs $\,\{ (\bx_i, y_i) \}_{i=1}^n\,$ from this model with the assumption that the underlying noise terms $\,\{\eps_i\}_{i=1}^n\,$ are i.i.d. random variables with mean zero.
In matrix notations, introducing
\begin{equation}
	\begin{aligned}
		\YY = \begin{bmatrix} y_1 \\ \vdots \\ y_n \end{bmatrix} \,\in \R^n ,\;\;\;
		\XX = \begin{bmatrix} \bx_1^\T \\ \vdots \\ \bx_n^\T \end{bmatrix} \,\in \R^{n\times d},\;\;\;
		\beps = \begin{bmatrix} \eps_1 \\ \vdots \\ \eps_n \end{bmatrix} \,\in \R^n,
	\nonumber
	\end{aligned}
\end{equation}
we rewrite our model as
\begin{equation}
	\begin{aligned}
		\YY= \XX \bbeta + \beps.
	\label{model}
	\end{aligned}
\end{equation}

Define the covariance matrix of the data $\,\Se \eqdef n^{-1} \sum_{i=1}^n \bx_i\bx_i^\T = n^{-1} \XX^\T \XX\in \R^{d\times d}$. Our goal is to estimate the unknown $\,\bbeta\,$ and analyze the quality of estimation in two different settings:
\begin{itemize}
	\item \textit{Fixed design}. That means, the vectors of covariates $\,\{ \bx_i \}_{i=1}^n\,$ are deterministic (without loss of generality we assume $\,\sum_{i=1}^n \bx_i = 0$). A standard way to measure the error of an estimator $\,\widetilde{\bbeta}\,$ in this case is the \textit{mean squared error} (MSE):
\begin{equation}
	\begin{aligned}
		\mathsf{MSE}(\widetilde{\bbeta})  \,\eqdef\, \frac{1}{n}\sum\limits_{i=1}^n (\bx_i^\T \widetilde{\bbeta} - \bx_i^\T \bbeta)^2 = \frac{1}{n} \| \XX\widetilde{\bbeta} - \XX\bbeta\|_2^2 = (\widetilde{\bbeta} - \bbeta)^\T \Se (\widetilde{\bbeta} - \bbeta).
	\nonumber
	\end{aligned}
\end{equation}
This differs from the prediction error for the fixed design by an amount of $\,\E[ \varepsilon^2]\,$ (independent of the model) and reflects the model error in the prediction for this case.

	\item \textit{Random design}. In this scenario the vectors of covariates $\,\{ \bx_i \}_{i=1}^n\,$ come independently from some unknown distribution with mean zero (for simplicity) and the covariance matrix $\,\St \eqdef \E[\bx\bx^\T] \in \R^{d\times d}$. We are interested in the performance of an estimator $\,\widetilde{\bbeta}\,$ measured by the \textit{expected prediction error} (PE):
\begin{equation}
	\begin{aligned}
		\mathsf{PE}(\widetilde{\bbeta})  \,\eqdef\, \E\left[ (\bx^\T \widetilde{\bbeta} - \bx^\T \bbeta)^2 \right] = (\widetilde{\bbeta} - \bbeta)^\T \St (\widetilde{\bbeta} - \bbeta).
	\nonumber
	\end{aligned}
\end{equation}
This quantity  differs also from the prediction error for random design by $\,\E[\varepsilon^2]\,$ and equals the excess risk
\begin{equation}
	\begin{aligned}
		\mathsf{PE}(\widetilde{\bbeta})  = \E\left[ (y - \bx^\T \widetilde{\bbeta})^2 \right] - \E\left[ (y - \bx^\T \bbeta)^2 \right] .
	\nonumber
	\end{aligned}
\end{equation}
\end{itemize}
In the sequel we refer to these quantities simply as the (\textit{absolute}) MSE and PE.  The reason we give two names is to differentiate their statistical behavior in high dimensions and to avoid confusions at various discussions.
We will also motivate and analyze the  \textit{relative errors} $\,\mathsf{MSE}(\widetilde{\bbeta})/\mathsf{MSE}(0)\,$ and $\,\mathsf{PE}(\widetilde{\bbeta})/\mathsf{PE}(0)$. Surprisingly, the relative errors in this form, appearing naturally and being well-motivated, have not gained much attention in the literature (some related, but still quite different relative measures of performance in different contexts were considered in \cite{Dobriban1, Dobriban2, Dobriban3}). As we will see, the importance of the relative errors arises as a high-dimensional effect.

Being a fundamental statistical problem, the \textit{high-dimensional} linear regression has been approached in various ways. Probably the simplest method is Principal Component Regression (PCR). The idea is to reduce the dimension first via Principal Component Analysis (PCA) \citep{Pearson}, and then use several leading principal components as covariates to construct the least squares estimator. This approach heavily relies on a very strong assumption that the response depends on just a few leading principal components of the data. Various examples were provided where PCR performs poorly, see \cite{Jolliffe}. Another related idea to use supervised principal components was proposed by \cite{SPCR}.  See also Chapters 10 and 11 of \cite{FLZZ20} for further discussions and applications.

Over the past two decades, the main approach to tackle high-dimensionality of the problem has been the sparsity assumption on $\,\bbeta$, which is reasonable for many real-world applications. This has given rise to such model selection procedures  as LASSO \citep{LASSO}, SCAD \citep{SCAD}, Least Angle Regression \citep{LAR}, Dantzig selector \citep{Dantzig}, SLOPE \citep{SLOPE}. The list of papers devoted to these methods is too long to be presented here, so we just mention some of them: \cite{Ritov,Precond,Dantzig2,Dalalyan,SLOPE2}. Typically, a theoretical analysis of such procedures requires assumptions on the design like restricted isometry property (RIP), restricted eigenvalue (RE) condition, incoherence. These assumptions are needed to make sure that the correlations among subsets of features are small. See \cite{vandeGeer} for an overview of conditions used in the theoretical analysis of sparse linear regression. We also refer to Chapters 3--5 of \cite{FLZZ20} for an overview of existing methods and theoretical results for high-dimensional linear regression under sparsity.

The methods from the previous two paragraphs were developed (partially) due to a belief that the unconstrained least squares estimator is hopeless in high dimensions.
Recent papers by \cite{Bartlett} and later \cite{Chinot} have shown that the minimum $\ell_2$-norm least squares estimator $\,\widetilde{\bbeta}^{LS} \eqdef (\XX^\T \XX)^+\XX^\T \YY\,$ (where $\,(\XX^\T \XX)^+\,$ is the generalized inverse of the matrix $\,\XX^\T \XX$) can generalize well (i.e. have small absolute PE) even interpolating the training data -- they call this phenomenon ``benign overfitting''. To deliver convergence of $\,\mathsf{PE}(\widetilde{\bbeta}^{LS})\,$ to zero they require quite specific conditions on $\,\St$: the decay of its eigenvalues should be fast, but not too fast. These requirements are quantified by two notions of effective rank of $\,\St$.
 A closely related paper by \cite{Montanari} also studies \,$\widetilde{\bbeta}^{LS}$, but in the regime $\,p/n \to \gamma$. They are interested in the dependence of $\,\mathsf{PE}(\widetilde{\bbeta}^{LS})\,$ on $\,\gamma$, and focus on the case $\,\lambda_{\min}(\St) \geq c > 0$, e.g. considering isotropic and equicorrelation covariances. One more work on the topic is \cite{Belkin}, where the authors try to mathematically explain double descent phenomenon in several different models.

Going beyond the linear regression, one basic idea to approach general (nonlinear) regression problem $\,y = f(\bx) + \eps\,$ is to decompose the regression function
$\,f(\bx) \approx \sum_{j=1}^D \beta_j \psi_j(\bx) = \bpsi(\bx)^\T \bbeta\,$ over a Fourier basis, wavelet basis, or basis of eigenfunctions in reproducing kernel Hilbert space (RKHS), denoted here by $\,\psi_1(\cdot), \ldots, \psi_D(\cdot)$. This reduces the nonlinear regression problem to a linear one (potentially very high-dimensional or even infinite-dimensional), and allows to apply the existing methods.
Though we do not pursue the analysis of nonlinear regression in our work, this setting provides an excellent motivation for the main structural assumptions we make in our results. One of them is \textit{fast decay of the eigenvalues} of $\,\St\,$ (or $\,\Se$). For instance, we require that the \textit{effective rank}
\begin{align}
	\reff[\St] \eqdef \frac{\Tr[\St]}{\| \St\|} \;\;\;\left(\text{or }\;\reff[\Se] \eqdef \frac{\Tr[\Se]}{\|\Se\|} \right)
	\nonumber
\end{align}
can be well-controlled. The spectral decay has been observed in real-world datasets (e.g. MNIST, see Figure 5 in \cite{Rakhlin}; financial data in \cite{FinData1}, Figure 5; economics data in \cite{FinData2}, Figure 5), which makes our assumption reasonable. Importance of the eigenvalue decay (not only of the covariance, but of general kernel matrices) is highlighted in \cite{Rakhlin}, where such kind of conditions on the spectral decay is called ``favorable data geometry''. Moreover, \cite{EigenPro, BelkinPMLR} analyze the super-polynomial decay of eigenvalues of smooth kernel matrices. Going even further in deep learning literature, neural tangent kernels also exhibit the spectral decay, as shown by \cite{NTK}, among others. However, the fast eigenvalue decay is not the only motivation behind our work; another structural assumption that can make our results meaningful is a fast decay of regression coefficients in eigenbasis. This is a very well-understood condition as well: it is well-known that Fourier coefficients decay at a polynomial rate, where the degree depends on the smoothness of the underlying regression function. In addition, the decay of coefficients in RKHS was studied by \cite{BelkinPMLR}.

With these structural assumptions, the idea behind our family of estimators $\,\ebeta\,$ is quite natural: in some eigendirections (e.g. the ones that correspond to small eigenvalues of $\,\St$) we do not gain much by estimating the associated coefficient, so it makes sense to estimate only those components that allow to significantly reduce the error; specifically, we use thresholding to cut the components associated with the insignificant directions off. When applied to the nonlinear regression with wavelet basis, one estimator from the proposed family coincides with the soft thresholding approach studied in the series of papers by \cite{Donoho1,Donoho2,Donoho3,Donoho4,Donoho5}, among others. We highlight that we will not require sparsity of $\,\bbeta\,$ or any restrictive conditions on the design.

Let us summarize some motivations behind our work:
\begin{itemize}
\item Our methods can be viewed as an attempt to fix PCR by relaxing its restrictive assumptions. Instead of working with the several leading principal components, our estimators automatically screen for the most important principal components, not necessarily the leading ones.
\item Remarkably, the procedures that we propose are a modification of LASSO, so one can view this work as an attempt to extend LASSO to non-sparse high-dimensional linear regression.
\item Though the papers by \cite{Bartlett} and \cite{Chinot} do not advocate the use of interpolating estimators rather justify why the overfitting may not be harmful (very relevant question in modern deep learning research), we aim to show that there is no necessity to give up the in-sample denoising quality to get good bounds on the prediction error.  In fact, our numerical results show that our method is better the least squares estimator in various situations.
\end{itemize}
Main contributions of this paper are:
\begin{itemize}
	\item We propose a new method for high-dimensional linear regression, called \textit{Natural Canonical Thresholding} (NCT), in Section~\ref{S:estimator}. The connection of this approach to LASSO and PCR is discussed in Section~\ref{S:LASSO} and Section~\ref{S:PCR}. In Section~\ref{S:extension} we extend the suggested procedure and present a richer family of estimators, called \textit{Generalized Canonical Thresholding} (GCT). Our estimators $\,\ebeta\,$ are given via an explicit expression and do not require any optimization. Though each estimator from the family has one hyperparameter, it can be tuned in an efficient way via cross-validation as shown in Section~\ref{S:tuning}. An optimality result for the cross-validation is also presented.
	\item We provide theoretical guarantees for the NCT estimator in the fixed design and random design settings in Section~\ref{S:theory}. The presented tight bounds have two-fold meaning:
	\begin{itemize}
		\item For the absolute errors $\,\mathsf{MSE}(\ebeta)\,$ and $\,\mathsf{PE}(\ebeta)$, studied in Section~\ref{S:FixedDesign} and Section~\ref{S:RandomDesign}, we state explicit sufficient conditions of the form ``the eigenvalues of $\,\Se\,$ or $\,\St\,$ decay fast enough'' to ensure convergence in high dimensions. No conditions on $\,\bbeta\,$ are imposed in this case.
		\item For the relative errors $\,\mathsf{MSE}(\ebeta)/\mathsf{MSE}(0)\,$ and $\,\mathsf{PE}(\ebeta)/\mathsf{PE}(0)$, motivated in Section~\ref{S:RelMotivation}, our bounds factorize into the newly defined notion of the \textit{joint effective dimension}, the signal-to-noise ratio, and a vanishing factor. To get good rates for the relative errors in high dimensions it is not enough to assume fast decay of eigenvalues of $\,\Se\,$ or $\,\St\,$ alone, and we need to impose conditions of $\,\St\,$ and $\,\bbeta\,$ together (Section~\ref{S:JED1}), which is reflected by the joint effective dimension that we analyze (Section~\ref{S:JED2}).
	\end{itemize}
	We introduce parameter classes for linear regression problems with bounded joint effective dimension, and demonstrate a minimax optimality of the NCT estimator over this classes in the fixed design setting (Section~\ref{S:minimax}).	
	
	Theoretical analysis of the GCT estimator is not that insightful, however we still present and discuss a tight bound on the absolute error $\,\mathsf{MSE}(\ebeta)\,$ (Section~\ref{S:theoryGCT}).
	
	\item Numerical experiments, conducted in Appendix~\ref{S:experiments}, confirm good performance of NCT and especially GCT in comparison with other existing methods.
\end{itemize}
All the proofs are collected in Appendix~\ref{S:mainproofs} and Appendix~\ref{S:auxproofs}.
We conclude this section with defining some notations used throughout the work.

For a positive integer $\,k$, we write $\,[k]\,$ as shorthand for the set $\,\{ 1, 2, \ldots, k\}$.
We use $\,\Oo_{k\times l}\,$ for $\,k\times l\,$ matrix of zeros and $\,\Id_k\,$ for the identity matrix of size $\,k \times k$.
For a vector $\,\ba = [a_1, \ldots, a_k]^\T \in \R^k\,$ and $\,q > 0$, the standard $\ell_q$-(pseudo)norm in $\,\R^k\,$ is $\,\| \ba \|_q \eqdef \left( \sum_{j=1}^k |a_j|^q \right)^{1/q}$.  We use the following convention for the $\,\ell_0$-pseudonorm: $\,\| \ba\|_0^0 = \| \ba \|_0 \eqdef \sum_{j=1}^k \mathbbm{1}\{ a_j \neq 0\}$. Also, the $\,\ell_\infty$-norm is $\,\| \ba\|_{\infty} = \max_{j\in[k]} |a_j|$.
For a matrix $\,\bA$, let
$\,\| \bA\|\,$ be the spectral norm (the largest singular value), $\,\rank[\bA]\,$ be the rank, and (if $\,\bA\,$ is square) $\,\Tr[\bA]\,$ be the trace.

For sequences $\,a_n\,$ and $\,b_n\,$ the relation $\,a_n \lesssim b_n\, $ means that there exists a positive absolute constant $\,C\,$ such that $\, a_n \leq Cb_n\,$ for all $\,n$, while $\,a_n \asymp b_n\,$ means that $\,a_n \lesssim b_n\,$ and $\,b_n \lesssim a_n$. Oftentimes, similar notations will be used to denote inequalities/equalities up to a multiplicative constant not across the sample size $\,n$, but across all indices $\,j\in [d]\,$ or $\,j \in [\min(d, n)]$. The exact meaning in each case will be clear from the context. By $\,c, C\,$ we denote absolute constants which may differ from place to place.

Throughout the work, $\,\bbeta\,$ stands for the true vector of regression coefficients in the model \eqref{model}, $\,\ebeta\,$ stands for our NCT or GCT estimators proposed in the next section, and a generic estimator is denoted as $\,\widetilde{\bbeta}$. If we want to refer to an abstract vector, for example, as a variable in an optimization problem, we will be using $\bbeta^\prime$.

\section{Estimators} \label{S:estimator}
Let $\,r \eqdef \rank[\Se]$. Typically, $\,r = \min(d, n)$. Consider the SVD of the data matrix $\,\XX\,$ (scaled by $n^{-1/2}$):
\begin{equation}
	\begin{aligned}
		\frac{1}{\sqrt{n}}\,\XX = \eV \eL \eU^\T,
	\nonumber
	\end{aligned}
\end{equation}
where $\,\eL = \mathsf{diag}\left( \widehat{\lambda}_1^{1/2}, \ldots, \widehat{\lambda}_{r}^{1/2} \right) \,\in \R^{r\times r}\,$ is a diagonal matrix consisting of the non-zero singular values of $\,n^{-1/2}\XX\,$ in non-increasing order, the columns of $\,\eV \,\in \R^{n\times r}\,$ are the left singular vectors of $\,\XX$, and the columns of $\,\eU = [\widehat{\bu}_1, \ldots, \widehat{\bu}_{r}] \,\in \R^{d\times r}\,$ are the right singular vectors of $\,\XX$.
Alternatively, it is also convenient to think of the eigendecomposition of $\,\Se$:
\begin{equation}
	\begin{aligned}
		\Se = \eU \eL^2 \eU^\T,
	\nonumber
	\end{aligned}
\end{equation}
where now the diagonal entries $\,\widehat{\lambda}_1, \ldots, \widehat{\lambda}_{r}\,$ of $\,\eL^2\,$ are interpreted as the non-zero eigenvalues of $\,\Se\,$ in non-increasing order, and the columns $\,\widehat{\bu}_1, \ldots, \widehat{\bu}_{r}\,$ of $\,\eU\,$ are the corresponding eigenvectors of $\,\Se$.
Similarly, in what follows we will actively use the eigendecomposition of $\,\St$:
\begin{align}
	\St  = \bU \bLambda^2 \bU^\T,
	\nonumber
\end{align}
where $\,\bLambda^2 = \mathsf{diag}(\lambda_1, \ldots, \lambda_d) \,\in \R^{d\times d}\,$ is a diagonal matrix consisting of the eigenvalues of $\,\St\,$ in non-increasing order, and $\,\bU = [\bu_1, \ldots, \bu_d]\,\in\R^{d\times d}\,$ consists of the corresponding eigenvectors.

We introduce the following definition, which will be extensively used throughout the work.
\begin{definition}  \label{canonical}
Rewrite the linear regression model $\,\YY = \XX\bbeta + \beps\,$ as
\begin{equation}
\begin{aligned}
\YY &=(\XX \eU \eL^{-1}) (\eL \eU^\T \bbeta) + \beps
= \ZZ \btheta + \beps.
\nonumber
\end{aligned}
\end{equation}
We call this representation the \textit{canonical form} of the linear regression model.
Here $\,\ZZ \,\eqdef\, \XX \eU \eL^{-1} \,\in \R^{n\times r}\,$ is the standardized design matrix and
\begin{align}
\btheta \,\eqdef\, \eL \eU^\T \bbeta \,\in \R^r
\nonumber
\end{align}
 is the new vector of coefficients, called the \textit{canonical regression coefficients vector}, or simply \textit{canonical coefficients}.
\end{definition}

\noindent Note that the standardized design coincides with the left singular vectors $\,\ZZ = n^{1/2} \widehat\bV\,$ and satisfies the orthonormality constraints: $\,n^{-1} \ZZ^\T \ZZ = \Id_r$. Hence, the least-squares estimator for the canonical parameter is $\,\widetilde{\btheta}^{LS} = n^{-1}\ZZ^\T \YY = n^{-1} \eL^{-1} \eU^\T \,\XX^\T \YY$.  As in \cite{Fan96}, we further regularize the estimated canonical coefficients vector by either thresholding or truncation (setting higher indices to zero), depending whether $\,\btheta\,$ is approximately sparse or concentrates on the leading principal components.  Transforming the canonical parameter back to the original domain leads to the canonical thresholding estimator or principal component regression estimator, as to be further elaborated below.  Our work pushes forward the interactions between the canonical parameters and the design matrix.

More specifically, in the canonical domain our estimator looks like
\begin{align}
	\widehat{\btheta} \,\eqdef\, \mathsf{SOFT}_{\tau}\left[ \widetilde{\btheta}^{LS} \right],
\nonumber
\end{align}
which in the original domain brings us to the \textit{Natural Canonical Thresholding} (NCT) estimator of $\,\bbeta$, defined as
\begin{equation}
	\begin{aligned}
		\ebeta \eqdef \eU \eL^{-1} \,\mathsf{SOFT}_{\tau}\left[ \eL^{-1} \eU^\T \,\frac{\XX^\T \YY}{n}\right],
	\label{estimator}
	\end{aligned}
\end{equation}
where $\,\mathsf{SOFT}_\tau[z] \eqdef z\cdot \max\left( 1 - \tau/|z|, 0\right)\,$ is the \textit{soft thresholding} function applied component-wise and $\,\tau \geq 0\,$ is a hyperparameter to be chosen.
Note that in an overparameterized setting $\,d > n\,$ there are infinitely many $\,d$-dimensional  vectors leading to the same canonical coefficients. Specifically, for any vector $\,\bh \in \R^d\,$ the estimator $\,\ebeta + (\Id_d - \eU \eU^\T) \bh\,$ leads to the same vector of canonical coefficients: $\,\eL \eU^\T [\ebeta + (\Id_d - \eU \eU^\T) \bh] = \widehat{\btheta}$. While having the same in-sample fit, these estimators may produce different predictions for new points $\,\bx$. Among all of these estimators, $\,\ebeta\,$ from \eqref{estimator}, namely the estimator with $\,\bh=0$, has the smallest $\,\ell_2$-norm, arguably being the most reasonable choice.

Let us explain some intuition behind the NCT estimator.
Neglecting the noise term, we plug $\,\YY \approx \ZZ\btheta\,$ in, use the eigendecomposition of $\,\Se\,$ and get
\begin{equation}
	\begin{aligned}
		\widehat\btheta \approx \mathsf{SOFT}_{\tau}\left[ \btheta \right].
	\nonumber
	\end{aligned}
\end{equation}
Due to the structure of our error (e.g. in the fixed design case)
\begin{equation}
	\begin{aligned}
		\mathsf{MSE}(\ebeta)  = \| \eL \eU^\T\ebeta - \eL \eU^\T\bbeta \|_2^2
 = \| \widehat\btheta - \btheta\|_2^2
	\nonumber
	\end{aligned}
\end{equation}
and since we assume the eigenvalue decay, it is likely that some components $\,\theta_j\,$ do not play role, and the estimation of them with $\,\widehat{\theta}_j\,$ is not that important. Hence, it is reasonable to cut such insignificant components off, and this is exactly what the thresholding does. This reduces the variance of the estimator, while not increasing the bias by too much.

We also mention that when $\,\tau = 0$, our estimator reduces to the minimum $\ell_2$-norm least squares solution
\begin{equation}
	\begin{aligned}
		\widehat\btheta = \widetilde{\btheta}^{LS}\;\;\;\text{ and }\;\;\;\ebeta = \widetilde{\bbeta}^{LS} = \Se^+\,\frac{\XX^\T \YY}{n} = (\XX^\T \XX)^+\, \XX^\T \YY
	\nonumber
	\end{aligned}
\end{equation}
(unbiased or slightly biased, but with large variance), while $\,\tau = +\infty\,$ corresponds to the trivial solution $\,\widehat\btheta = 0\,$ and $\,\ebeta = 0\,$ (very biased, but with zero variance).

\subsection{Relation to LASSO} \label{S:LASSO}
Recall that the standard LASSO estimator is a solution of the following optimization problem:
\begin{equation}
	\begin{aligned}
		\widetilde{\bbeta}^{LASSO} \in \arg\min\limits_{\bbeta^\prime \in \R^d} \left\{ \frac{1}{2n} \| \YY  - \XX\bbeta^\prime \|_2^2 + \tau \| \bbeta^\prime \|_1\right\}.
	\nonumber
	\end{aligned}
\end{equation}
In practice one usually standardizes the columns of $\,\XX\,$ so that they are on the same scale and the coefficients corresponding to different covariates are penalized equally. Now imagine that we standardize our $\,\XX\,$ in the canonical manner as in Definition~\ref{canonical}.
If we run LASSO for the vector of coefficients $\,\btheta$, then the solution is expressed via the soft thresholding:
\begin{equation}
	\begin{aligned}
		\widetilde{\btheta}^{LASSO} &\,=\, \arg\min\limits_{\btheta^\prime \in \R^r} \;\left\{ \frac{1}{2n} \| \YY  - \ZZ\btheta^\prime \|_2^2 + \tau \| \btheta^\prime \|_1\right\}
		=\mathsf{SOFT}_\tau\left[ \frac{\ZZ^\T \YY}{n} \right],
	\nonumber
	\end{aligned}
\end{equation}
which is exactly our estimator $\,\widehat\btheta\,$ in the canonical domain.
Going back to $\,\ebeta = \eU \eL^{-1} \,\widetilde{\btheta}^{LASSO}\,$ we recover the NCT estimator~\eqref{estimator}. The solution is the soft thresholding on the canonical regression coefficients. This is why we call the method canonical thresholding.

We also note that the NCT estimator can be represented as the min-$\ell_2$-norm solution of the optimization problem
\begin{equation}
	\begin{aligned}
		\ebeta \in \arg\min\limits_{\bbeta^\prime \in \R^d} \left\{ \frac{1}{2n} \| \YY  - \XX\bbeta^\prime \|_2^2 + \tau \| \eL\eU^\T \bbeta^\prime \|_1\right\}.
	\nonumber
	\end{aligned}
\end{equation}
Our estimator is nothing more than LASSO penalized on the canonical regression coefficients.

\subsection{Relation to PCR} \label{S:PCR}
Principal Component Regression (PCR) approaches the high dimensionality of the problem by taking only $\,m\,$ ($m < r = \min(d,n)$) leading principal components of the original data.
The new design matrix becomes
\begin{equation}
	\begin{aligned}
		\ZZ_m =  \XX \eU_{\leq m} \eL_{\leq m}^{-1} \,\in \R^{n \times m},
	\nonumber
	\end{aligned}
\end{equation}
where $\,\eU_{\leq m} \in \R^{d\times m}\,$ consists of the first $\,m\,$ columns of $\,\eU\,$ and $\,\eL_{\leq m} \in \R^{m \times m}\,$ is $\,m\times m\,$  leading principal submatrix of $\,\eL$.
The new regression problem
\begin{equation}
	\begin{aligned}
		\YY &= \ZZ_m \btheta_m + \beps
	\nonumber
	\end{aligned}
\end{equation}
is solved via the least squares, yielding the solution
\begin{equation}
	\begin{aligned}
		\widetilde{\btheta}^{LS}_m \,=\, \frac{\ZZ_m^\T\,\YY}{n} \,\in \R^{m},
	\nonumber
	\end{aligned}
\end{equation}
and thus
\begin{equation}
	\begin{aligned}
		\widetilde{\bbeta}^{PCR} \,\eqdef\, \eU_{\leq m} \eL_{\leq m}^{-1} \,\widetilde{\btheta}^{LS}_m \;\in \R^d.
	\nonumber
	\end{aligned}
\end{equation}
Note that $\,\widetilde{\btheta}^{LS}_m\,$ is essentially the first $\,m\,$ components of $\,\widetilde{\btheta}^{LS}$,
and we can express
\begin{equation}
	\begin{aligned}
		\widetilde{\bbeta}^{PCR} = \eU \eL^{-1} \,\mathsf{ZERO}_m\left[ \eL^{-1} \eU^\T \,\frac{\XX^\T \YY}{n}\right],
	\nonumber
	\end{aligned}
\end{equation}
where $\,\mathsf{ZERO}_m[\bz] = [z_1, \ldots, z_m, 0, \ldots, 0]^\T\,$ is the operator zeroing out all the components of a vector $\,\bz \in \R^r\,$ except the first $\,m$. (Again, in an overparameterized case this estimator has the smallest $\ell_2$-norm among all estimators leading to the canonical coefficients $\,\widetilde{\btheta}^{LS}_m$.) The similarity of the PCR estimator to the NCT estimator~\eqref{estimator} is now clear.
While PCR blindly selects the coefficients corresponding to the first $\,m\,$ principal components, our procedure screens for the most important principal directions, which may be different from the leading ones, and leaves only those with significant contribution exceeding $\,\tau$. However, if there is a strong prior indicating the canonical coefficients spike at the principal directions, then of course PCR should also be a suitable procedure, and NCT will simply mimic its behavior, with some small costs.  The above contrasts between truncation and thresholding appear in \cite{Fan96} under a simpler model.

\subsection{Extension to family of canonical thresholding estimators} \label{S:extension}
PCR focuses only on the estimators on the principal component directions, and NCT does not have any preferences.  We now propose a family of canonical thresholding estimators to bridge these two extremes, progressively putting more preferences on the principal directions.  In addition, we also generalize the thresholding function.

First, the soft thresholding can be replaced by \textit{generalized thresholding rules} (see e.g. Definition 9.3 in \cite{FLZZ20}), introduced for completeness in the following definition.
\begin{definition} \label{thres}
 The function $\,\mathsf{T}_\tau: \R \to \R\,$ is called a generalized thresholding function, if
 \begin{enumerate}[(i)]
 	\item $\left|\mathsf{T}_\tau[z]\right| \leq c|z^\prime|\,$ for all $\,z, z^\prime\,$ satisfying $\,|z-z^\prime| \leq \tau/2\,$ and some constant $\,c$;
 	\item $\left|\mathsf{T}_\tau[z] - z\right| \leq \tau\,$ for all $\,z \in  \R$.
 \end{enumerate}
 The parameter $\,\tau\,$ is called the thresholding level.
\end{definition}

Second, if there is some prior that the spike canonical coefficients are more likely to be in the lower principal components rather than in the higher ones, we can introduce additional multiplicative weights, equal to the eigenvalues raised to a non-negative power $\,\varphi/2$, under thresholding to reflect this preference.  This is equivalent to applying a larger thresholding on higher principal components, with $\varphi$ controlling the degree.
Implementing this strategy, we propose the following more general family of estimators, parameterized by $\,\varphi \geq 0$:
\begin{equation}
	\begin{aligned}
		\widehat{\btheta} \,\eqdef\, \eL^{-\varphi} \;\mathsf{T}_{\tau}\left[ \eL^{\varphi} \;\widetilde{\btheta}^{LS}\right]
	\nonumber
	\end{aligned}
\end{equation}
in the canonical domain, or
\begin{equation}
	\begin{aligned}
		\ebeta \,\eqdef\, \eU \eL^{-1-\varphi} \;\mathsf{T}_{\tau}\left[ \eL^{-1+\varphi}\,\eU^\T \,\frac{\XX^\T \YY}{n}\right]
	\label{estimator2}
	\end{aligned}
\end{equation}
in the original domain. Here $\,\mathsf{T}_\tau[\bz]\,$ is a generalized thresholding function from Definition~\ref{thres}  applied component-wise and $\,\tau \geq 0\,$ is a hyperparameter to be chosen. The estimators from this family are called the \textit{Generalized Canonical Thresholding} (GCT) estimators.

When
$\,\varphi = 0\,$ and the soft thresholding  function is used, GCT reduces to the NCT estimator~\eqref{estimator}.
The intuition behind GCT is somewhat similar to NCT: the estimators automatically screen the most significant principal components. However, the choice of $\,\varphi\,$ allows to put different importance to eigenvalue $\,\widehat\lambda_j\,$ and projection $\,\widehat\bu_j^\T\bbeta\,$ when deciding whether to threshold $\,j$-th principal component or not. While in NCT this importance is calibrated in accordance to the scaling appearing in the decomposition of the absolute MSE (this justifies the word ``Natural'' in the name), GCT with $\,\varphi > 0\,$ gives more weight to the leading canonical coefficients, making the method closer to PCR, and selects other components when absolutely necessary.

There is one common situation where PCR is preferable:  pervasive latent factors that drives both the covariates and the response \citep{FWY17}.  In this case, the principal components are used to learn latent factors and these learned factors are used as the covariates for regressing the response $y$.  This leads to PCR.  We introduce GCT to better accommodate this situation. Figure~\ref{F:NCTvsPCR} visually illustrates the conceptual difference between NCT, GCT, and PCR approaches. We highlight once again that the selection of the components in NCT and GCT is data-driven, unlike in PCR where the selected components are fixed before the data are observed (modulo cross-validation, which helps to select the number of leading components, but not the components itself).

\begin{figure}
     \begin{subfigure}{\textwidth}
         \centering\includegraphics[width=0.6\textwidth]{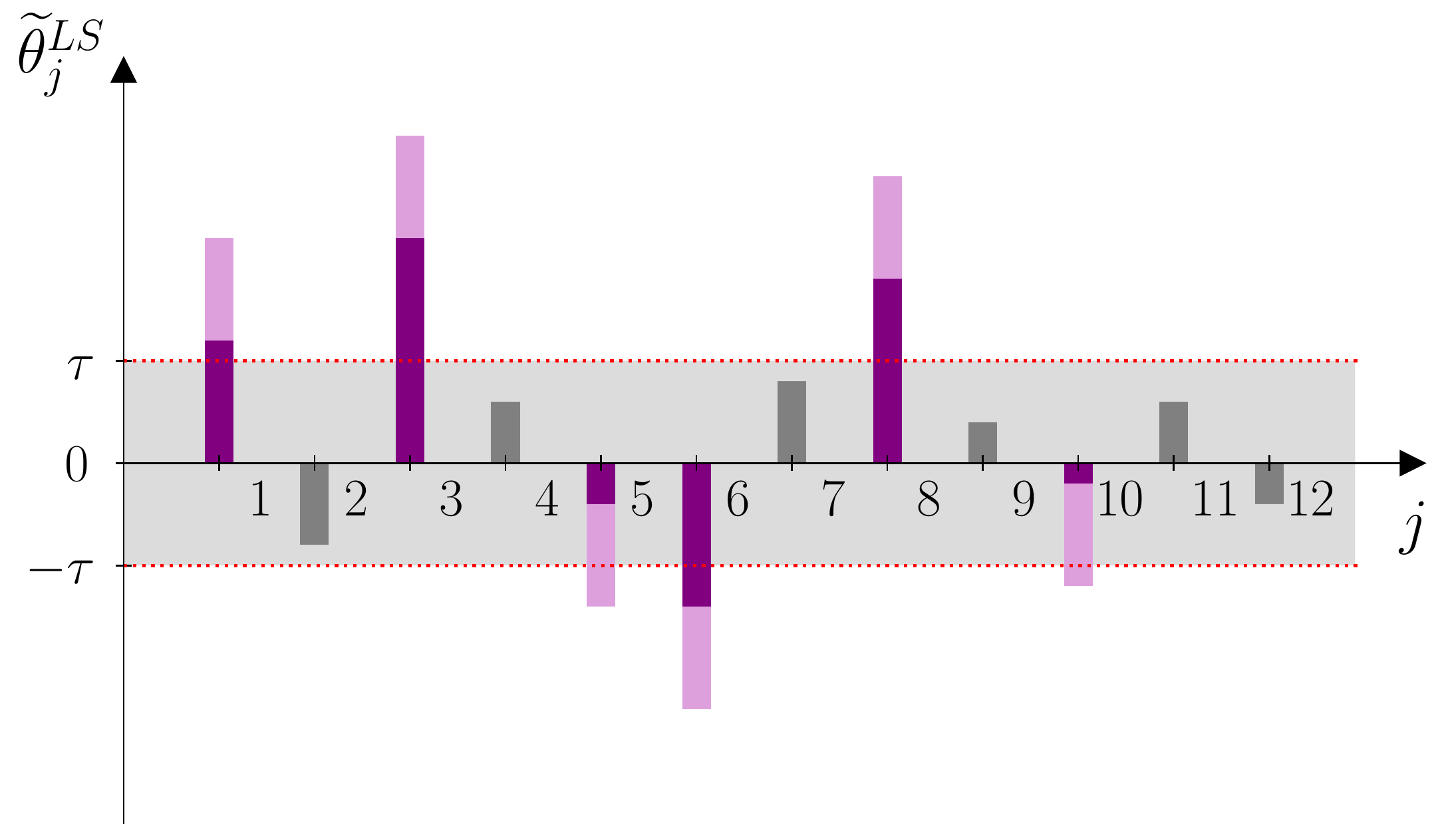}
         \vspace{-0.3cm}
         \caption{Natural Canonical Thresholding: light purple indicates the portion of the active coefficients subtracted during soft thresholding, and purple depicts the remaining surviving coefficients;}
     \end{subfigure}
     \vfill
     \vspace{0.2cm}
     \begin{subfigure}{\textwidth}
         \centering\includegraphics[width=0.6\textwidth]{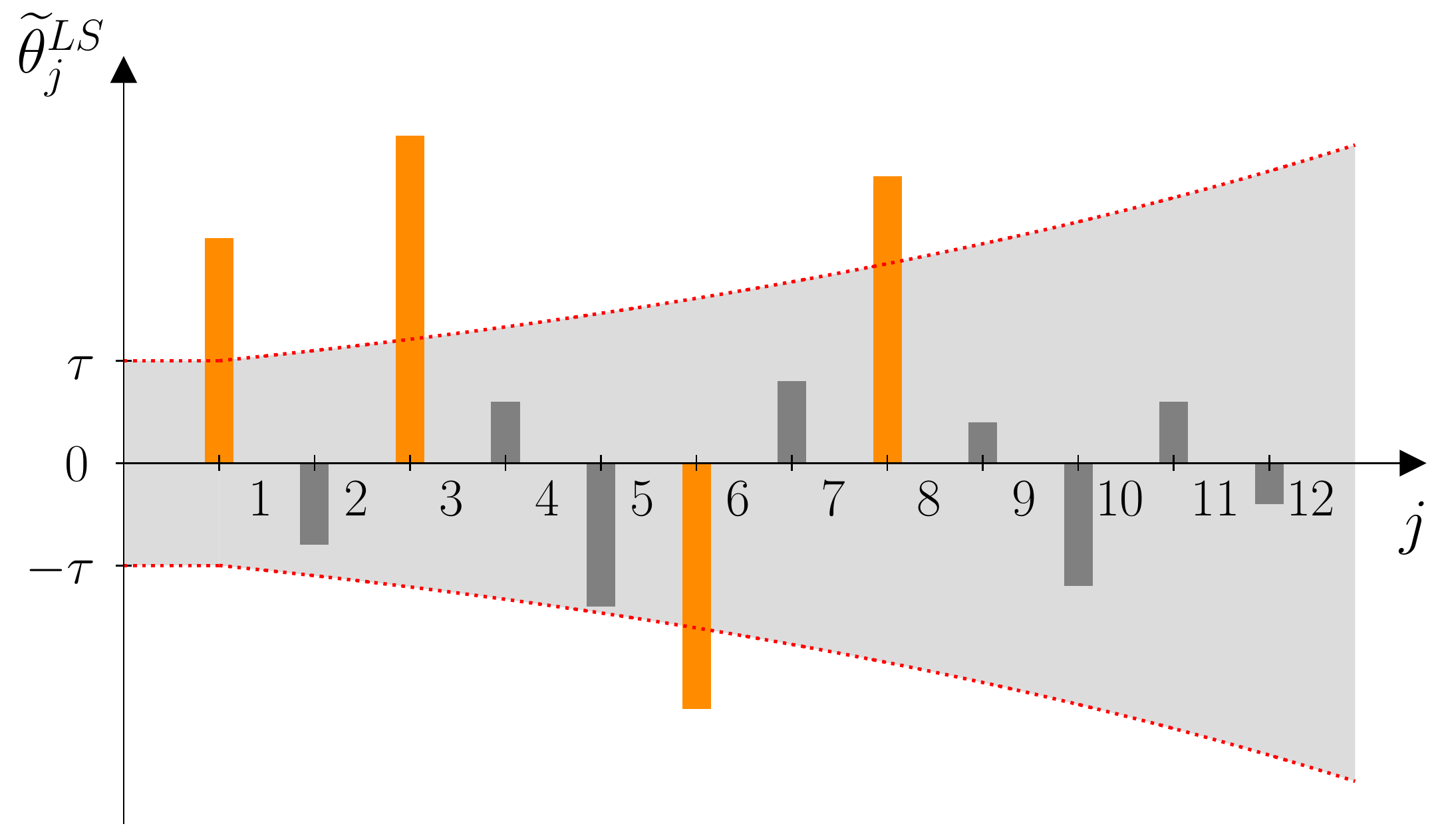}
         \vspace{-0.2cm}
         \caption{Generalized Canonical Thresholding with the hard thresholding rule, $\varphi = 1$, and assumed eigenvalue decay $\widehat{\lambda}_j = 1.21^{(-j+1)}$;}
     \end{subfigure}
     \vfill
     \vspace{0.2cm}
     \begin{subfigure}{\textwidth}
         \centering\includegraphics[width=0.6\textwidth]{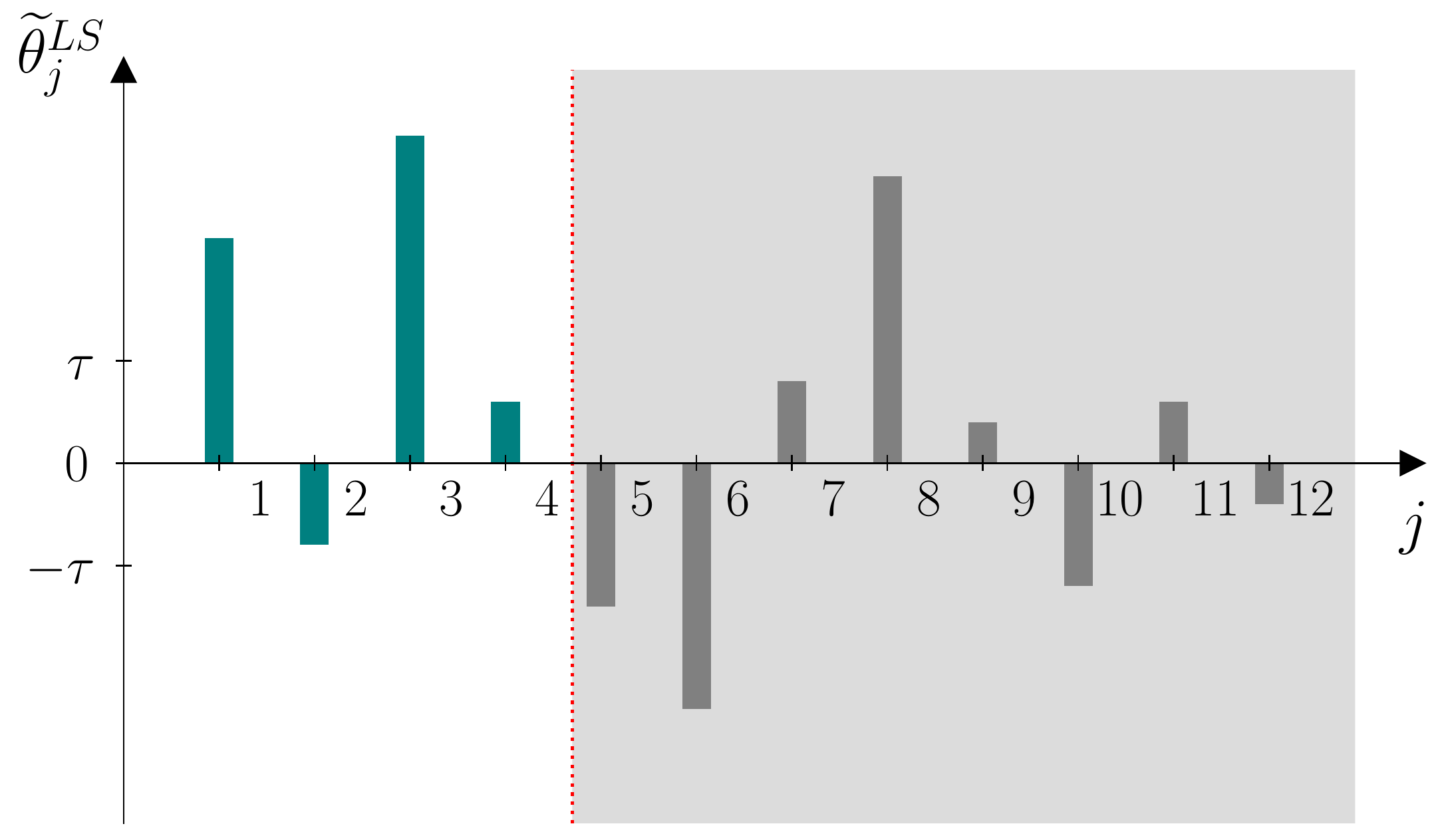}
         \vspace{-0.2cm}
         \caption{Principal Component Regression with $m = 4$;}
        \end{subfigure}
       \caption{Comparison of NCT, GCT, and PCR estimators on an artificial example with $r = 12$. On the horizontal axes -- an index of the component $j$, on the vertical axes -- the coefficient of the canonical least squares $\widetilde{\theta}^{LS}_j$. The red dotted lines depict the thresholding boundaries. The coefficients falling into the shaded area are thresholded/truncated to zero and depicted in gray. The coefficients surviving the thresholding/truncation are depicted in purple, orange, and teal, respectively.}
        \label{F:NCTvsPCR}
\end{figure}

It turns out that the theoretical results for GCT are not that nice and insightful as for NCT.
However, in practice we observed that GCT may behave much better in some scenarios, as to be shown in the corresponding section. In principle, one can even tune $\,\varphi\,$ in addition to tuning $\,\tau\,$ via cross-validation, which might further enhance the practical utility of the procedure.

\section{Theoretical properties of the NCT estimator} \label{S:theory}

The first condition needed for our theoretical analysis is the following assumption on the noise, which will be used in both fixed design and random design settings.

\begin{assumption}[Sub-Weibull noise] \label{A:Noise}
	The noise vector $\,\beps\,$ is independent of $\,\XX\,$ and is jointly sub-Weibull random vector with parameter $\,0 < \alpha \leq 2\,$ (see \cite{Weibull}). That is, there exists  $\,\sigma < \infty\,$ such that
    \begin{equation}
            \begin{aligned}
            	\sup\limits_{\|\bw\|_2=1} \| \bw^{\top} \beps\|_{\psi_\alpha} \leq \sigma,
                \nonumber
            \end{aligned}
        \end{equation}
    where $\,\|\cdot\|_{\psi_\alpha}\,$ is the Orlicz norm for $\,\psi_\alpha = e^{x^\alpha} - 1$.
    The following tail bound takes place:
    \begin{equation}
            \begin{aligned}
            	& \Prob\left[ |\bw^{\top} \beps| \geq t \right] \leq 2\,\exp\left(-\left(t/\sigma\right)^\alpha\right)
                \nonumber
            \end{aligned}
        \end{equation}
    for all $\,\bw, \, \|\bw\|_2= 1\,$ and $\,t > 0$.
\end{assumption}
\noindent This allows to go slightly beyond sub-Gaussian and sub-Exponential tails.
For i.i.d. sub-Gaussian noise, $\,\sigma^2\,$ coincides (up to a multiplicative constant) with the variance of a single $\,\eps_i$. So, $\,\sigma\,$ can be interpreted as the magnitude of the noise.

Define the \textit{signal-to-noise ratio} of the linear regression problem in the fixed design setting as
	\begin{equation}
	\begin{aligned}
		\snr \,\eqdef\, \left( \frac{n^{-1} \sum_{i=1}^n (\bx_i^\T \bbeta)^2}{\sigma^2} \right)^{1/2} = \left( \frac{\bbeta^\T \Se \bbeta}{\sigma^2}\right)^{1/2} = \frac{\| \btheta \|_2}{\sigma}\, ,
	\nonumber
	\end{aligned}
\end{equation}
while for the random design we use
	\begin{equation}
	\begin{aligned}
		\snr \,\eqdef\, \left( \frac{\E[(\bx^\T \bbeta)^2]}{\sigma^2} \right)^{1/2} = \left( \frac{ \bbeta^\T \St \bbeta }{\sigma^2} \right)^{1/2} = \frac{\| \bLambda \bU^\T \bbeta\|_2}{\sigma}\,.
	\nonumber
	\end{aligned}
\end{equation}
For our problem to be meaningful, we assume for the rest of the work that $\,\snr > 0\,$ in both settings.

Recall the canonical regression coefficients $\,\btheta = \eL \eU^\T \bbeta\,$ for the fixed design.  Its normalized version $\,\btheta/\|\btheta\|_2\,$ has the first $k$ components $\,\btheta_{\leq k} / \|\btheta\|_2\,$ where $\,\btheta_{\leq k} = \eL_{\leq k} \eU_{\leq k}^\T \bbeta$.  Here as usual $\,\eL_{\leq k}\in\R^{k\times k}\,$ is the leading principal submatrix of $\,\eL\,$ (containing the square roots of the first $\,k\,$ eigenvalues of $\,\Se\,$ on the diagonal) and $\,\eU_{\leq k} \in \R^{d\times k}\,$ is the matrix consisting of the first $\,k\,$ columns of $\,\eU\,$ (which are the $\,k\,$ leading eigenvectors of $\,\Se$).  Measuring the first $\,k\,$ normalized components $\,\btheta_{\leq k}/\|\btheta\|_2\,$ in $\,\ell_q$-(pseudo)norm gives
\begin{equation}
	\begin{aligned}
	\eff _{q, k}(\Se, \bbeta) \,\eqdef\, \frac{\| \btheta_{\leq k} \|_q^q}{\|\btheta\|_2^q}.
	\nonumber
	\end{aligned}
\end{equation}
We call this quantity the \textit{joint effective dimension of order $q$ up to index $k$} of $\,\Se\,$ and $\,\bbeta$.  Note that when $\,q=2$, it measures the proportion of $\,\btheta\,$ explained by $\,\btheta_{\leq k}\,$; when $\,q=0$, it counts the sparsity among $\,\btheta_{\leq k}$. Similar quantity can be defined for the random design setting:
\begin{equation}
	\begin{aligned}
	\eff _{q, k}(\St, \bbeta) \,\eqdef\, \frac{\| \bLambda_{\leq k} \bU_{\leq k}^\T \bbeta \|_q^q}{\|\bLambda \bU^\T \bbeta\|_2^q}.
	\nonumber
	\end{aligned}
\end{equation}
It turns out that this joint effective dimension
will play crucial role in our bounds for the NCT estimator~\eqref{estimator}.

For shortness, we introduce the following quantity that will be appearing regularly throughout the section:
\begin{align}
	\rho \,\eqdef\, \frac{2}{\sqrt{n}}\left(\log(2r/\delta)\right)^{1/\alpha},
\label{def:rho}
\end{align}
where $\,\delta\,$ is from the statements ``with probability $1-\delta$...''.
The thresholding level $\,\tau\,$ for both NCT and GCT will be expressed in terms of $\,\rho$.

%

\subsection{Fixed design} \label{S:FixedDesign}
We first provide a simple guarantee on the mean squared error $\,\mathsf{MSE}(\ebeta)\,$ of the NCT estimator~\eqref{estimator}.
\begin{theorem} \label{Th1}
	Suppose Assumption~\ref{A:Noise} is satisfied. Take $\,\tau = \sigma \rho\,$ with $\,\rho $ given by \eqref{def:rho}.
Then, with probability $\,1-\delta$, the NCT estimator $\,\ebeta\,$ from \eqref{estimator} with thresholding at level $\,\tau\,$ satisfies
	\begin{align}
		\mathsf{MSE}(\ebeta) & \asymp \sum\limits_{j=1}^r \min\left( \sigma\rho,\, |\theta_j|\right)^2 \label{bound0}\\
		&\leq
		\inf\limits_{q \in [0, 2] } \left\{ \| \btheta \|_q^q \;(\sigma \rho)^{2-q}\right\}
		\label{bound1}  \\
		&= \mathsf{MSE}(0)\,\inf\limits_{q \in [0, 2] } \left\{  \eff_{q, r}(\Se,\bbeta)\; \left[\snr^{-2}\,  \frac{\left(\log (2r/\delta)\right)^{2/\alpha}}{n}\right]^{1-q/2}\right\}.
	\label{bound2}
	\end{align}
\end{theorem}
\noindent The proof of this result almost repeats the classical proof for hard and soft thresholding in case of orthonormal design.
\begin{remark}[Choice of $\tau$] \label{R1}
	The choice of $\,\tau\,$ in the above theorem depends on the noise magnitude $\,\sigma$, the probability $\,\delta$, and the quantity $\,\alpha$, but this is not a significant problem. Later in Section~\ref{S:tuning} we will show how to tune $\,\tau\,$ using an efficient cross-validation procedure.
\end{remark}

In Theorem~\ref{Th1} we present several bounds on $\,\mathsf{MSE}(\ebeta)$. Note that the bound~\eqref{bound0} is tight. One can argue that the matching lower bound is due to the excessive bias introduced by soft thresholding. However, even if we replace soft thresholding with hard thresholding $\,\mathsf{HARD}_{\tau}[z] \eqdef z\cdot \mathbbm{1}\{ |z| \geq \tau\}$, the upper bound in probability stays the same as in~\eqref{bound0}, and at the same time it is not difficult to get a nearly matching lower bound (for Gaussian noise $\beps \sim \mathcal{N}(0, \sigma^2\Id_n)$) in expectation:
	\begin{align}
		\E\left[ \mathsf{MSE}(\ebeta) \right]  \,\gtrsim\, \sum\limits_{j=1}^r \min\left( \frac{\sigma}{\sqrt{n}},\, |\theta_j|\right)^2.
		\nonumber
	\end{align}
In any case, the bound~\eqref{bound0} is not really interpretable. The bounds~\eqref{bound1} and \eqref{bound2} have much deeper intuition as we will see later. Nevertheless, a reasonable question is how tight the inequality leading from  \eqref{bound0} to \eqref{bound1} is. Though in general the opposite inequality (up to a multiplicative constant) does not seem to hold, the following proposition states that in several important cases the inequality is actually tight (or almost tight).
\begin{proposition} \label{Prop:tight}
	(i) Sparsity. Assume $\,\| \theta \|_0 = s \ll r\,$ with $\,|\theta_j| \gtrsim \sigma\rho\,$ for all $\,j\,$ such that $\,\theta_j \neq 0$. Then (taking $\,q=0$)
	\begin{align}
		\sum\limits_{j=1}^r \min\left( \sigma\rho,\, |\theta_j|\right)^2 \,\gtrsim\,
		\inf\limits_{q \in [0, 2] } \left\{ \| \btheta \|_q^q \;(\sigma \rho)^{2-q}\right\}.
		\nonumber
	\end{align}
	\\
	(ii) Approximate sparsity. Assume there exists a small $\,\nu > 0\,$ such that $\,\| \btheta \|_\nu \lesssim \sigma$, and there are at least $\,s\,$ significant canonical coefficients: $\,|\{ j\in[r]\,:\; |\theta_j| \gtrsim\sigma\rho \}| \geq s$. Then (taking $\,q = \nu$)
	\begin{align}
		\sum\limits_{j=1}^r \min\left( \sigma\rho,\, |\theta_j|\right)^2 \,\gtrsim\,
		\inf\limits_{q \in [0, 2] } \left\{ \| \btheta \|_q^q \;(\sigma \rho)^{2-q}\right\} \times s\rho^\nu.
		\nonumber
	\end{align}
	\\
	(iii) Polynomial decay. Let $\,|\theta_{(1)}| \geq |\theta_{(2)}| \geq \ldots \geq |\theta_{(r)}|\,$ be the absolute values of the components of $\,\btheta\,$ arranged in descending order. Assume polynomial decay of the ordered canonical coefficients: for some $\,a > 0\,$ holds
	\begin{align}
		\frac{|\theta_{(j)}|}{|\theta_{(1)}|} \asymp j^{-a}\;\;\;\;\;\;\;\;\text{ for all}\;\;\;j\in [r].
		\nonumber
	\end{align}
	Then (taking $\,q = \min(1/a,\, 2)$)
		\begin{align}
		\sum\limits_{j=1}^r \min\left( \sigma\rho,\, |\theta_j|\right)^2 \,\gtrsim\,
		\inf\limits_{q \in [0, 2] } \left\{ \| \btheta \|_q^q \;(\sigma \rho)^{2-q}\right\} \times \begin{cases}
		1/\log(r), & a \geq 1/2,\\
			1, & a < 1/2.
		\end{cases}
		\nonumber
	\end{align}
\end{proposition}

Continuing the discussion on the bounds presented in Theorem~\ref{Th1}, we claim that though \eqref{bound1} and \eqref{bound2} coincide, the way we state them reflects two different messages.
The first one, if we take $\,q=1\,$ and apply the inequality $\,\| \btheta\|_1 \leq  \| \Se\|^{1/2} \| \bbeta\|_2 \reff[\Se]^{1/2}\,$ (which follows from the Cauchy-Schwarz inequality), then from the bound~\eqref{bound1} we get
		\begin{equation}
	\begin{aligned}
		\mathsf{MSE}(\ebeta) &\lesssim \sigma\|\Se\|^{1/2} \| \bbeta\|_2 \sqrt{ \frac{\reff[\Se] \,\left(\log (2r/\delta)\right)^{2/\alpha}}{n}}
	\label{effective_rank_bound}
	\end{aligned}
\end{equation}
with high probability.
This means that if one is interested in the absolute error $\,\mathsf{MSE}(\ebeta)$, then essentially $\,\reff[\Se]/n = o(1)\,$ is enough to guarantee $\,\mathsf{MSE}(\ebeta) = \sigma\| \Se\|^{1/2}\|\bbeta\|_2 \cdot o(1)$, omitting logarithmic terms.
No additional assumptions on $\,\bbeta\,$ are required, and there is no necessity to worry about the joint effective dimension and the signal-to-noise ratio from the bound~\eqref{bound2} in this situation. Note that the bound~\eqref{effective_rank_bound} is tight up to a logarithmic factor when $\,\widehat{\lambda}_j = \widehat{\lambda}_1/j\,$ and $\,|\widehat{\bu}_j^\T \bbeta| \asymp \| \bbeta\|_2/(\sqrt{j \log(r)})$.
	
Nevertheless, the bound~\eqref{bound2} is useful to better understand the structure of the error. Taking into account that the main motivation in our work is the decay of eigenvalues of $\,\Se\,$ or $\,\St$, it may easily be the case that even the trivial estimator $\,\widetilde{\bbeta} = 0\,$ has a very small error $\,\mathsf{MSE}(0) = \bbeta^\T \Se \bbeta$. Hence, it makes sense to care more about the relative error $\,\mathsf{MSE}(\widetilde{\bbeta})/\mathsf{MSE}(0)$. In this case, the joint effective dimension and the signal-to-noise ratio control the upper bound on the relative error. We will get back to the analysis of the relative error and the joint effective dimension after we state an analogous result for the random design case.

\subsection{Random design} \label{S:RandomDesign}

In addition to the noise assumption, to study the performance of the NCT estimator in the random design setting we need to impose a couple more conditions on the distribution of the covariates.

\begin{assumption}[Sub-Gaussian covariates] \label{A:X}
	The scaled generic random vector of covariates $\,\St^{-1/2} \bx\,$ is sub-Gaussian.
\end{assumption}

\begin{assumption}[Convex decay of eigenvalues] \label{A:cvxdecay}
	There exists a convex decreasing function $\,\lambda(\cdot)\,$ such that the eigenvalues of $\,\St\,$ satisfy $\,\lambda_j = \lambda(j)\,$ for $\,j \in [d]$.
\end{assumption}
\noindent The previous assumption is technical and we impose it in our main result just for concreteness. Later in Remark~\ref{R:cvx_decay} we mention how our result can be modified if this assumption does not hold.

One more assumption is needed just to make the rates more friendly-looking. If it is not satisfied, our result below will be meaningless, so there is no loss of generality in this condition.
\begin{assumption}[Technical conditions] \label{A:tech}
	The effective rank satisfies $\,\reff[\St] \leq n$.
	Also, whenever we say ``with probability $\,1-\delta$...'', we suppose that the quantity
	\begin{align}
		\epsilon  \eqdef \sqrt{\frac{\log(d/\delta)}{n}}
		\nonumber
	\end{align}	
	satisfies $\,\epsilon \leq c\,$ for properly chosen implicit absolute constant $\,c > 0\,$ (this constant comes from the proof).
\end{assumption}
\noindent In addition to the assumptions above, in the sequel we take the convention $\,\lambda_k = 0\,$ for all $\,k > d$.
Now we are ready to present the following result.
\begin{theorem} \label{Th2}
	Suppose Assumptions~\eqref{A:Noise} -- \eqref{A:tech} are fulfilled. Recall $\,\rho\,$ from \eqref{def:rho}, $\,\epsilon  \eqdef \sqrt{\log(d/\delta)/n}$, and define $\,k^* \eqdef (\epsilon \log(1/\epsilon))^{-2/3}\,$ (essentially, $\,k^* \asymp n^{1/3}\,$ up to a logarithmic term).
	Then:
		\begin{enumerate}[(i)]
		\item With probability $\,1-\delta$, the NCT estimator $\,\ebeta\,$ from \eqref{estimator} with thresholding at level
		 \begin{equation}
	\begin{aligned}
		\tau = \sigma\rho
		\nonumber
		\end{aligned}
		\end{equation}
		 satisfies
\begin{equation}
	\begin{aligned}
		\mathsf{PE}(\ebeta) &\lesssim \inf\limits_{q \in [0, 2] } \left\{ \| \bLambda_{\leq k} \bU_{\leq k}^\T \bbeta\|_q^q \;(\sigma\rho)^{2-q}\right\} + \\
		&\qquad +\|\St\|\|\bbeta\|_2^2\left(
		\frac{\lambda_{k}}{\lambda_1} + \sqrt{\frac{\reff[\St] + \log(1/\delta)}{n}} +\epsilon\sum\limits_{j=1}^{k} \frac{\lambda_j(1+\epsilon\,j^2)}{\lambda_1} \right)
	\nonumber
	\end{aligned}
\end{equation}
for all $1 \leq  k \leq k^*$.
		\item  With probability $\,1-\delta$, the NCT estimator $\,\ebeta\,$ from \eqref{estimator} with thresholding at level
		 \begin{equation}
	\begin{aligned}
		\overline{\tau} = \sigma\rho + C \|\St\|^{1/2} \|\bbeta\|_2 \,\epsilon^{1/2} \max\limits_{j\in[k]} \left( \frac{\lambda_j\,(1+\epsilon j^2)}{\lambda_1} \right)^{1/2}
		\nonumber
		\end{aligned}
		\end{equation}
		 for some  $C$ satisfies
\begin{equation}
	\begin{aligned}
		\mathsf{PE}(\ebeta) &\lesssim  \inf\limits_{q \in [0, 2] } \left\{ \| \bLambda_{\leq k} \bU_{\leq k}^\T \bbeta\|_q^q \;\overline{\tau}^{2-q}\right\} + \|\St\|\|\bbeta\|_2^2\left( \frac{\lambda_{k}}{\lambda_1}  +
		\sqrt{\frac{\reff[\St] + \log(1/\delta)}{n}} \right)
	\nonumber
	\end{aligned}
\end{equation}
for all $1 \leq  k \leq k^*$.
\end{enumerate}
\end{theorem}
\noindent Several comments are in order.
\begin{remark}[Why two bounds?] \label{R:decay}
	We present two separate bounds in (i) and (ii) for distinct thresholds $\,\tau\,$ and $\,\overline{\tau}$, because they behave differently for various eigenvalue regimes. The bound from (i) outperforms the bound from (ii) in a wide variety of settings (e.g. in polynomial and superpolynomial decay scenario), however there are cases when the bound (ii) can be better (e.g. specific cases in factor model regime).
\end{remark}

\begin{remark}[Meaning of terms]
We call the term $\,\inf_{q \in [0, 2] } \left\{ \ldots \right\}\,$ in the bounds of Theorem~\ref{Th2} (i) and (ii) the ``main term'',
since it is almost the same as what we had in Theorem~\ref{Th1} for the fixed design. The other terms in these bounds are referred to as ``additional terms'' as they appear only in the random design case. Allowing $\,1\leq k \leq k^*\,$ provides a tradeoff, as some of the terms increase with growing $\,k$, while others decrease. In what follows we are typically interested in $\,k = k^*\,$ just for concreteness. The meaning of different parts of the ``additional terms'' is the following. The parts including $\,\reff[\St]\,$ are the payment for the covariance matrix estimation. The part $\,\lambda_{k^*}/\lambda_1\,$ appears due to the difficulty of control of the empirical eigenvectors beyond $k^*$-th. The parts with $\,\sum_{j=1}^{k^*} \lambda_j\,(1+\epsilon j^2)/\lambda_1\,$ and $\,\max_{j\in[k^*]} \lambda_j\,(1+\epsilon j^2)/\lambda_1\,$ are the payment for the control of the sample eigenvalues and eigenvectors up to index $\,k^*$.
\end{remark}
\begin{remark}[Moderate noise: simplifications and sufficient conditions for convergence] \label{R:sufficient}
Consider the moderate noise situation  $\,\sigma^2 \lesssim \|\St\|\|\bbeta\|_2^2$. In this case the ``additional terms'' become dominating: simply taking $\,k=k^*\,$ and $\,q=1$, applying $\,\| \bLambda_{\leq k} \bU_{\leq k}^\T \bbeta\|_1 \leq \|\St\|^{1/2} \|\bbeta\|_2\reff[\St]^{1/2} \,$ and plugging in $\,\tau\,$ and $\,\overline{\tau}\,$ makes the ``main term'' negligible.
Omitting logarithmic terms, the bound (i) reduces to
\begin{equation}
	\begin{aligned} \label{bound3}
		\mathsf{PE}(\ebeta) &\lesssim \|\St\|\|\bbeta\|_2^2\left(
		\frac{\lambda_{k^*}}{\lambda_1} + \sqrt{\frac{\reff[\St]}{n}} +\epsilon\sum\limits_{j=1}^{k^*} \frac{\lambda_j(1+\epsilon\,j^2)}{\lambda_1} \right),
	\end{aligned}
\end{equation}
while the bound (ii) reduces to
\begin{equation}
	\begin{aligned}
		\mathsf{PE}(\ebeta) &\lesssim \|\St\|\|\bbeta\|_2^2\left( \frac{\lambda_{k^*}}{\lambda_1}  +
		\left(\frac{\reff[\St]^2}{n}\right)^{1/4}\max\limits_{j\in[k^*]} \left( \frac{\lambda_j\,(1+j^2/\sqrt{n})}{\lambda_1} \right)^{1/2}\right)
	\nonumber
	\end{aligned}
\end{equation}
with high probability.

 From here we can easily deduce sufficient conditions to ensure the convergence of the absolute error $\,\mathsf{PE}(\ebeta) = \| \St\|\|\bbeta\|_2^2 \cdot o(1)\,$ as $\,n\to \infty\,$ without any conditions on $\,\bbeta$. In particular, $\,\lambda_n = o(1)$, $\,\reff[\St] = o(n)$, and $\,\sum_{j=1}^{k^*} \frac{\lambda_j\,(1+j^2/\sqrt{n})}{\lambda_1} = o(n^{1/2})\,$
 is enough (again, up to logarithmic factors). Essentially, these sufficient conditions require the decay of eigenvalues to be fast enough (but in a more sophisticated fashion than for MSE).
\end{remark}

\begin{remark}[Moderate noise: further simplifications in specific examples] \label{R:specific}
Continuing the moderate noise situation, for the sake of exposition, let us consider a couple of specific examples of eigenvalue regimes and illustrate how the bound from Theorem~\ref{Th2} (i) simplifies. As above, we omit logarithmic terms.
\begin{itemize}
\item Polynomial decay. If $\,\lambda_j \lesssim j^{-a}\,$ with $\,a \geq 1\,$ or $\,d \lesssim n^{(3-2a)/(3-3a)}\,$ with $\,a < 1$, it is easy to verify that the bracket factor of \eqref{bound3} is dominated by
    $\,\lambda_{k^*}/ \lambda_1 + n^{-1/2}\,$ (again ignoring logarithmic terms), and with high probability
    $$
		\mathsf{PE}(\ebeta) \lesssim \frac{\|\St\|\|\bbeta\|_2^2}{n^{\min(a/3,\, 1/2)}}.
	$$
In particular, when $\,a = 1\,$ (the boundary that has a good control of $\,\reff[\St]\,$ in high dimensions), we have with high probability
\begin{equation}
	\begin{aligned}
		\mathsf{PE}(\ebeta) &\lesssim \frac{\|\St\|\|\bbeta\|_2^2}{n^{1/3}}.
	\nonumber
	\end{aligned}
\end{equation}
When $\,a \geq 3/2$, we have  with high probability
	\begin{equation}
	\begin{aligned}
		\mathsf{PE}(\ebeta) &\lesssim \frac{\|\St\|\|\bbeta\|_2^2}{n^{1/2}}.
	\nonumber
	\end{aligned}
\end{equation}
\item Factor model regime. If $\,\lambda_1 \asymp \ldots \asymp \lambda_m \asymp d$, $\,\lambda_{m+1} \asymp \ldots \asymp \lambda_d \asymp 1\,$ for some $\,m \lesssim k^*$, then taking $\,k = m+1\,$ yields with high probability
	\begin{equation}
	\begin{aligned}
		\mathsf{PE}(\ebeta) &\lesssim \|\St\|\|\bbeta\|_2^2
		\left( \frac{1}{d} + \frac{m}{\sqrt{n}} + \frac{m^3}{n} \right) .
	\nonumber
	\end{aligned}
\end{equation}
\end{itemize}
\end{remark}

\begin{remark}[Large noise]
In the large noise case, when $\,\sigma^2 \gg \|\St\|\|\bbeta\|_2^2$, the ``main term'' dominates. Similarly to the fixed design case, we can factorize the ``main term'' into the error of the trivial estimator $\,\mathsf{PE}(0) = \bbeta^\T\St\bbeta$, the joint effective dimension and the signal-to-noise ratio: the ``main term'' from (i) becomes
\begin{align}
\mathsf{PE}(0)\,\inf\limits_{q \in [0, 2] } \Bigg\{  \eff_{q, k^*}(\St,\bbeta)\, \left[\snr^{-2} \, \frac{\left(\log (2d/\delta)\right)^{2/\alpha}}{n}\right]^{1-q/2} \Bigg\},
\nonumber
\end{align}
and the ``main term'' from (ii) can be rewritten as
\begin{align}
&\mathsf{PE}(0)\,\inf\limits_{q \in [0, 2] } \Bigg\{ \eff_{q,k^*}(\St,\bbeta)\, \Bigg[ \snr^{-2} \, \frac{\left(\log (2d/\delta)\right)^{2/\alpha}}{n} +\nonumber
		\\&\hspace{5cm}+ \frac{\|\St\|\|\bbeta\|_2^2}{\bbeta^\T\St\bbeta}\, \epsilon \max\limits_{j\in[k^*]} \left( \frac{\lambda_j\,(1+\epsilon j^2)}{\lambda_1} \right) \Bigg]^{1-q/2}\Bigg\}.
\nonumber
\end{align}
In this regime, the relative error $\,\mathsf{PE}(\ebeta)/\mathsf{PE}(0)\,$ is essentially controlled by the joint effective dimension and the signal-to-noise ratio.
More detailed analysis of the joint effective dimension $\,\eff_{q, k}(\St,\bbeta)\,$ is conducted in the next section.
\end{remark}

\begin{remark}[Comparison with the least squares]
It is straightforward to notice that the faster decay of eigenvalues, the better bound we obtain. This contrasts the min-norm least squares estimator considered in \cite{Bartlett, Chinot}, where the decay is required to be not too fast.  It also reveals the benefits of the thresholding even in such a situation.
\end{remark}

\begin{remark}[Dependence on the dimension $\,d$]
One can see that while the bound of Theorem~\ref{Th1} contains $\,\log(r)\,$ term with $\,r = \min(d,n)\,$ and gives meaningful result even in ultra-high or infinite dimension, the results of Theorem~\ref{Th2} contain the factor $\,\log(d)\,$ directly.
This logarithmic factor appears through the definition of $\,\epsilon$, which is an upper bound in Lemma~\ref{L:psi_n} of Appendix~\ref{S:mainproofs}: with probability $\,1-\delta$
\[
	\max\limits_{l,l^\prime\in[d]} \left| ( \mathbf{\Sigma}^{-1/2} \mathbf{\widehat{\Sigma}} \mathbf{\Sigma}^{-1/2}- \mathbb{I}_d)_{l,l^\prime} \right| \leq \epsilon.
\]
This condition is crucial for application of the relative perturbation bounds from \cite{Wahl}, which are the foundation of our proof technique for Theorem~\ref{Th2}. If we could choose $\,\epsilon\,$ in a completely dimension-free way to satisfy the above, it would yield totally dimension-free bounds in Theorem~\ref{Th2}, but currently applying the union bound inevitably brings $\,\log(d)\,$ factor. 
\end{remark}

\begin{remark}[Relaxing Assumption~\ref{A:cvxdecay}] \label{R:cvx_decay}
	Assumption~\ref{A:cvxdecay} can be avoided. If one defines $k^*$ as
		\begin{equation}
	\begin{aligned}
		k^* = \min\left( (\epsilon\log(1/\epsilon))^{-2/3},\, \max\left\{ j \in [d] \,\Bigg|\, \sum\limits_{\substack{l=1\\l\neq j}}^d \frac{\lambda_l}{|\lambda_j-\lambda_l|} +\frac{\lambda_j}{\min(\lambda_{j-1}-\lambda_j, \lambda_j - \lambda_{j+1})} \leq \frac{1}{3\epsilon}\right\} \right),
	\nonumber
	\end{aligned}
\end{equation}
then the same conclusion as in Theorem~\ref{Th2} is true with a slightly different rate.
\end{remark}

To compare the natural canonical thresholding with the canonical truncation of higher index components, i.e. PCR, we state the next proposition.
\begin{proposition}
	Assume the conditions of Theorem~\ref{Th2} hold and let $\,k^*\,$ be defined in the same way. Then, with probability $\,1-\delta$, the PCR estimator $\,\widetilde{\bbeta}^{PCR}\,$ with the number of leading principal components set to $\,m \lesssim k^*\,$ satisfies
	\begin{align}
		\mathsf{PE}(\widetilde{\bbeta}^{PCR}) \lesssim \|\St\|\|\bbeta\|_2^2 \left( \frac{\lambda_m}{\lambda_1} + \sqrt{\frac{\reff[\St] + \log(1/\delta)}{n}} \right) + \frac{\sigma^2 m}{n}\left( \log (2m/\delta)\right)^{1/\alpha}.
		\nonumber
	\end{align}
\end{proposition}
\noindent We omit the proof of this result, since it essentially uses the same techniques and follows the same strategy as the proof of Theorem~\ref{Th2}.
Note that in the moderate noise scenario the rate essentially coincides with what we obtained for the NCT estimator estimator in Remark~\ref{R:sufficient}. The adaptivity of our estimator $\,\ebeta\,$ comes into play in large noise case: the ``main term'' in the bounds on $\,\mathsf{PE}(\ebeta)\,$ is better than $\,\sigma^2m/n\,$ in situations when $\,\bU_{\leq k}^\T \bbeta\,$ is approximately sparse.

\subsection{Relative errors and joint effective dimension} \label{S:RelErrors}
So far we were able to establish some sufficient conditions for convergence of absolute errors $\,\mathsf{MSE}(\ebeta)\,$ and $\,\mathsf{PE}(\ebeta)\,$ of the NCT estimator without any assumptions on $\,\bbeta\,$ by simply taking $\,q=1\,$ (bound~\eqref{effective_rank_bound} and Remark~\ref{R:sufficient}). The analysis of the relative errors for fixed design and (in large noise case) random design requires more careful study of $\,\eff_{q,k}(\St,\bbeta)$. Let us  motivate why relative errors $\,\mathsf{MSE}(\ebeta)/\mathsf{MSE}(0)\,$ and $\,\mathsf{PE}(\ebeta)/\mathsf{PE}(0)\,$ might be of interest in the first place.
\subsubsection{Motivation for relative errors} \label{S:RelMotivation}
One reason behind studying the relative errors was already mentioned previously. Note that if there is no relation between $\,\Se\,$ and $\,\bbeta$, meaning that $\,\eU^\T \bbeta\,$ is a ``random'' vector, then we can expect  $\,\| \eU^\T \bbeta\|_{\infty} \asymp \| \bbeta\|_2 \sqrt{\log(d)/d}$. In this case, the trivial estimator $\,\widetilde{\bbeta} = 0\,$ achieves error
 \begin{equation}
	\begin{aligned}
		\mathsf{MSE}(0) = \bbeta^\T \Se \bbeta \leq \Tr[\Se] \| \eU^\T \bbeta\|_{\infty}^2 \asymp \|\Se\|  \| \bbeta\|_2^2\, \frac{\reff[\Se]\log(d)}{d}.
	\nonumber
	\end{aligned}
\end{equation}
As long as the eigenvalues of $\,\Se\,$ decay fast, even the trivial estimator gives error close to zero in high dimensions. Here we should highlight that this effect does not appear in low dimensions (and even in high-dimensional but isotropic situations), where the absolute and relative errors are just a multiplicative constant apart. (Same reasoning works for the PE.)
Hence, it is not satisfactory for us to show that the absolute error of our estimator goes to zero with growing sample size and dimension. We would like to get more meaningful conclusions from our results, which would confirm that the proposed estimator does better than the trivial estimator. This naturally leads to the relative errors.

Another motivation comes from the way statisticians evaluate and compare estimators in practical applications. In particular, a widely used performance measure is the \textit{coefficient of determination}, or simply $\,R^2$. For instance, the in-sample version for an estimator $\,\widetilde{\bbeta}\,$ is defined as
\begin{equation}
	\begin{aligned}
		R^2_{in}(\widetilde{\bbeta}) \eqdef 1 - \frac{\sum_{i=1}^n (y_i - \bx_i^\T\widetilde{\bbeta})^2}{\sum_{i=1}^n y_i^2}.
	\nonumber
	\end{aligned}
\end{equation}
The larger this quantity is, the better method we have; its largest possible value is $\,1$, and the value of $\,0\,$ indicates that the estimator does no better than the trivial estimator.
Maximization of $\,R^2_{in}(\widetilde{\bbeta})\,$ would try to fit the observed data perfectly, and in this sense it is not equivalent to minimizing $\,\mathsf{MSE}(\widetilde{\bbeta})/\mathsf{MSE}(0)$. Nevertheless, a crucial observation is that $\,R^2_{in}(\widetilde{\bbeta})\,$ is a relative quantity, which takes into account the performance of the trivial estimator. This supports our choice of the relative MSE as an error measure.

Similar intuition applies to the out-of-sample $\,R_{out}^2\,$ and the relative prediction error $\,\mathsf{PE}(\widetilde{\bbeta})/\mathsf{PE}(0)$, and intuitively it seems that they are linked even stronger.
Note that in applications it is often the case that even small but positive $\,R^2\,$ (e.g. $0.05$) can be considered a success. Therefore, the hope to have $\,\mathsf{MSE}(\widetilde{\bbeta})/\mathsf{MSE}(0)\,$ or $\,\mathsf{PE}(\widetilde{\bbeta})/\mathsf{PE}(0)\,$ converging to $\,0\,$ might be too optimistic in some situations. Having these relative errors smaller than $\,1\,$ already means that the procedure is able to extract some useful signal from the data.

\subsubsection{Why joint conditions on design and regression coefficients?} \label{S:JED1}
Prior to describing the properties of $\,\eff_{q,k}(\St, \bbeta)$, let us show why imposing conditions on the design alone, or imposing conditions on $\,\bbeta\,$ alone can be not enough to establish convergence of the relative errors. It is easier to do for the fixed design case, so let us focus on this setting for now.

For a given design matrix $\,\XX\,$ and an estimator $\,\widetilde\bbeta\,$ we can construct another estimator $\,\widetilde\btheta = \eL\eU^\T\widetilde\bbeta$.  Therefore,
\begin{align}
	\frac{(\widetilde{\bbeta} - \bbeta)^\T \Se (\widetilde{\bbeta} - \bbeta)}{\bbeta^\T \Se \bbeta} = \frac{\| \widetilde{\btheta} - \btheta\|_2^2}{\|\btheta\|_2^2}.
	\nonumber
\end{align}
The relative MSE in canonical domain has nothing to do with $\,\Se$. This demonstrates that getting a good rate is hopeless in high dimension assuming only fast decay of eigenvalues of $\,\Se$.

On the other hand, we might impose strong conditions on $\,\bbeta$, such as sparsity, in which case one could expect even $\,\mathsf{MSE}(\widetilde{\bbeta}) \asymp \sigma^2 s/n\,$ (up to a logarithmic factor) for some appropriate estimator $\,\widetilde{\bbeta}$, where $\,s\,$ measures the degree of sparsity. However, as we mentioned previously, if $\,\Se\,$ and $\,\bbeta\,$ are not related and the eigenvalues of $\,\Se\,$ decay fast, we might have $\,\mathsf{MSE}(0) \asymp \|\Se\|\|\bbeta\|_2^2/d\,$ (again up to logarithmic factors) for the trivial estimator. This implies that there is no much hope in getting vanishing relative error $\,\mathsf{MSE}(\widetilde{\bbeta})/\mathsf{MSE}(0)\,$ in high dimensions.
That is why it seems natural that our bound on the relative error depends on the joint effective dimension, that takes into account not only decay of eigenvalues or only assumptions on $\,\bbeta$, but the joint structure of $\,\Se\,$ and $\,\bbeta$.

\subsubsection{Joint effective dimension} \label{S:JED2}
Now, once we supported the appearance of the joint effective dimension, let us mention its basic properties. (For concreteness we choose the random design setting and consider $\,\eff_{q,k}(\St,\bbeta)$, though the following ideas apply to $\,\eff_{q,r}(\Se,\bbeta)\,$ appearing in the fixed design case.)
\begin{itemize}
	\item $\eff_{q,k}(\St, \bbeta) \leq k$.
	\item $\eff_{q,k}(\St, \bbeta)\,$ is decreasing in $\,q\,$ and increasing in $\,k$.
	\item $\eff_{2, k}(\St, \bbeta) \leq 1$, $\,\eff_{2, d}(\St, \bbeta) = 1$.
	\item $\eff_{0,k}(\St, \bbeta) = \| \bLambda_{\leq k}\bU_{\leq k}^\T\bbeta \|_0\,$ is essentially the sparsity of $\,\bU_{\leq k}^\T\bbeta$.
\end{itemize}
For now let us assume $\,\snr \geq c > 0\,$ and focus on $\,\eff_{q,k}(\St,\bbeta)\,$ only. Recall that the ``main term'' in the relative bounds looks like
\begin{equation}
	\begin{aligned}
		\inf\limits_{q \in [0, 2] } \Bigg\{  \eff_{q,k}(\St,\bbeta)\,\zeta_n^{1-q/2} \Bigg\}	
	\nonumber
	\end{aligned}
\end{equation}
where $\,\zeta_n\,$ is some vanishing rate like $\,1/n$.
Hence, the properties above reveal a tradeoff in the main term:
\begin{itemize}
	\item When $\,q\,$ is large, i.e. closer to $2$, it is easier to control $\,\eff_{q,k}(\St,\bbeta)$; however, the vanishing rate $\,\zeta_n\,$ is raised to a small power, making the convergence slow.
	\item When $\,q\,$ is small, i.e. closer to $0$, it is more difficult to control $\,\eff_{q,k}(\St,\bbeta)$; in contrast, $\,\zeta_n\,$ is raised to a large power potentially enabling fast convergence rate.
\end{itemize}
So, the bound allows to find largest $\,q\,$ for which $\,\eff_{q,k}(\St,\bbeta)\,$ can be bounded in dimension-free and sample size-free manner (or at least the dependence on $\,d\,$ and $\,n\,$ should not be that severe) to facilitate faster convergence rate.

Some scenarios where $\,\eff_{q,k}(\St,\bbeta)\,$ can be bounded more explicitly (for some $q<2$) are discussed below:
\begin{itemize}
	\item  Sparsity of $\,\bU_{\leq k}^\T\bbeta$. Denote $\,s \eqdef \| \bU_{\leq k}^\T \bbeta\|_0\,$ to be the sparsity level. Then, as already mentioned previously ,$\,\eff_{0,k}(\St,\bbeta) = s$.
	\item Approximate sparsity of $\,\bLambda_{\leq k}\bU_{\leq k}^\T\bbeta$. Suppose there exists a small set $\,\mathcal{J} \subseteq [k]\,$ (of size $\,|\mathcal{J}| = s$) of significant components, so that the rest of the components satisfy
	\begin{align}
		\lambda_j^{1/2} |\bu_j^\T \bbeta| \,\lesssim\, \frac{1}{d} \| \bLambda \bU^\T \bbeta\|_2\;\;\;\;\;\text{ for all } j\in [k] \setminus\mathcal{J}.
	\nonumber
	\end{align}
	Then $\,\eff_{1,k}(\St,\bbeta) \lesssim s$.
	\item Polynomial decay. Let $\,\lambda_j/\lambda_1 \asymp j^{-a}$, $\,|\bu_j^\T \bbeta|/|\bu_1^\T \bbeta| \asymp j^{-b}$, where $\,a \geq 0$. We have several cases:
	\begin{itemize}
		\item If $\,a+2b \geq 1\,$ and $\,q \geq 2/(a+2b)$, then $\,\eff_{q,k}(\St,\bbeta) \asymp 1$.
		\item If $\,a+2b \geq 1\,$ and $\,q  < 2/(a+2b)$, then $\,\eff_{q,k}(\St,\bbeta) \asymp k^{1-(a+2b)q/2}$.
		\item If $\,a+2b < 1$, then $\,\eff_{q,k}(\St,\bbeta) \asymp k^{1-(a+2b)q/2}/d^{(1-a-2b)q/2}$.
	\end{itemize}
	Here we omitted logarithmic factors.
	To better understand the dependence of $\,\eff_{q,d}(\St,\bbeta)\,$ on $\,d\,$ in specific case $\,k=d$, in Figure~\ref{F:effdim} we depict this dependence in $\,(a+2b)$ -- $q\,$ axes. In the green region $\,\eff_{q,d}(\St,\bbeta)\,$ does not grow with $\,d$, while in the yellow region $\,\eff_{q,d}(\St,\bbeta)\,$ grows with $\,d\,$ polynomially, and the contours of constant power are illustrated with different colors.
\end{itemize}

\begin{figure}
         \centering\includegraphics[width=0.7\textwidth]{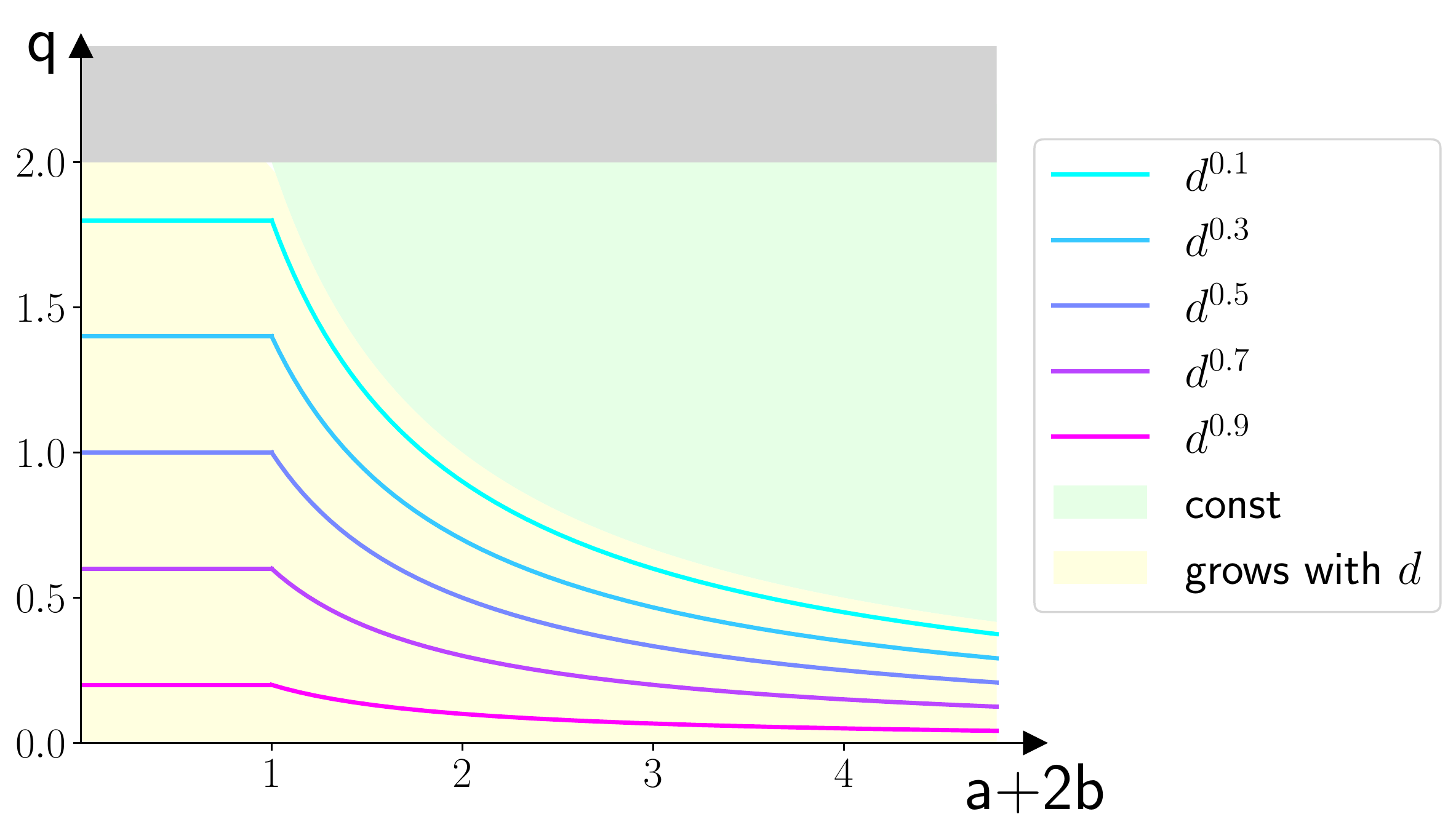}
        \caption{Dependence of $\,\eff_{q,d}(\St,\bbeta)\,$ on $\,d$.}
        \label{F:effdim}
\end{figure}

\subsubsection{Bounds on relative errors in polynomial decay scenario} \label{S:PolyDecay}
Since the ``main term'' conveniently decomposes into several factors, among which the most interesting one -- the joint effective dimension -- was discussed above, the analysis of the relative errors of the NCT estimator in the fixed design and (for large noise case) random design is pretty much complete. It is intriguing though, what happens to the bounds on PE in moderate noise case: recall that in this scenario the ``main term'' is absorbed by the ``additional term'', which does not have a structure allowing a direct analysis of the bounds on the relative error $\,\mathsf{PE}(\ebeta)/\mathsf{PE}(0)$. It is not clear whether they can be stated in a way that will provide better understanding of the relative error. Instead, we can take a look at the particular case of polynomial decay of eigenvalues and regression coefficients: $\,\lambda_j/\lambda_1 \asymp j^{-a}$, $\,|\bu_j^\T \bbeta|/|\bu_1^\T \bbeta| \asymp j^{-b}$. After tedious calculations, one may express the bounds from Theorem~\ref{Th1} and Theorem~\ref{Th2} (i), (ii) in terms of $\,d,n,a,b\,$ only. It turns out, that in this scenario the bound from Theorem~\ref{Th2} (ii) is always worse than the bound from Theorem~\ref{Th2} (i), so we exclude it from consideration. In Figure~\ref{F:rates} we plot the  contours of constant convergence rate on $\,a$ -- $b\,$ plane for the bounds from Theorem~\ref{Th1} and Theorem~\ref{Th2} (i). Different colors of the contours correspond to different rates. The background color describes the assumptions on $\,d\,$ and $\,n\,$ that we make in different regions: in light green zones $\,d\,$ can be much larger than $\,n\,$ (though this is not necessary), in light yellow zones $\,d\,$ is allowed to be at most of the same order as $\,n$, i.e. $\,d \lesssim n$, and in grey zone the rates do not go to zero unless $\,d\,$ is significantly smaller than $\,n$. We again disregard the logarithmic terms.

\begin{figure}
     \begin{subfigure}{\textwidth}
         \centering\includegraphics[width=0.95\textwidth]{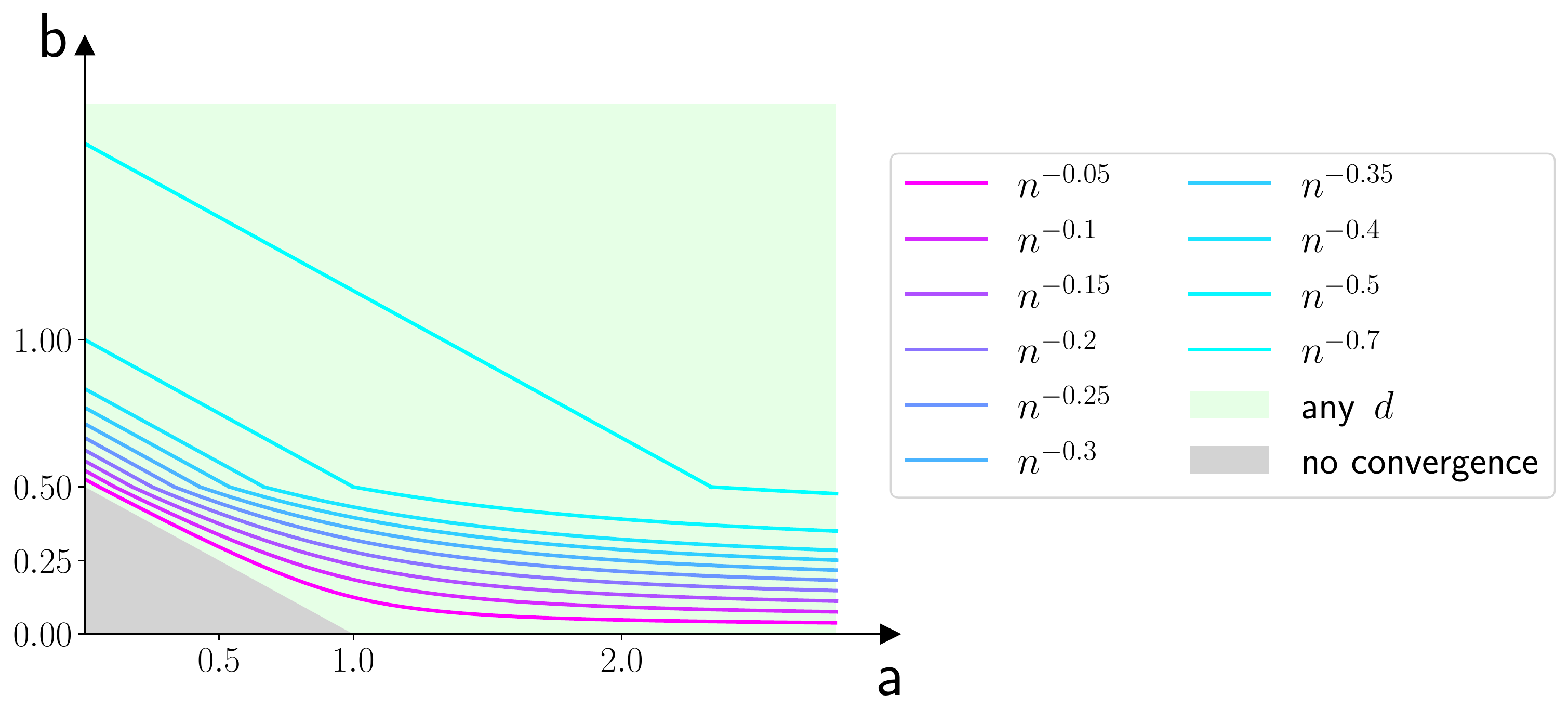}
         \vspace{-0.1cm}
         \caption{Bound on the relative mean squared error $\,\mathsf{MSE}(\ebeta)/\mathsf{MSE}(0)\,$ from Theorem~\ref{Th1};}
     \end{subfigure}
     \vfill
     \vspace{1.5cm}
     \begin{subfigure}{\textwidth}
         \centering\includegraphics[width=0.95\textwidth]{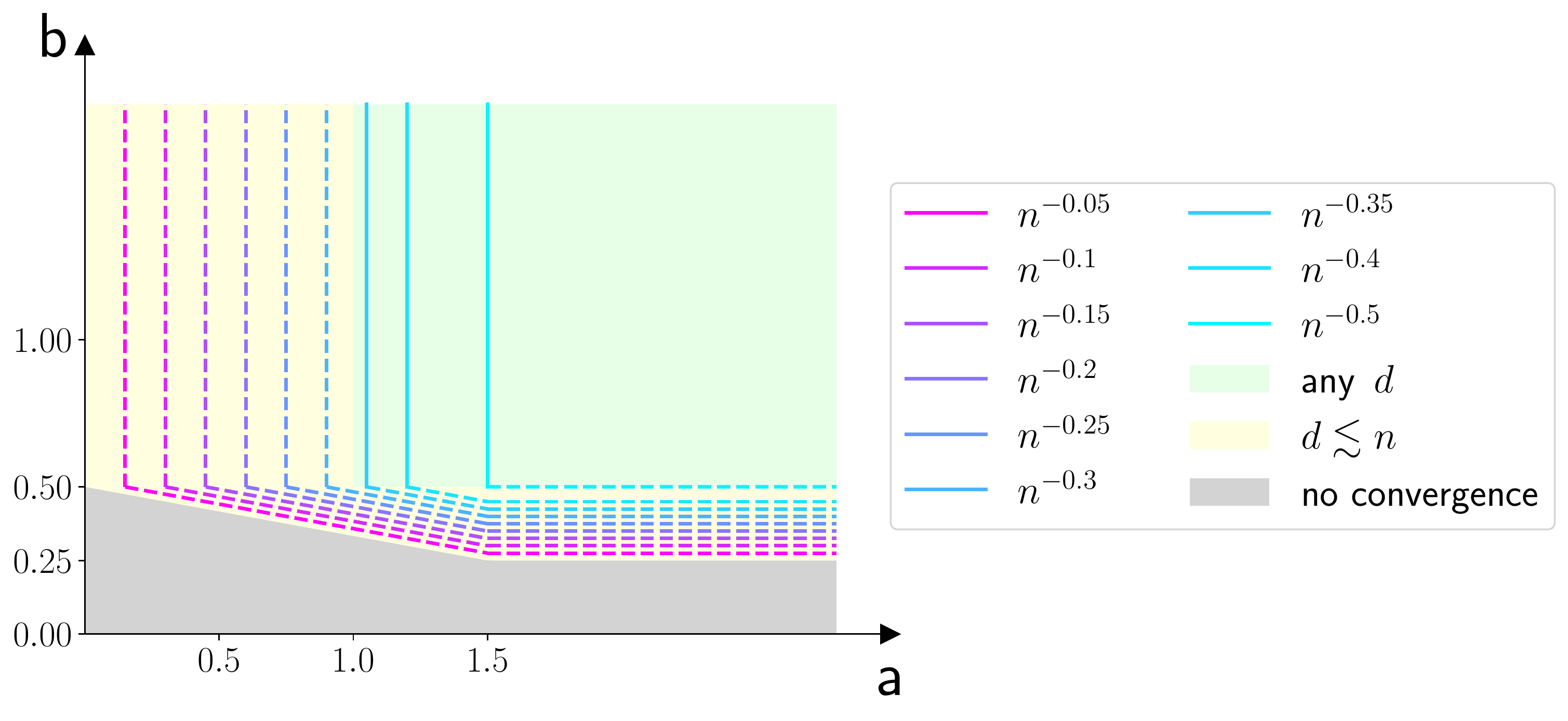}
         \vspace{-0.1cm}
         \caption{Bound on the relative prediction error $\,\mathsf{PE}(\ebeta)/\mathsf{PE}(0)\,$ from Theorem~\ref{Th2} (i);}
     \end{subfigure}
       \caption{Rates for the relative errors of the NCT estimator in polynomial decay scenario.}
        \label{F:rates}
\end{figure}

%
%
%

\section{(Near) minimax optimality of the NCT estimator} \label{S:minimax}
This section focuses on the fixed design setting.
In light of the above discussion, we introduce the following parameter classes for the fixed design linear regression problem (with Gaussian noise, for simplicity): for any design matrix $\,\XX \in \R^{n \times d}$
\begin{align}
	\mathcal{P}^{\XX}(q, \mathsf{D}, \mathsf{S}) \,\eqdef\,
	\left\{
		(\bbeta, \sigma) \in  \R^d \times \R_+:\;\;\;\eff_{q,r}(\Se, \bbeta) \leq \mathsf{D}, \;\snr \geq \mathsf{S}
	\right\}.
	\nonumber
\end{align}
Clearly, we have the following relations:
\begin{align}
	&\mathcal{P}^{\XX}(q, \mathsf{D}, \mathsf{S})   \subset \mathcal{P}^{\XX}(q^\prime, \mathsf{D}, \mathsf{S})  \;\;\;\text{ for }\;\;q < q^\prime, \nonumber\\
	&\mathcal{P}^{\XX}(q, \mathsf{D}, \mathsf{S})   \subset \mathcal{P}^{\XX}(q, \mathsf{D}^\prime, \mathsf{S})  \;\;\;\text{ for }\;\;\mathsf{D} < \mathsf{D}^\prime, \nonumber\\
	&\mathcal{P}^{\XX}(q, \mathsf{D}, \mathsf{S})  \subset \mathcal{P}^{\XX}(q, \mathsf{D}, \mathsf{S}^\prime)  \;\;\;\text{ for }\;\;\mathsf{S} > \mathsf{S}^\prime. \nonumber
	\nonumber
\end{align}
In a sense, this orders instances of problems by difficulty. Note that for fixed $\,d,n\,$ the family of classes is parameterized by three quantities: $\,q$, $\mathsf{D}$ and $\mathsf{S}$. While $\mathsf{S}$ is independent of the other two (since it is the only quantity related to the magnitude of noise), $\,q\,$ and $\,\mathsf{D}\,$ bring some ambiguity. More specifically, if a specific instance of a problem belongs to $\,\mathcal{P}^{\XX}(q, \mathsf{D}, \mathsf{S})$, it also belongs to $\,\mathcal{P}^{\XX}(q-\Delta q, \mathsf{D} \exp(\Delta \log\mathsf{D}), \mathsf{S})\,$ for some perturbations $\,\Delta q$ and $\,\Delta \log\mathsf{D}$. To reduce this indeterminacy, we restrict $\mathsf{D}$ to be at most of logarithmic order when $q > 0$, i.e. $\,\mathsf{D} \lesssim \log(r)$. Hence, just two quantities essentially control a ``complexity'' of an instance of the linear regression problem: smallest $\,q\,$ for which $\eff_{q,r}(\Se, \bbeta)  \asymp \log(r)$ and the signal-to-noise ratio.

It is reasonable to ask whether our NCT estimator is minimax optimal over these classes with respect to the relative MSE. The following theorem answers this question.
\begin{theorem} \label{Th:minimax}
	(i) Let $\,q \in (0; 2)\,$ and  $\,\mathsf{D} \lesssim \log(r)\,$ be large enough. Then for any design matrix $\,\XX\in\R^{n\times d}$
	\begin{align}
		\inf\limits_{\widetilde{\bbeta}} \sup\limits_{\mathcal{P}^{\XX}(q, \mathsf{D}, \mathsf{S}) } \;\E\left[ \frac{\mathsf{MSE}(\widetilde{\bbeta})}{\mathsf{MSE}(0)}\right] \;\gtrsim\; \frac{1}{(\mathsf{S}^2\, n)^{1-q/2}}\,.
		\nonumber
	\end{align}
	(ii) Let $\,q = 0$ (in this case $\,\mathsf{D}\,$ plays role of the sparsity of the principal components). Then for any design matrix $\,\XX\in\R^{n\times d}$
	\begin{align}
		\inf\limits_{\widetilde{\bbeta}} \sup\limits_{\mathcal{P}^{\XX}(q, \mathsf{D}, \mathsf{S})} \;\E\left[ \frac{\mathsf{MSE}(\widetilde{\bbeta})}{\mathsf{MSE}(0)}\right] \;\gtrsim\; \frac{\mathsf{D} \log(er/{\mathsf{D}})}{\mathsf{S}^2\, n}\,.
		\nonumber
	\end{align}
\end{theorem}
\noindent This result almost directly follows from classical minimax lower bounds for Gaussian sequence model, with some slight adjustments. The bounds align well with \cite{lqballs}, which studies minimax lower bounds over $\ell_q$-balls.
We emphasize that the lower bounds presented here hold for any design matrix $\,\XX$, unlike a lot of minimax optimality results that just find one ``difficult'' design to show the lower bound or impose restrictive assumptions on $\,\XX$. This supports our intuition that for the relative bounds design matrix does not play role as much as the interactions of the covariance and the regression coefficients.

One could notice that in (i) the rate does not completely match the upper bound of Theorem~\ref{Th1}, which also contains a factor $\,\mathsf{D}\,$ and a logarithmic factor. However, as we highlighted earlier, without much loss of generality we restrict $\,\mathsf{D}\,$ to be of at most of order $\,\log(r)$. Therefore, the discrepancy between the upper bound for NCT and the associated minimax lower bounds is just in the logarithmic factors. 

\section{Theoretical properties of the GCT estimator} \label{S:theoryGCT}

Now we move to a brief study of the theoretical guarantees for the GCT estimator~\eqref{estimator2}. The results are not as insightful as the ones for the NCT estimator, and we focus on the fixed design setting only. However, once we state the MSE bound, the theory for the random design can be developed in the same way as for the NCT estimator.

\begin{theorem} \label{Th1a}
	Suppose Assumption~\ref{A:Noise} is satisfied. Let $\,\varphi \geq 0$. Take $\,\tau = \widehat{\lambda}_1^{\varphi/2}\sigma \rho\,$ with $\,\rho = \frac{2}{\sqrt{n}}\left(\log(2r/\delta)\right)^{1/\alpha}$. Then, with probability $\,1-\delta$, the GCT estimator $\,\ebeta\,$ from~\eqref{estimator2} with parameter $\,\varphi\,$ and thresholding at level $\,\tau\,$ satisfies
	\begin{align}
		\mathsf{MSE}(\ebeta) &\lesssim \sum\limits_{j=1}^r \min\left( \frac{\widehat\lambda_1^{\varphi/2}}{\widehat\lambda_j^{\varphi/2}} \,\sigma \rho,\, |\theta_j|\right)^2.
	\nonumber
	\end{align}
Moreover, when $\,\mathsf{T}_\tau[\,\cdot\,] = \mathsf{SOFT}_\tau[\,\cdot\,]$, the matching lower bound takes place, i.e. the above upper bound is tight:
	\begin{align}
		\mathsf{MSE}(\ebeta) &\gtrsim \sum\limits_{j=1}^r \min\left( \frac{\widehat\lambda_1^{\varphi/2}}{\widehat\lambda_j^{\varphi/2}} \,\sigma \rho,\, |\theta_j|\right)^2
	\nonumber
	\end{align}
	with probability $\,1-\delta$.
\end{theorem}
\noindent
The obtained bound may be difficult to comprehend, but we state it in the most general form to make sure it is tight and applicable in wide range of scenarios. When $\,\theta_j = 0\,$ for $\,j \geq m$, it has no estimation errors beyond the first $\,m\,$ principal components, adapting very well to focusing only on the low dimensions estimation like PCR.  If, in addition, $\,\widehat{\lambda}_1 / \widehat{\lambda}_m\,$ is bounded, we have
$$
\mathsf{MSE}(\ebeta) \lesssim \sum\limits_{j=1}^m \min\left( \sigma \rho,\, |\theta_j|\right)^2,
$$
which is not much larger than the MSE of PCR.  It can even be much smaller than PCR when $\,|\theta_j|\,$ are small for many indices $\,j$.  In general, GCT outperforms NCT  when $\,|\theta_j|\,$ decays fast enough.

The next corollary allows to make sure that in general the bound is essentially dimension-free: as for Theorem~\ref{Th1}, the rate can be expressed in terms of the effective rank $\,\reff[\Se]$.

\begin{corollary} \label{Cor1}
	Under assumptions of Theorem~\ref{Th1a}, with probability $\,1-\delta$, the GCT estimator $\,\ebeta\,$ from~\eqref{estimator2} with parameter $\,\varphi \geq 0\,$ satisfies
	\begin{align}
		\mathsf{MSE}(\ebeta) &\lesssim \left( \|\Se\|\|\bbeta\|^2 + \sigma \|\Se\|^{1/2}\|\bbeta\| \,\reff[\Se]^{1/2}\right) \frac{(\log (2r/\delta))^{2/(2+\varphi)\alpha}}{n^{1/(2+\varphi)}}.
	\nonumber
	\end{align}
\end{corollary}
\noindent Note that when $\,\varphi = 0\,$ we essentially recover the rate obtained after Theorem~\ref{Th1}.
The rate deteriorates when $\,\varphi\,$ is far from $0$, and this is explainable: the GCT procedure significantly deviates from the natural one, leading to a worse bound in the worst case, i.e. when the only assumption is the control of effective rank, and spikiness of canonical coefficients is not justified.
\begin{remark}[Simplifications in specific cases]
In several specific cases the rate from Theorem~\ref{Th1a} can be made much more explicit. We omit logarithmic terms.
\begin{itemize}
\item Polynomial decay. If $\,\widehat\lambda_j/\widehat{\lambda_1} \asymp j^{-a}\,$ and $\,|\bu_j^\T\bbeta|/|\bu_1^\T\bbeta| \asymp j^{-b}\,$ for $\,a \geq 0,\, a+2b \geq 1$, then with high probability
	\begin{equation}
	\begin{aligned}
		\mathsf{MSE}(\ebeta) &\lesssim \left( \frac{\sigma^2}{n}\right)^{\frac{a+2b-1}{a(\varphi+1)+2b}}.
	\nonumber
	\end{aligned}
\end{equation}
\item Sparsity. If there exists a set $\,\mathcal{J}\,$ of size $\,|\mathcal{J}| = s\,$ such that $\,\bu_j^\T\bbeta = 0\,$ for $\,j \notin \mathcal{J}\,$ and $\,\widehat\lambda_{\max \mathcal{J}} \gtrsim \widehat\lambda_1$, then with high probability
	\begin{equation}
	\begin{aligned}
		\mathsf{MSE}(\ebeta) &\lesssim \frac{s\sigma^2}{n}.
	\nonumber
	\end{aligned}
\end{equation}
\item Approximate sparsity. If there exists a set $\,\mathcal{J}\,$ of size $\,|\mathcal{J}| = s\,$ such that $\,\widehat{\lambda}_j^{1/2} |\bu_j^\T\bbeta| \lesssim \|\btheta\|_2/d\,$ for $\,j \notin \mathcal{J}\,$ and $\,\widehat\lambda_{\max \mathcal{J}} \gtrsim \widehat\lambda_1$, then with high probability
	\begin{equation}
	\begin{aligned}
		\mathsf{MSE}(\ebeta) &\lesssim \frac{s\sigma^2}{n} + \frac{\|\btheta\|_2^2}{d}.
	\nonumber
	\end{aligned}
\end{equation}
\item Factor Model regime. If $\,\lambda_1 \asymp \ldots \asymp \lambda_m \asymp d$, $\,\lambda_{m+1} \asymp \ldots \asymp \lambda_d \asymp 1\,$ for some $\,m$, then with high probability
	\begin{equation}
	\begin{aligned}
		\mathsf{MSE}(\ebeta) &\lesssim \frac{m\sigma^2}{n} + \frac{\|\bbeta\|_2^2}{d}.
	\nonumber
	\end{aligned}
\end{equation}
\end{itemize}
Therefore, the rate from Theorem~\ref{Th1a} can adapt well to these specific structures despite the deteriorating rate of Corollary~\ref{Cor1}, which is only an upper bound.
\end{remark}

\section{Miscellaneous aspects} \label{S:tuning}
\subsection{Computational complexity for single $\tau$}
To start with, we focus on the case when a good value of $\tau$ is somehow known, and analyze the computational complexity of the GCT estimators. In particular, this includes the NCT estimator.
The computation of SVD of $\,\XX\,$ (specifically, $\,\eL\,$ and $\,\eU$) takes $\,O(dn\min(d,n))\,$ time.  Once we have SVD, computing the matrix $\,\bA \eqdef \eU \eL^{-1-\varphi}\,$ and the vector $\,\bbb \eqdef \eL^{-1+\varphi} \eU^\T \,\frac{\XX^\T \YY}{n}\,$ needed before the generalized thresholding takes $\,O(dn)\,$ time.
Obtaining $\,\ebeta\,$ for already computed $\,\bA\,$ and $\,\bbb\,$ takes $\,O(d\min(d,n))\,$ time. Therefore, the total computational time of our procedure is $\,O(dn\min(d,n))$.
The computation is as fast as the SVD of the design matrix.

Note that computational complexity of the LASSO is $\,O(n\min(d,n)^2)$, when we compute its solution path via a modification of Least Angle Regression, see \cite{LAR}.

\subsection{Efficient tuning of thresholding level $\,\tau$} \label{eff_tuning}

Our approach requires to tune the hyperparameter $\,\tau$.
Whatever $\,\tau\,$ is, we anyway have to compute $\,\bA\,$ and  $\,\bbb$. This already takes $\,O(dn\min(d,n))\,$ time.
Applying the generalized thresholding and combining the result into the vector $\,\ebeta\,$ takes $\,O(d\min(d,n))\,$ time.
This means that we can try $\,n\,$ different values of $\,\tau\,$ ``for free'' --- the computational complexity will be still of the same order as computing  $\,\ebeta\,$ for a single value of $\,\tau$. But we can go even further, if we focus on the GCT estimators with the soft or hard thresholding.

Notice that varying $\tau$ continuously from $0$ to $+\infty$, we still can get only $\,(\min(d,n)+1)\,$ different solutions $\,\ebeta\,$ for given $\,\XX$, $\,\YY\,$ because we threshold a vector of size $\,\min(d,n)$.
Therefore, we can compute the whole solution path for $\,\tau\,$ from $0$ to $+\infty$. However, we are not that interested in the solution path, since we do not expect to get coefficients entering the picture one by one as in LASSO for sparse regression. Instead, this can be useful for $L$-fold cross validation, where we will have at most $\,(L\min(d,n)+2)\,$ ``interesting'' values of $\,\tau\,$ giving different solutions $\,\ebeta$. In total, this implies that we need
\begin{equation}
	\begin{aligned}
		O(dn\min(d,n)) + (L\min(d,n)+2) \cdot O(d\min(d,n)) = O(dn\min(d,n) + Ld\min(d,n)^2)
	\nonumber
	\end{aligned}
\end{equation}
operations to find the best $\,\tau\,$ (providing smallest cross-validation error). In practice, we typically use $\,L\,$ of constant order, e.g.  $L=5\,$ or $\,L=10$, which leads to the total computational complexity of $\,O(dn\min(d,n))\,$ for our optimally tuned estimator -- same as for a single value of $\,\tau$. Leave-one-out cross validation takes slightly more computations, namely, $\,O(dn\min(d,n)^2)$.

Similar ``free tuning'' property holds for LASSO (we again refer to \cite{LAR}).
However, for example the ridge regression does not possess this nice properties: different values of regularization parameter will lead to different estimators, and one  has to ``guess'' a discrete set of values to be tried.

\subsection{Optimality of cross-validation} \label{optimality_of_CV}
More formally, let $\{ \mathcal{B}_l \}_{l=1}^L$ be the split of the data point indices $\,[n]\,$ into $L$ approximately equally-sized disjoint blocks, i.e.
\begin{align}
	\mathcal{B}_l \cap \mathcal{B}_{l^\prime} = \varnothing\;\;\;\text{ for all }\;l\neq l^\prime,\; l,l^\prime \in [L]\;\;\;\text{ and }\;\;\; \bigcup_{l=1}^L \mathcal{B}_l = [n]
	\nonumber
\end{align}
satisfying $\,\lfloor n/L \rfloor \leq |\mathcal{B}_l| \leq \lfloor n/L \rfloor + 1\,$ for all $\,l \in[L]$.
The $L$-fold cross-validation leads to the following choice of the hyperparameter $\tau$:
\begin{align}
	\tau^{cv} &\,\eqdef\, \arg\min\limits_{\tau \geq 0}\; \frac{1}{L} \sum\limits_{l=1}^L \frac{1}{|\mathcal{B}_l|} \sum\limits_{i \in \mathcal{B}_l} \left( y_i - \bx_i^\T \ebeta_\tau^{(l)} \right)^2
	\label{tau_cv} \\ &\,\;=\;\, \arg\min\limits_{\tau \in \mathcal{T}}\; \frac{1}{L} \sum\limits_{l=1}^L \frac{1}{|\mathcal{B}_l|} \sum\limits_{i \in \mathcal{B}_l} \left( y_i - \bx_i^\T \ebeta_\tau^{(l)} \right)^2,
	\nonumber
\end{align}
where $\,\ebeta_\tau^{(l)}\,$ is the NCT estimator with thresholding at level $\,\tau\,$ computed on the part of the sample $\,\{ (\bx_i, y_i) \}_{i\in[n]\setminus\mathcal{B}_l}$.
Note that here $\,\mathcal{T}\,$ is a set of $\,|\mathcal{T}| \leq Lr +2\,$ known ``interesting'' values giving all possible variety of estimators (i.e. $\,\{ \ebeta_\tau^{(l)}\}_{l\in[L],\,\tau\geq 0} = \{ \ebeta_\tau^{(l)}\}_{l\in[L],\,\tau\in\mathcal{T}}$),
as discussed in Section~\ref{eff_tuning}.
An oracle counterpart of $\,\tau^{cv}\,$, defining optimal value of the hyperparameter w.r.t. the expected cross-validation error, is given by 
\begin{align}
	\tau^{oracle} &\,\eqdef\, \arg\min\limits_{\tau \geq 0}\; \frac{1}{L} \sum\limits_{l = 1 }^L \E \left[ \left( y - \bx^\T \ebeta_\tau^{(l)} \right)^2\right]
\label{tau_oracle}		\\ &
	\,\;=\;\,
	\arg\min\limits_{\tau \in \mathcal{T}}\; \frac{1}{L} \sum\limits_{l=1}^L \mathsf{PE}(\ebeta_\tau^{(l)}).
	\nonumber
\end{align}
We have the following result stating that the choice of $\tau$ based on cross-validation performs as well as the oracle choice (in terms of the expected cross-validation error). .
\begin{theorem} \label{Th:CV}
Suppose Assumptions~\eqref{A:Noise} -- \eqref{A:tech} are fulfilled.
For technical simplicity, consider the truncated NCT estimator:
 \begin{equation}
	\begin{aligned}
		\ebeta_\tau \eqdef \eU_{\leq k^*} \eL_{\leq k^*}^{-1} \,\mathsf{SOFT}_{\tau}\left[ \eL_{\leq k^*}^{-1} \eU_{\leq k^*}^\T \,\frac{\XX^\T \YY}{n}\right]
	\nonumber
	\end{aligned}
\end{equation}
with $\,k^*\,$ from Theorem~\ref{Th2}.
Then, with probability $\,1-\delta$, the above estimator with thresholding at level $\,\tau^{cv}$ chosen by $L$-fold cross-validation satisfies
\begin{align}
	&\frac{1}{L} \sum\limits_{l=1}^L \mathsf{PE}(\ebeta_{\tau^{cv}}^{(l)})\lesssim \frac{1}{L} \sum\limits_{l=1}^L \mathsf{PE}(\ebeta_{\tau^{oracle}}^{(l)})
	\nonumber \\ &\qquad\qquad+ \left(\|\St\| \|\bbeta\|_2^2 + \sigma^2( \log(2r/\delta) )^{2/\alpha}\right) \left( \sqrt{\frac{L\log(Lr/\delta)}{n}} + \frac{L(\log(Lr/\delta))^{2/\alpha}}{n}\right).
	\nonumber
\end{align}
\end{theorem}


The result resembles, for instance, \cite{CV0, CV} (see Theorem 7.1 in the former and Theorem 1 in the latter) and many other works, and the logic behind it is quite standard. However, since our framework is not as general as in \cite{CV0} or \cite{CV}, we state the high probability bound rather than in expectation, and we avoid almost sure boundedness condition on the response and the possible predictions of our estimator (unlike the aforementioned literature). We also emphasize once again that $\,\tau^{cv}\,$ and $\,\tau^{oracle}\,$ from \eqref{tau_cv} and \eqref{tau_oracle} are minimizers across all $\,\tau \geq 0$, and the structure of our estimator allows to compute $\tau^{cv}$ defined in such a way in a reasonable time, which is quite unusual feature. 

From Theorem~\ref{Th2} and Theorem~\ref{Th:CV} it follows that $\,L^{-1} \sum_{l=1}^L \mathsf{PE}(\ebeta_{\tau^{cv}}^{(l)}) \lesssim \Diamond_\delta\,$ with probability $\,1-\delta$, where $\,\Diamond_\delta$ is the error bound from Theorem~\ref{Th2} (best of (i) and (ii)). This bound holds since $\,\tau^{oracle}\,$ is no worse than the choices of $\,\tau\,$ from Theorem~\ref{Th2} in terms of the expected cross-validation error, and because the extra term in Theorem~\ref{Th:CV} does not exceed the error bounds from Theorem~\ref{Th2} (treating $\,L$ as constant).

 It worth mentioning that the result of Theorem~\ref{Th:CV} is not really what one aims for. Ideally, we would like to obtain a high probability bound of the form $\,\mathsf{PE}(\ebeta_{\tau^{cv}}) \lesssim \mathsf{PE}(\ebeta_{\tau^{opt}}) + \Delta_\delta\,$ with small $\,\Delta_\delta$, where $\,\tau^{opt} \eqdef \arg\min_{\tau \geq 0}\; \mathsf{PE}(\ebeta_{\tau})\,$ is the optimal value of the hyperparameter, which may differ from $\,\tau^{oracle}$. This would consequently imply $\,\mathsf{PE}(\ebeta_{\tau^{cv}}) \lesssim \Diamond_\delta+ \Delta_\delta\,$ with high probability. Instead of analyzing the expected prediction error of the estimator trained on the whole sample, Theorem~\ref{Th:CV} is concerned with the expected cross-validation error. However, as we already mentioned, even classical works on cross-validation, such as \cite{CV0} and \cite{CV}, also state results of this flavor, and many papers focusing specifically on cross-validation actually work with the expected cross-validation error, which anyway is believed to be a good proxy for the expected prediction error. With this in mind, we hope our result gives a convincing confirmation that the cross-validation procedure applied to our estimator is reasonable, even though in terms of the expected cross-validation error.
Since the main focus of this paper is not on the cross-validation, we do not go beyond this.

To conclude the discussion on the cross-validation, we mention that in principle one may want to tune $\,\varphi\,$ in addition to $\,\tau\,$ using cross-validations as well. In that case, pairs of the hyperparameters $(\tau, \varphi)$ should be chosen from some prespecified candidate set $\mathcal{T} \times \Phi \subset \R_+ \times \R_+$ (e.g. the Cartesian product of two grids) of a finite size $|\mathcal{T}| \times |\Phi| < \infty$. Optimal pairs $\,(\tau^{cv}, \varphi^{cv})\,$ and $\,(\tau^{oracle}, \varphi^{oracle})\,$ should be defined as the solutions to similar optimization problems as \eqref{tau_cv} ad \eqref{tau_oracle}, respectively, but this time the estimator $\,\ebeta_{\tau, \varphi}\,$ depends also on $\,\varphi$, and the optimization is over the finite candidate set. Then a result similar to Theorem~\ref{Th:CV} holds, but with$\,\log(Lr)\,$ in the bound replaced by $\,\log(|\mathcal{T}|\times|\Phi|)$.

\section{Discussion} \label{S:discussion}

We provide a new prospective on the non-sparse high-dimensional linear regression problem. The proposed family of GCT estimators serves as a bridge between two classical paradigms: sparse regression and principal components regression. W.r.t. the absolute errors, the fast decay of eigenvalues of the covariance is enough to ensure convergence even in high dimensions without any assumptions on the regression coefficients. Moreover, we argue that the relative errors are more appropriate in the high-dimensional regression with the eigenvalue decay, and that the complexity of a linear regression problem is characterized by the signal-to noise ratio (instead of the magnitude of noise) and the interaction between the covariance and regression coefficients, expressed by the joint effective dimension (instead of assumptions on the regression coefficients). It is not really important what the design matrix is, and we do not need to impose restrictive assumptions on it, if we choose relative errors as a measure of performance and standardize the data properly. The NCT estimator is minimax optimal for any design over suitable parameter classes in this paradigm. Hopefully, our insights shed some light on the nature of the non-sparse high-dimensional linear regression.

	We leave several important directions for further investigation.
		First of all, despite our joint effective dimension is quite well-motivated, it does not mean that there are no other structural assumptions related to the eigenvalue decay. New discoveries in the structure of the high-dimensional linear regression can potentially lead to other procedures, whose minimax optimality should be analyzed over appropriate parameter classes.
		
	Also, even though the analysis of our structural assumptions and procedures for fixed design seems quite complete, there are unanswered questions in the random design setting. It is not clear whether the bounds of Theorem~\ref{Th2} can be improved and whether the associated relative errors can be represented in a convenient way. An uncertainty brought by the covariance matrix requires developing and applying new advanced statistical tools for the analysis of minimax optimality in random design linear regression in high dimensions.

	Furthermore, from the numerical experiments, postponed to Appendix~\ref{S:experiments}, we observe that even regularized estimators (such as NCT)  behave unexpectedly around the interpolation threshold $\,d = n$. Our theoretical results do not predict the bumps that errors as functions of the dimension exhibit in this region. This definitely worth studying in the future.
	
In addition, the uncertainty quantification for the estimated parameters and function values is of significant interest, as well as 
possible extensions of our ideas to nonparametric regression in reproducing kernel Hilbert space (RKHS). 
Beyond the linear model, can penalized quasi-likelihood on canonical parameters share similar properties to those in the regression problem?

\section*{Acknowledgements}
We thank gratefully the Editor, the Associate Editor and the Referees for constructive comments and valuable suggestions which led to significant improvements on the paper.  The research was supported by ONR grant N00014-19-1-2120, NSF grant DMS-1662139,  and NIH grant 2R01-GM072611-14.

\begin{appendices}

\section{Simulation studies} \label{S:experiments}
We compare the following methods:
\begin{itemize}
	\item ``\textsf{NCT}'': Natural Canonical Thresholding estimator~\eqref{estimator} with efficient hyperparameter tuning by 10-fold CV.
	\item ``\textsf{GCT}'': Generalized Canonical Thresholding estimator~\eqref{estimator2} with $\,\varphi=1$, the soft thresholding, and  with efficient hyperparameter tuning by 10-fold CV.
	\item ``\textsf{OLS}'': Ordinary Least Squares. When $\,d > n$, the min norm solution is considered.
	\item ``\textsf{PCR}'': Principal Component Regression. The number of PCs is chosen by 10-fold CV.
	\item ``\textsf{Ridge}'': Ridge regression with 10-fold CV (default implementation from R-package \textsf{glmnet}).
	\item ``\textsf{LASSO}'': LASSO with 10-fold CV (default implementation from R-package \textsf{glmnet}).
\end{itemize}

\begin{figure}
     \begin{subfigure}{\textwidth}
         \includegraphics[width=0.49\textwidth]{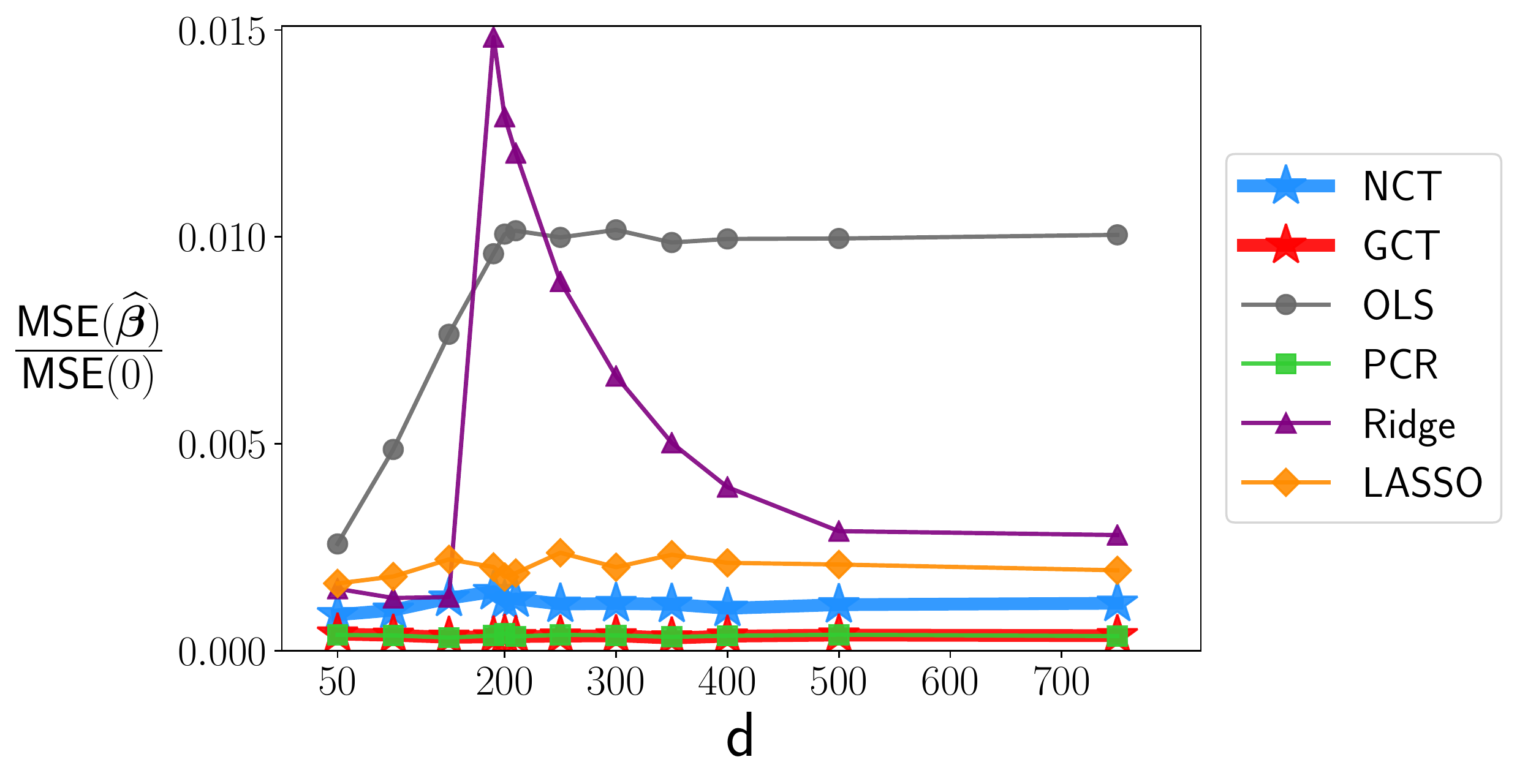}
         \hfill\hspace{0.3cm}
         \includegraphics[width=0.49\textwidth]{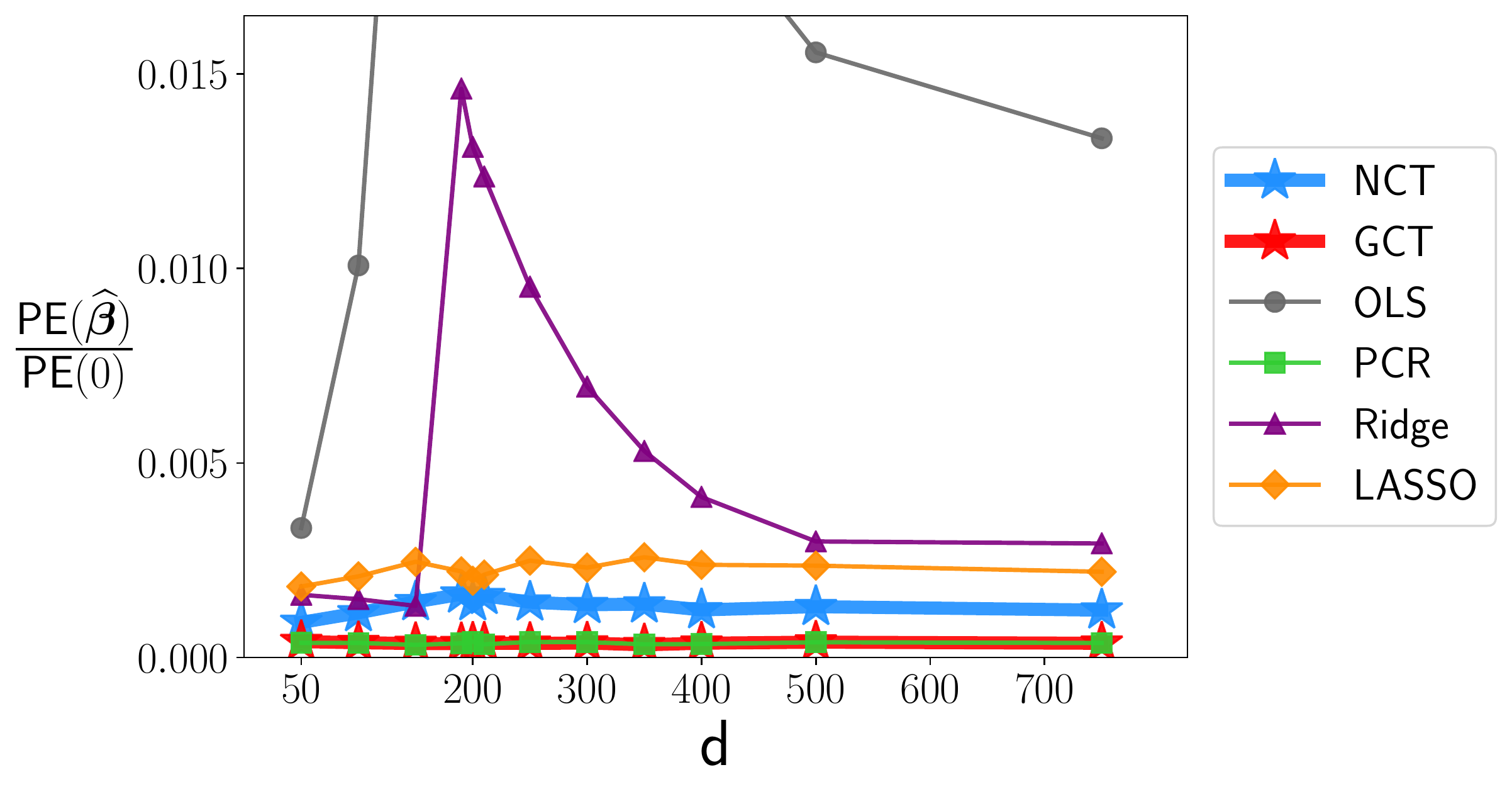}
         \vspace{-0.7cm}
         \caption{$a = 2$, $\,b = 2$;}
     \end{subfigure}
     \vfill
     \vspace{1cm}
     \begin{subfigure}{\textwidth}
         \includegraphics[width=0.49\textwidth]{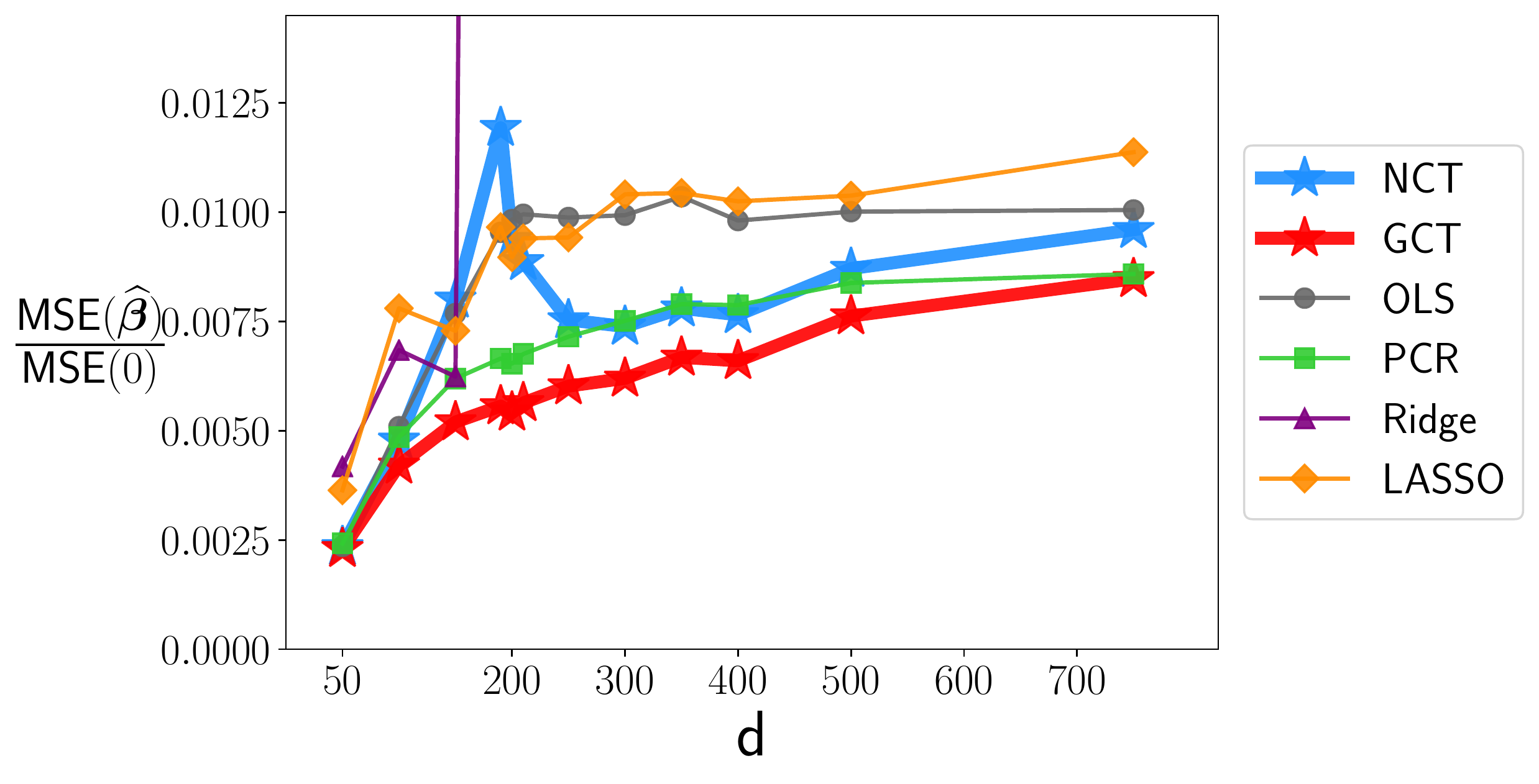}
         \hfill\hspace{0.3cm}
         \includegraphics[width=0.49\textwidth]{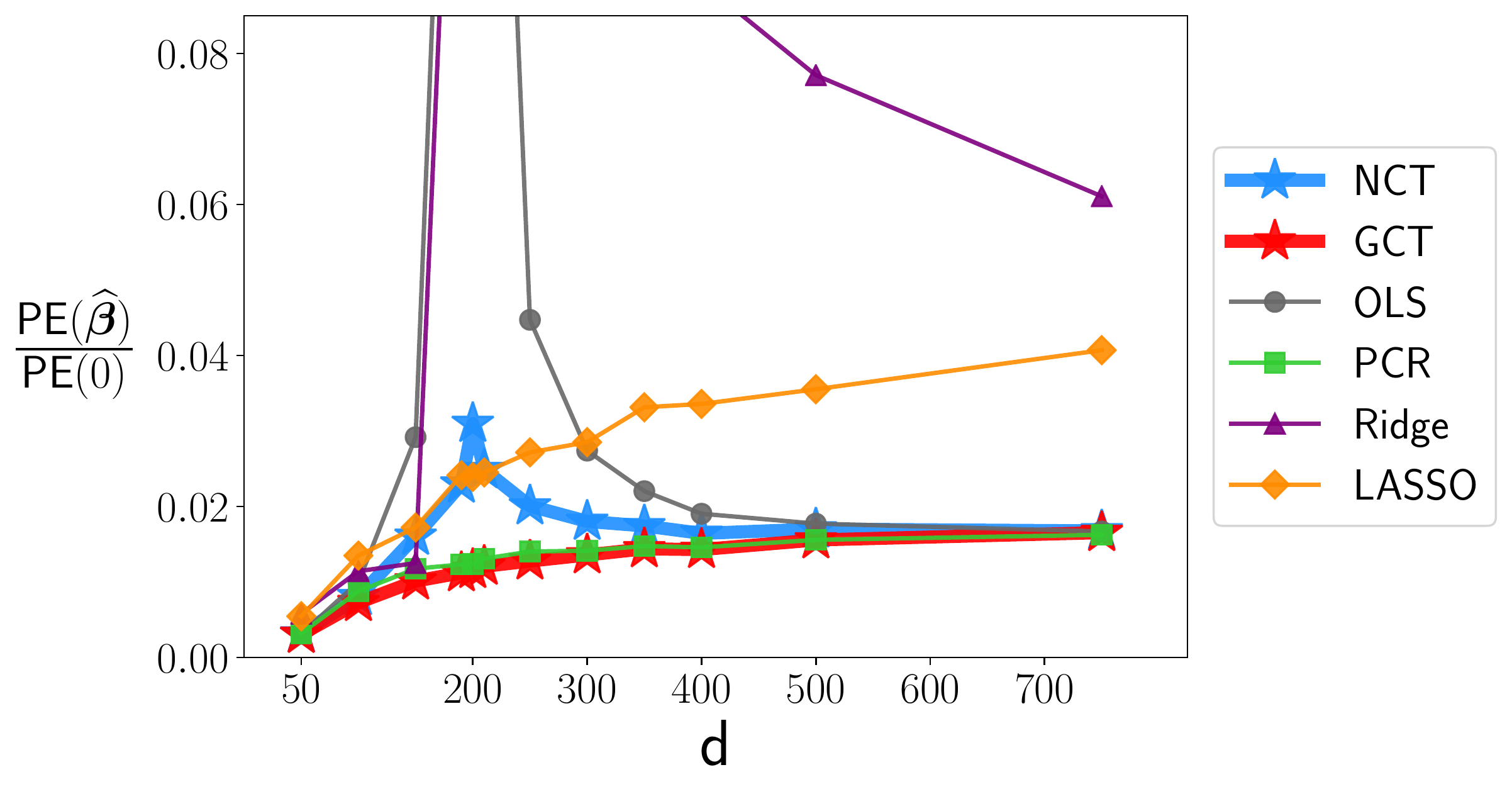}
         \vspace{-0.7cm}
         \caption{$a = 1$, $\,b = 0.5$;}
     \end{subfigure}
     \vfill
     \vspace{1cm}
     \begin{subfigure}{\textwidth}
         \includegraphics[width=0.49\textwidth]{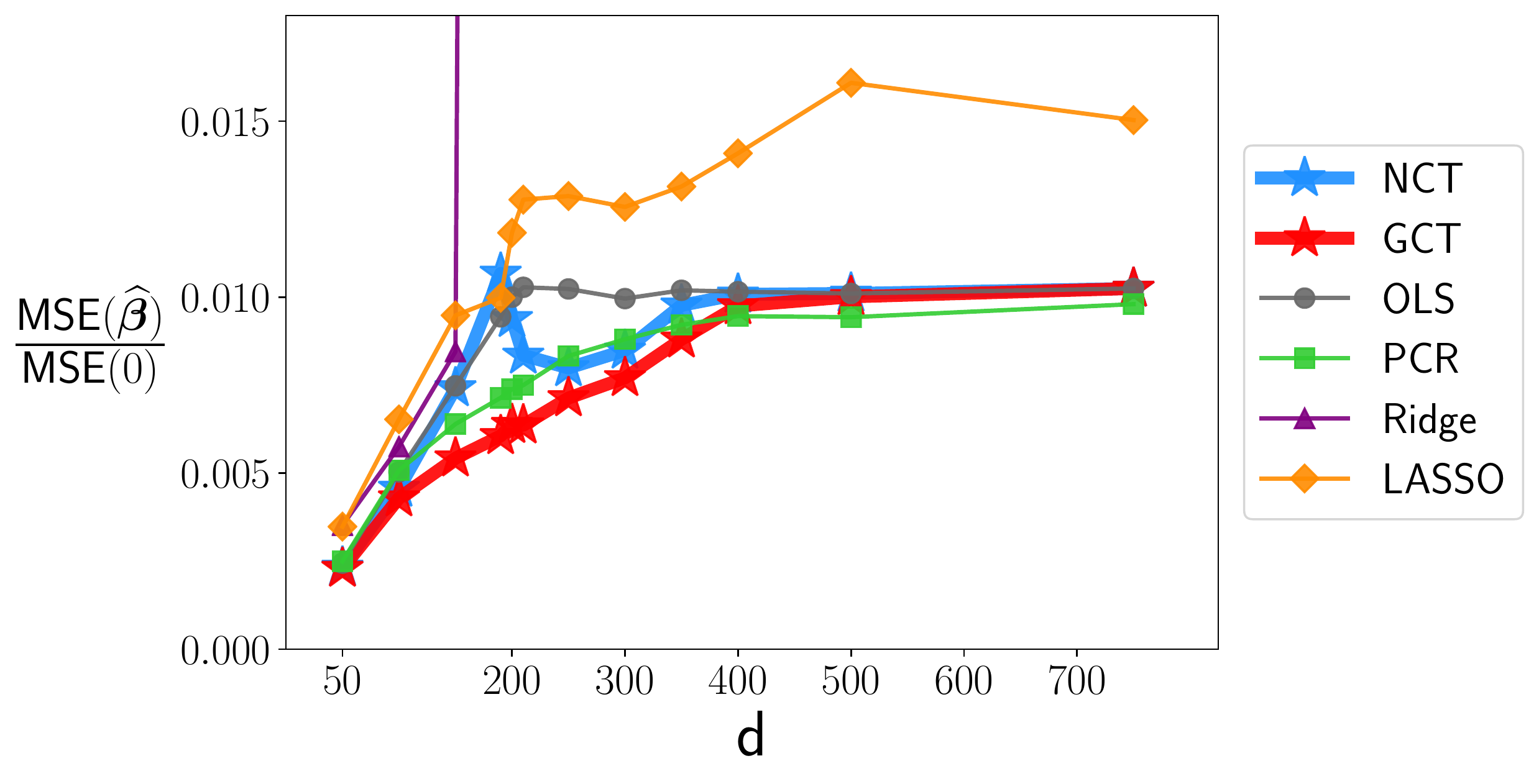}
         \hfill\hspace{0.3cm}
         \includegraphics[width=0.49\textwidth]{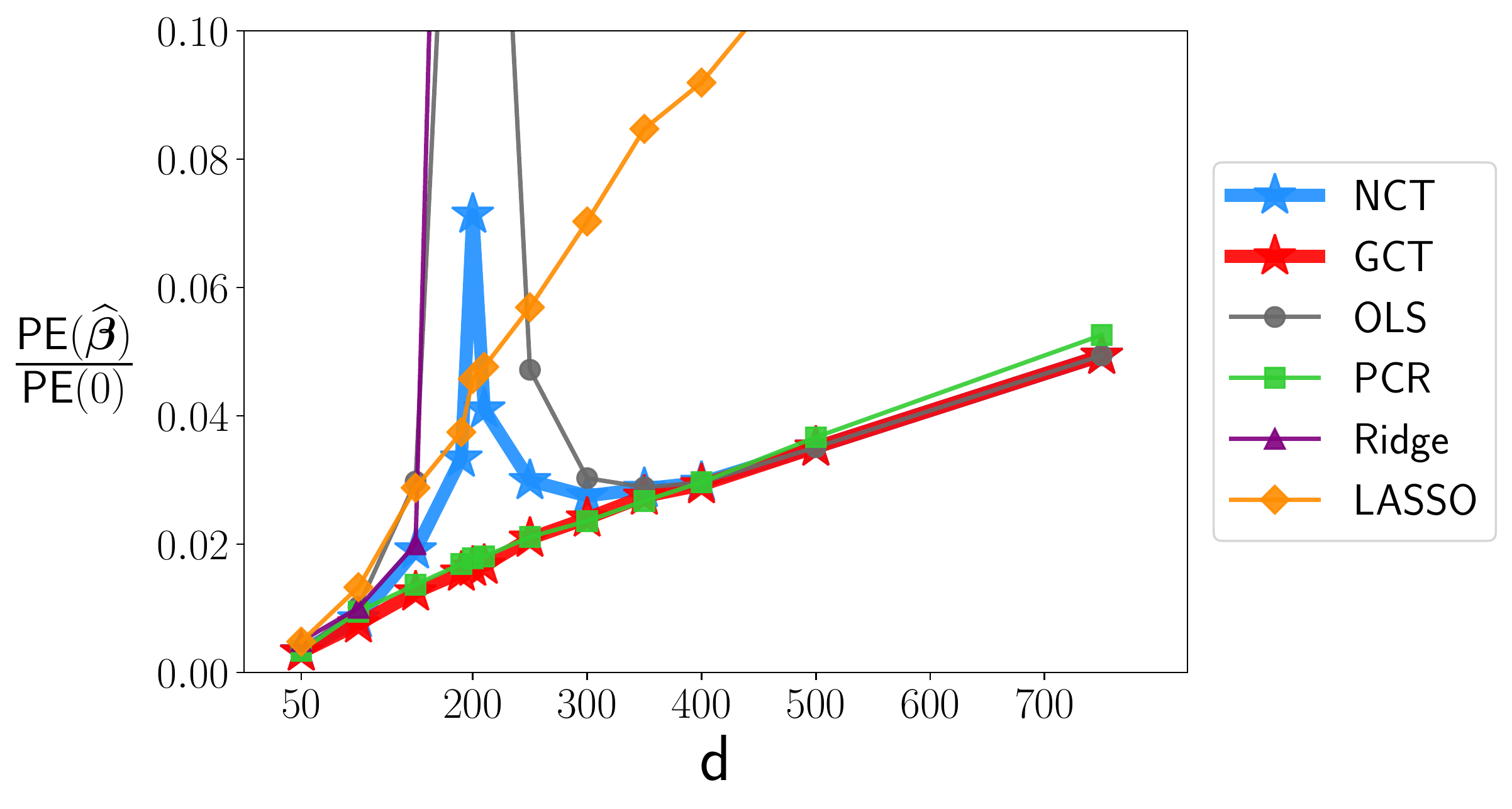}
         \vspace{-0.7cm}
         \caption{$a = 0.5$, $\,b = 1$;}
     \end{subfigure}
        \caption{The relative errors $\,\mathsf{MSE}(\ebeta)/\mathsf{MSE}(0)\,$ (left) and $\,\mathsf{PE}(\ebeta)/\mathsf{PE}(0)\,$ (right) for different estimators with $\,n=200$, $\,\snr=10$. Polynomial decay of eigenvalues and coefficients in eigenbasis: $\,\lambda_j = j^{-a}$, $\,\bu_j^\T \bbeta = j^{-b}$. }
        \label{plots1}
\end{figure}

\begin{figure}
     \begin{subfigure}{\textwidth}
         \includegraphics[width=0.49\textwidth]{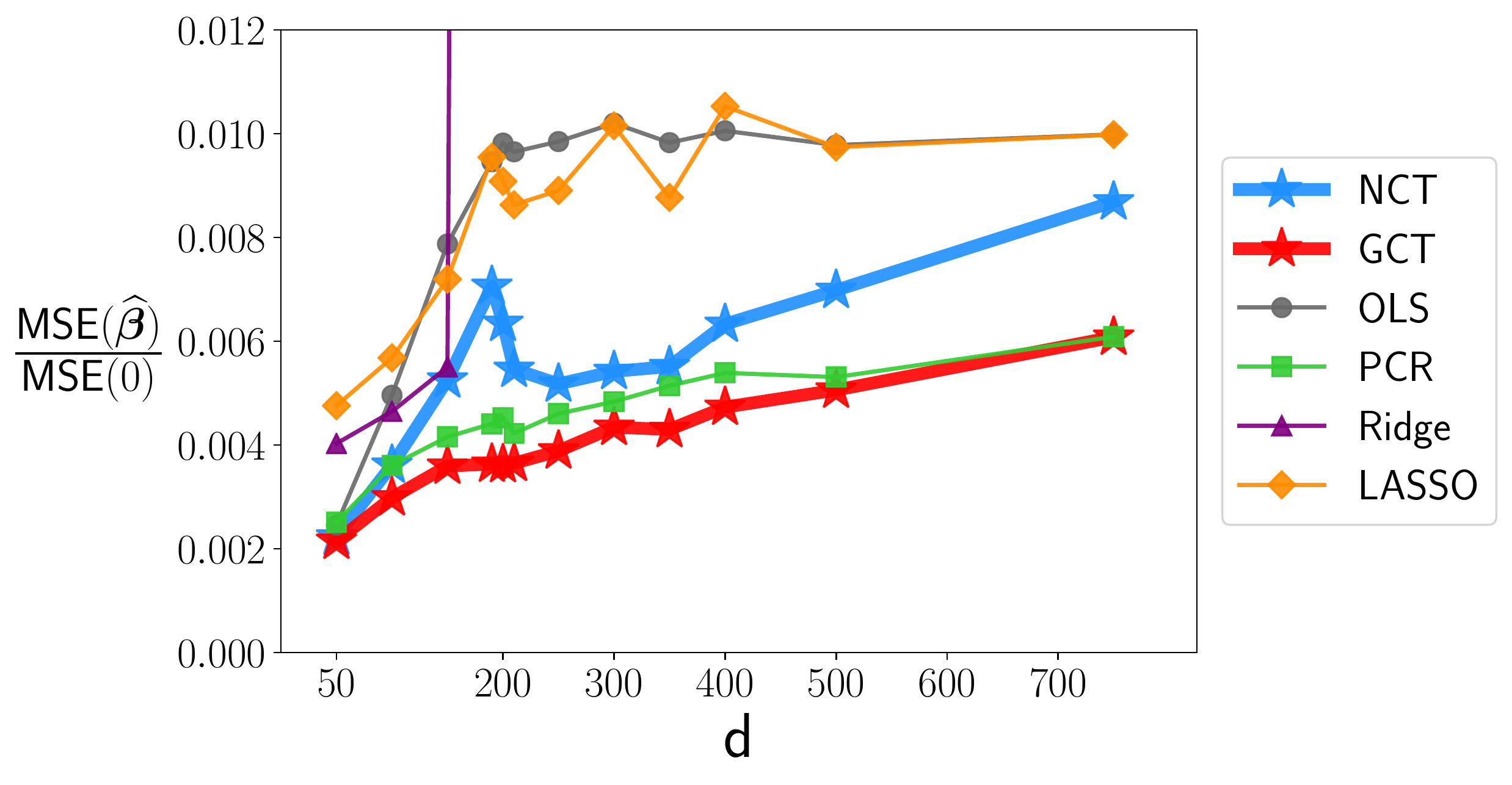}
         \hfill\hspace{0.3cm}
         \includegraphics[width=0.49\textwidth]{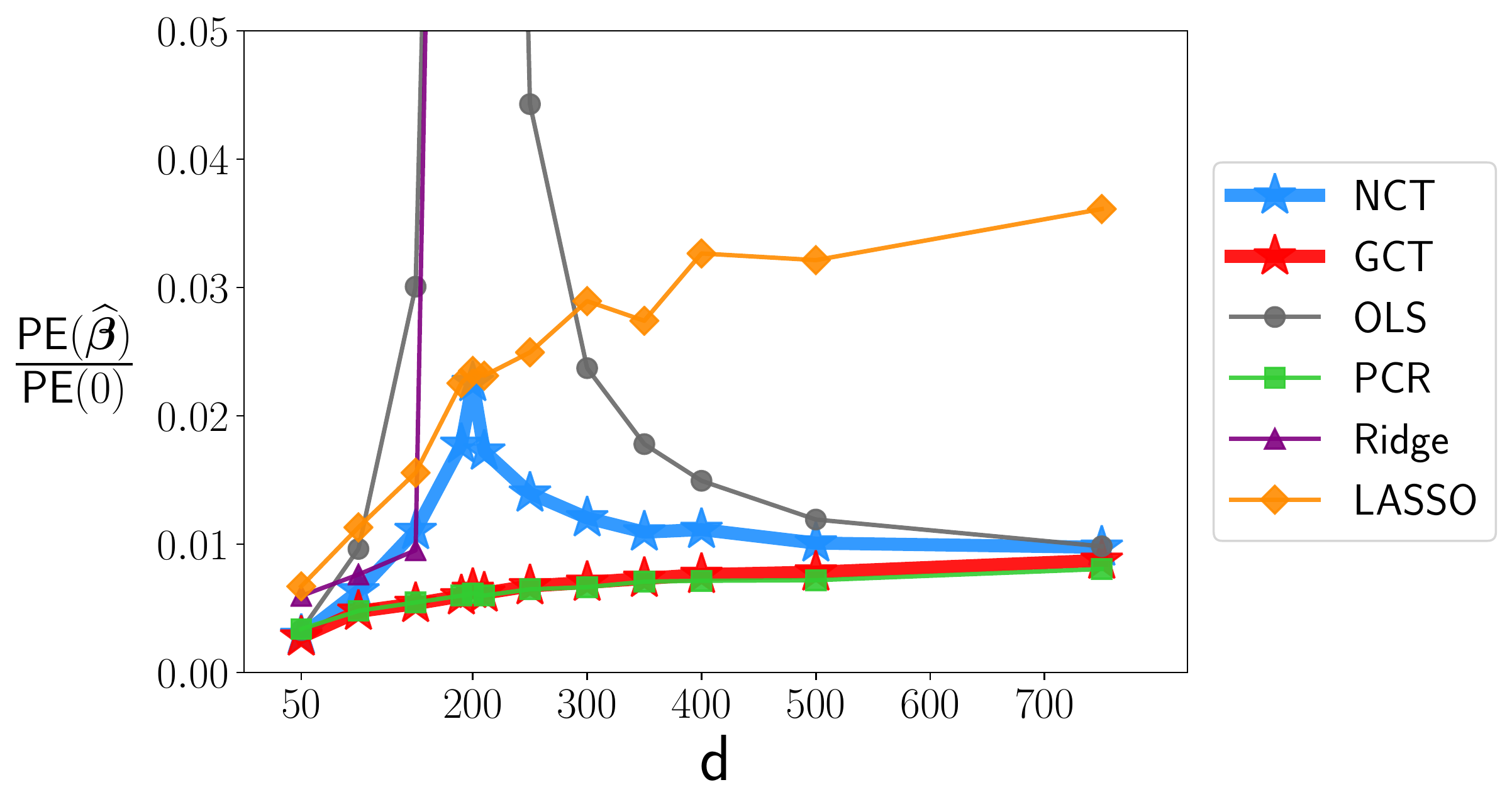}
         \vspace{-0.7cm}
         \caption{$a = 1$, $\,\bu_j^\T \bbeta = 1$ for $\,j=1,\ldots, 10\,$ and $0$ otherwise;}
     \end{subfigure}
     \vfill
     \vspace{1cm}
     \begin{subfigure}{\textwidth}
         \includegraphics[width=0.49\textwidth]{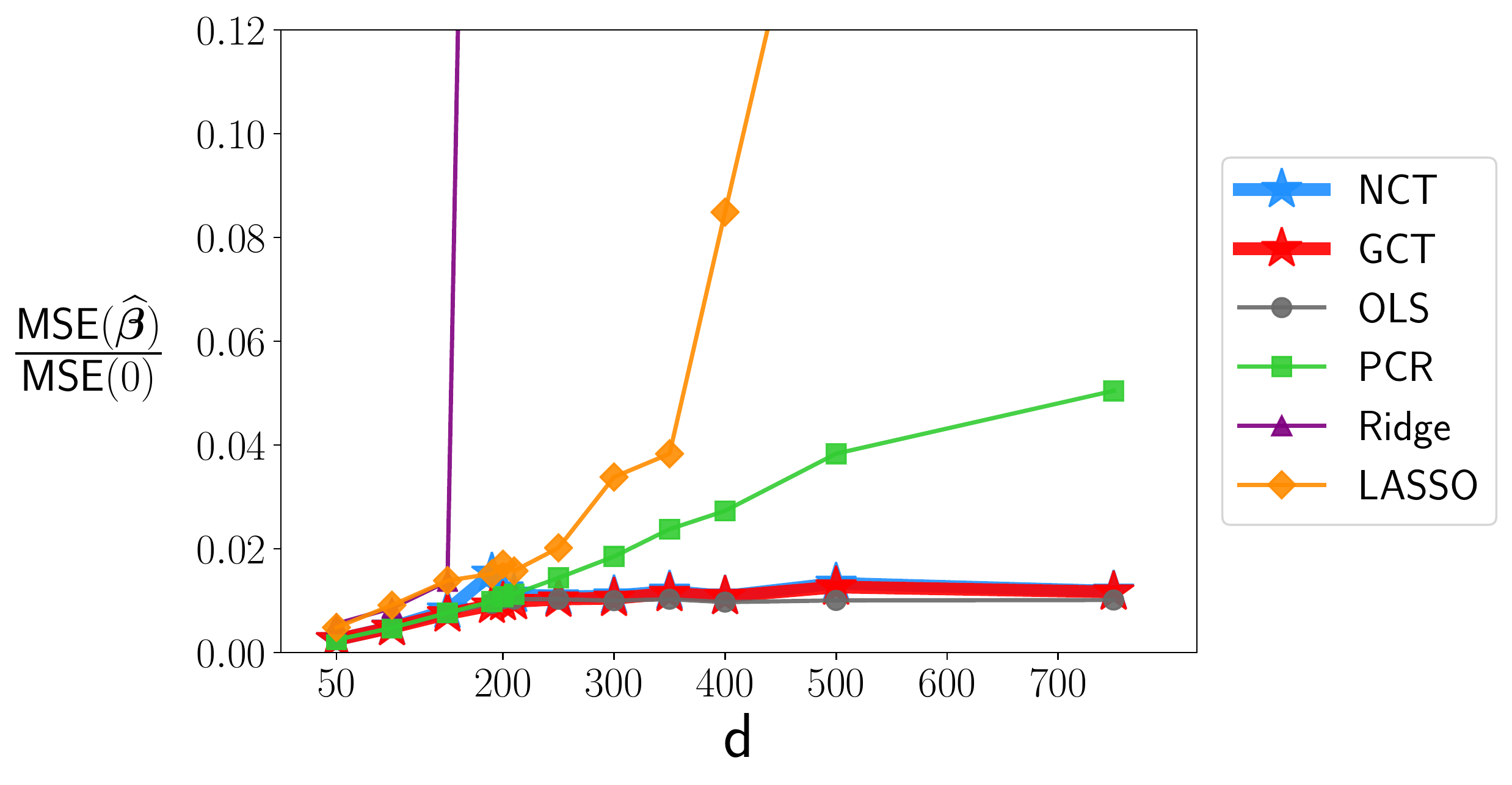}
         \hfill\hspace{0.3cm}
         \includegraphics[width=0.49\textwidth]{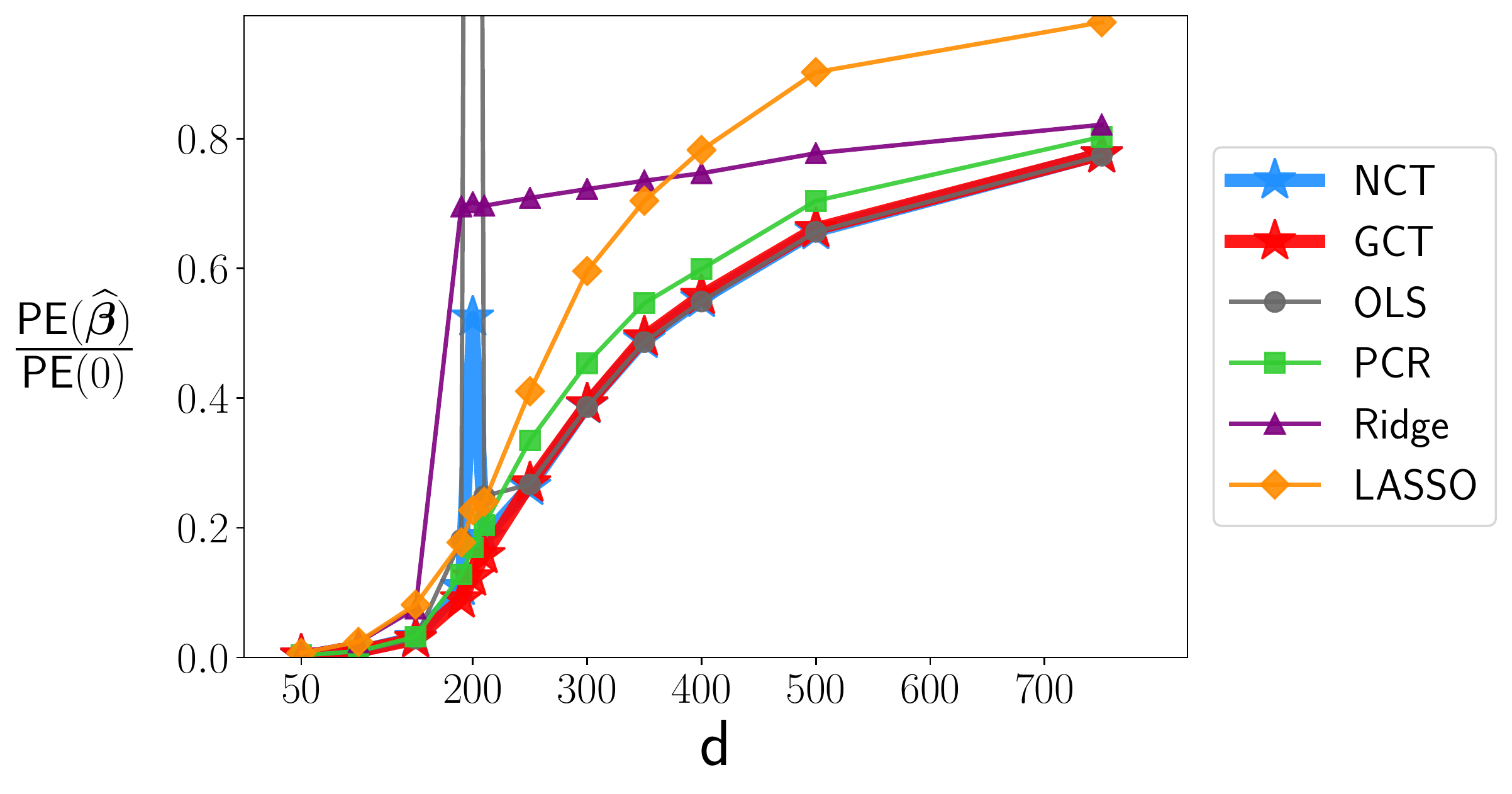}
         \vspace{-0.7cm}
         \caption{$a = 0.1$, $\,\bu_j^\T \bbeta = 1\,$ for 10 randomly chosen $\,j\in\{ d-25,\ldots, d\}$, and the rest components are i.i.d. $\,\mathcal{N}(0, d^{-1})$;}
     \end{subfigure}
     \vfill
     \vspace{1cm}
     \begin{subfigure}{\textwidth}
         \includegraphics[width=0.49\textwidth]{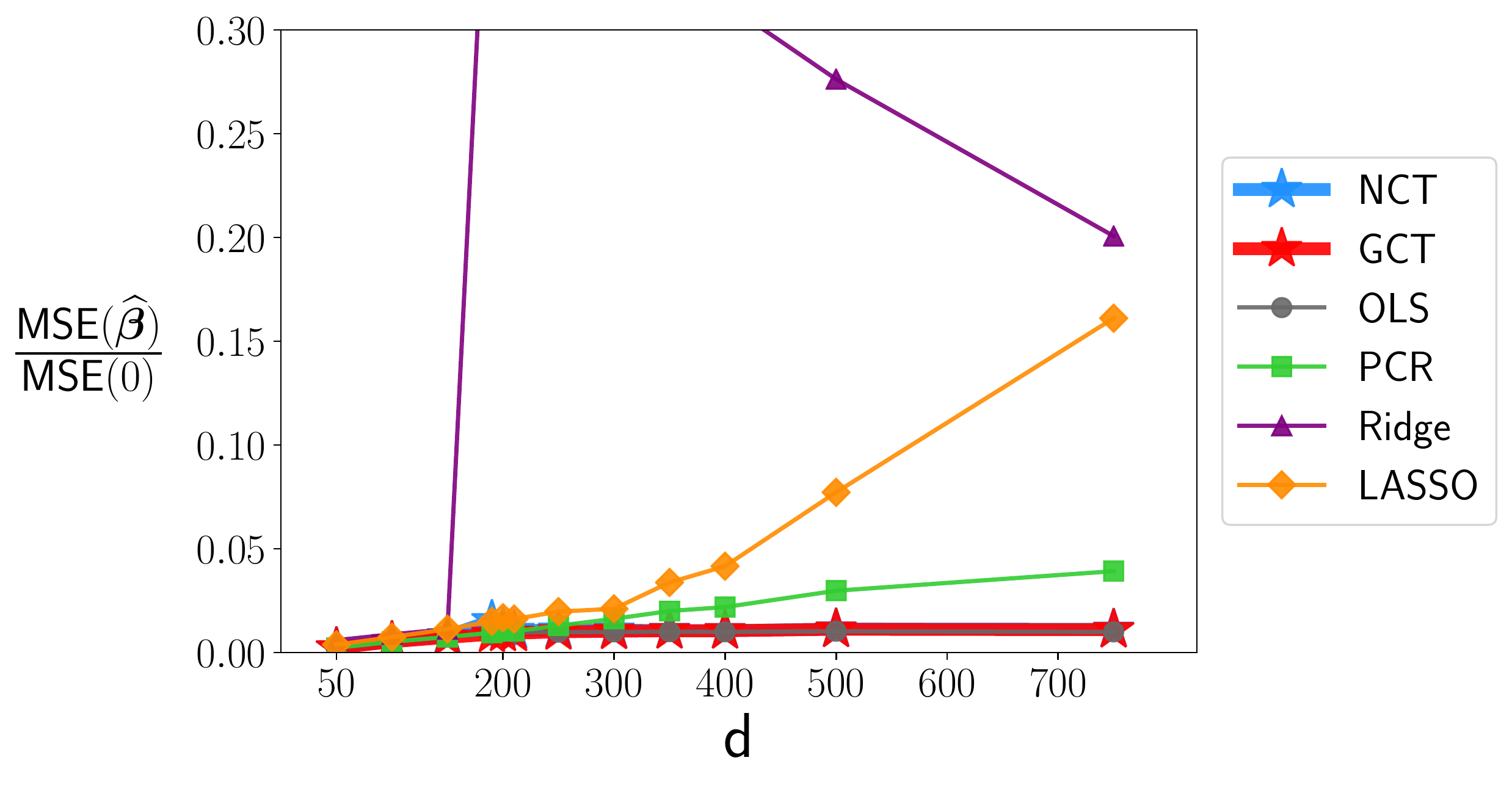}
         \hfill \hspace{0.3cm}
         \includegraphics[width=0.49\textwidth]{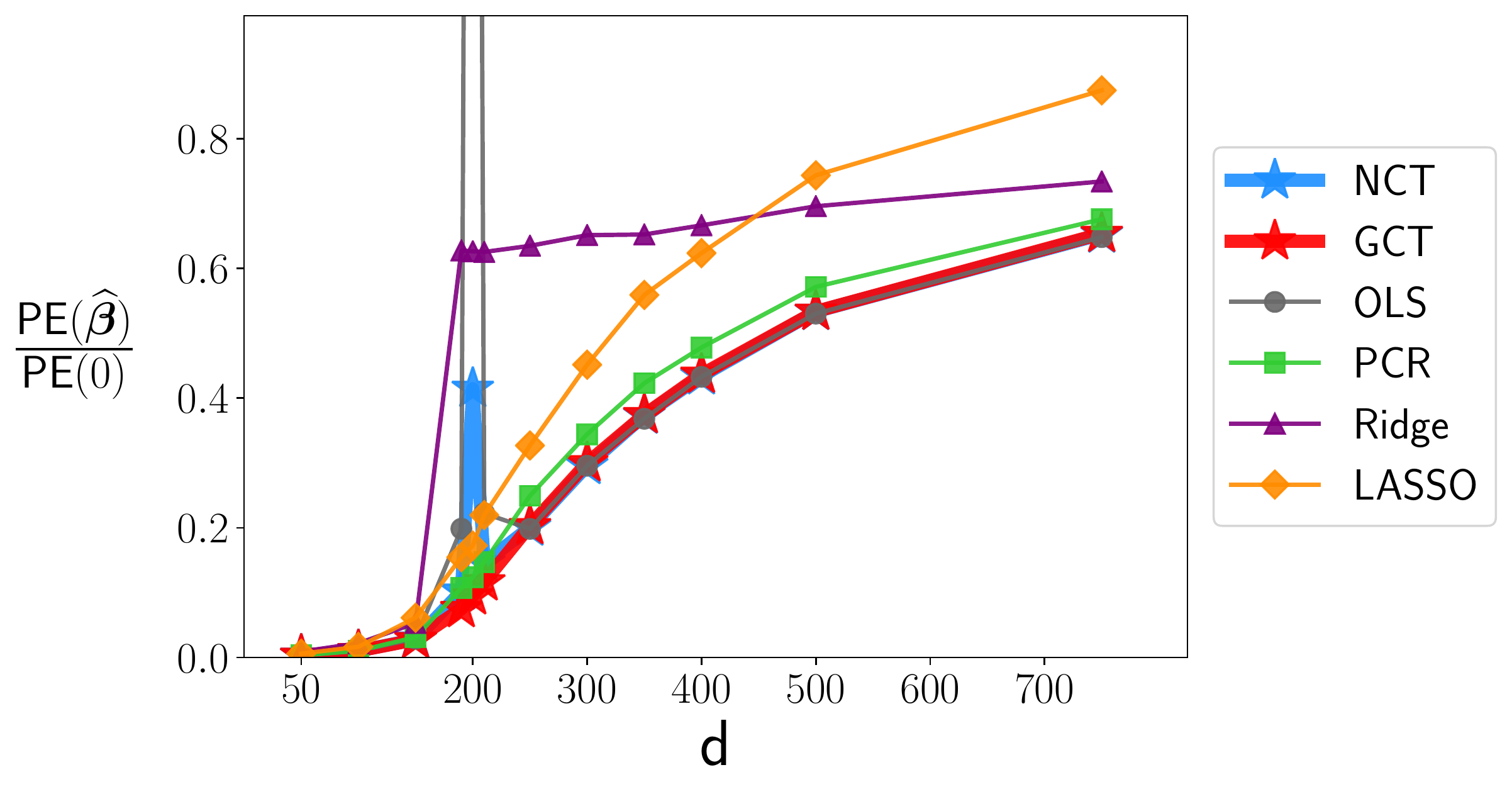}
         \vspace{-0.7cm}
         \caption{$a = 0.5$, $\,\bU^\T \bbeta \sim \mathcal{N}(0, \Id_d)$;}
     \end{subfigure}
        \caption{The relative errors $\,\mathsf{MSE}(\ebeta)/\mathsf{MSE}(0)\,$ (left) and $\,\mathsf{PE}(\ebeta)/\mathsf{PE}(0)\,$ (right) for different estimators with $\,n=200$, $\,\snr=10$. Polynomial decay of eigenvalues $\,\lambda_j = j^{-a}\,$ and different regimes of coefficients in eigenbasis $\,\bu_j^\T \bbeta$.}
        \label{plots2}
\end{figure}

\begin{figure}
     \begin{subfigure}{\textwidth}
         \includegraphics[width=0.49\textwidth]{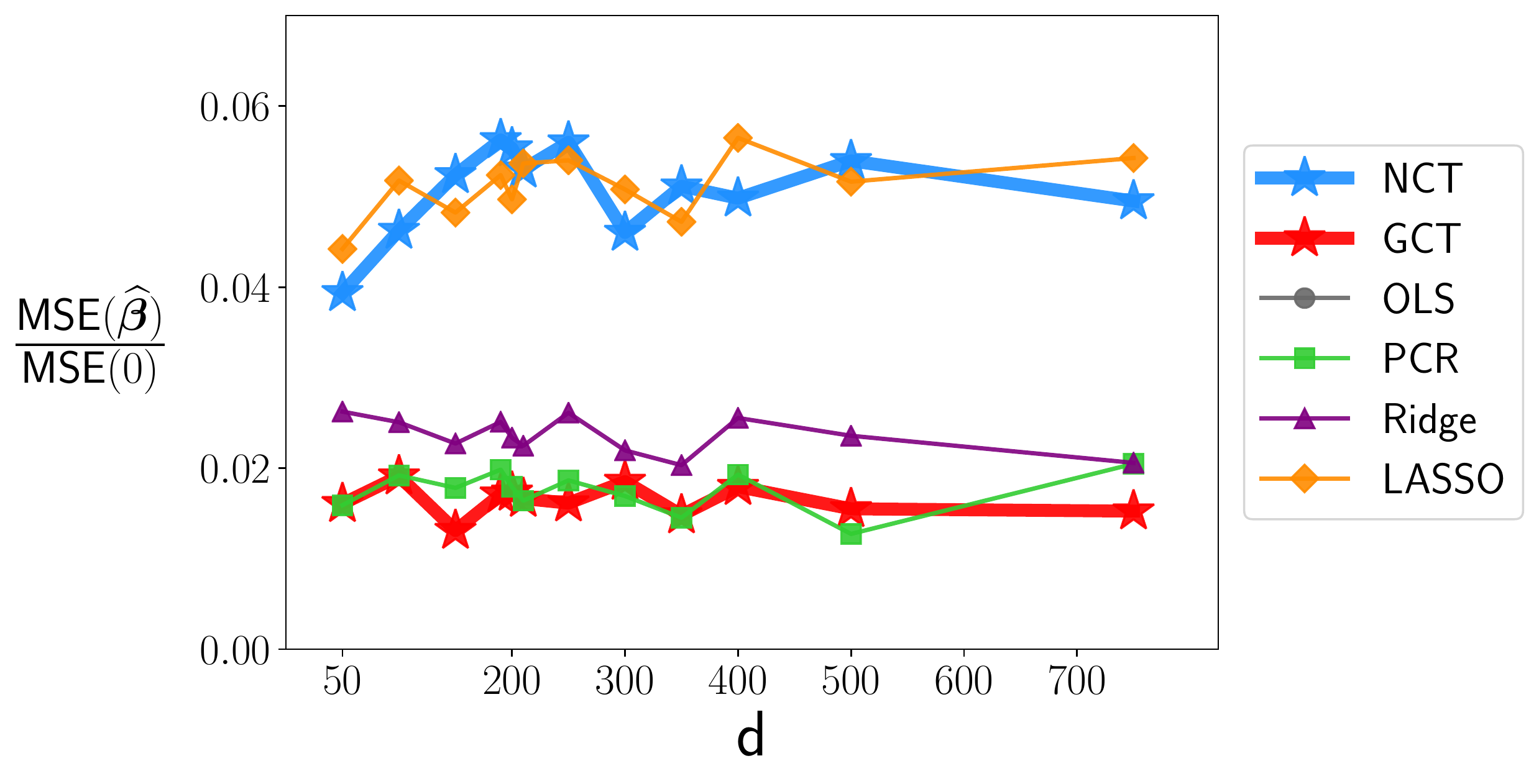}
         \hfill\hspace{0.3cm}
         \includegraphics[width=0.49\textwidth]{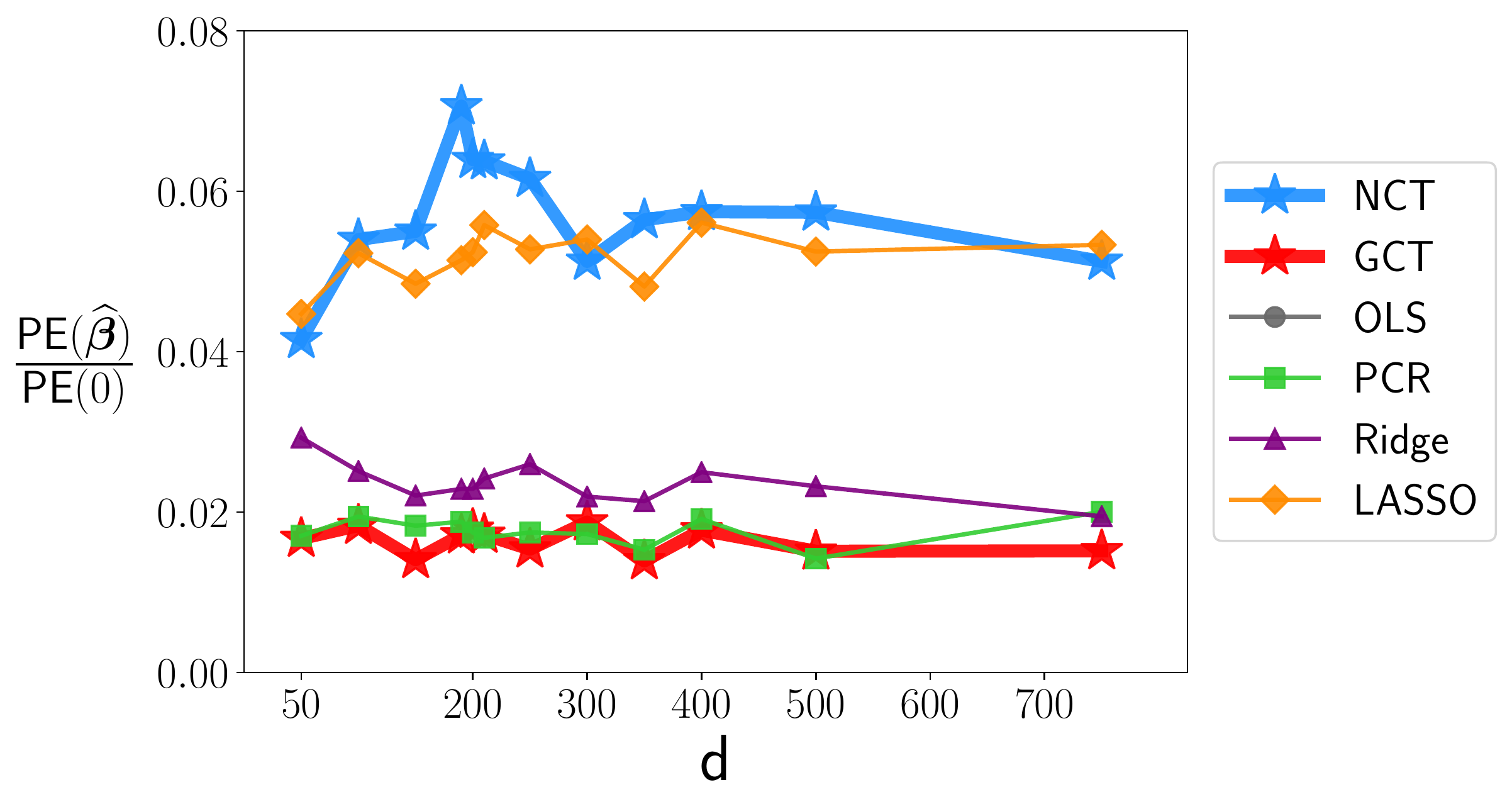}
         \vspace{-0.5cm}
         \caption{$a = 2$, $\,b = 2$;}
     \end{subfigure}
     \vfill
     \vspace{1cm}
     \begin{subfigure}{\textwidth}
         \includegraphics[width=0.49\textwidth]{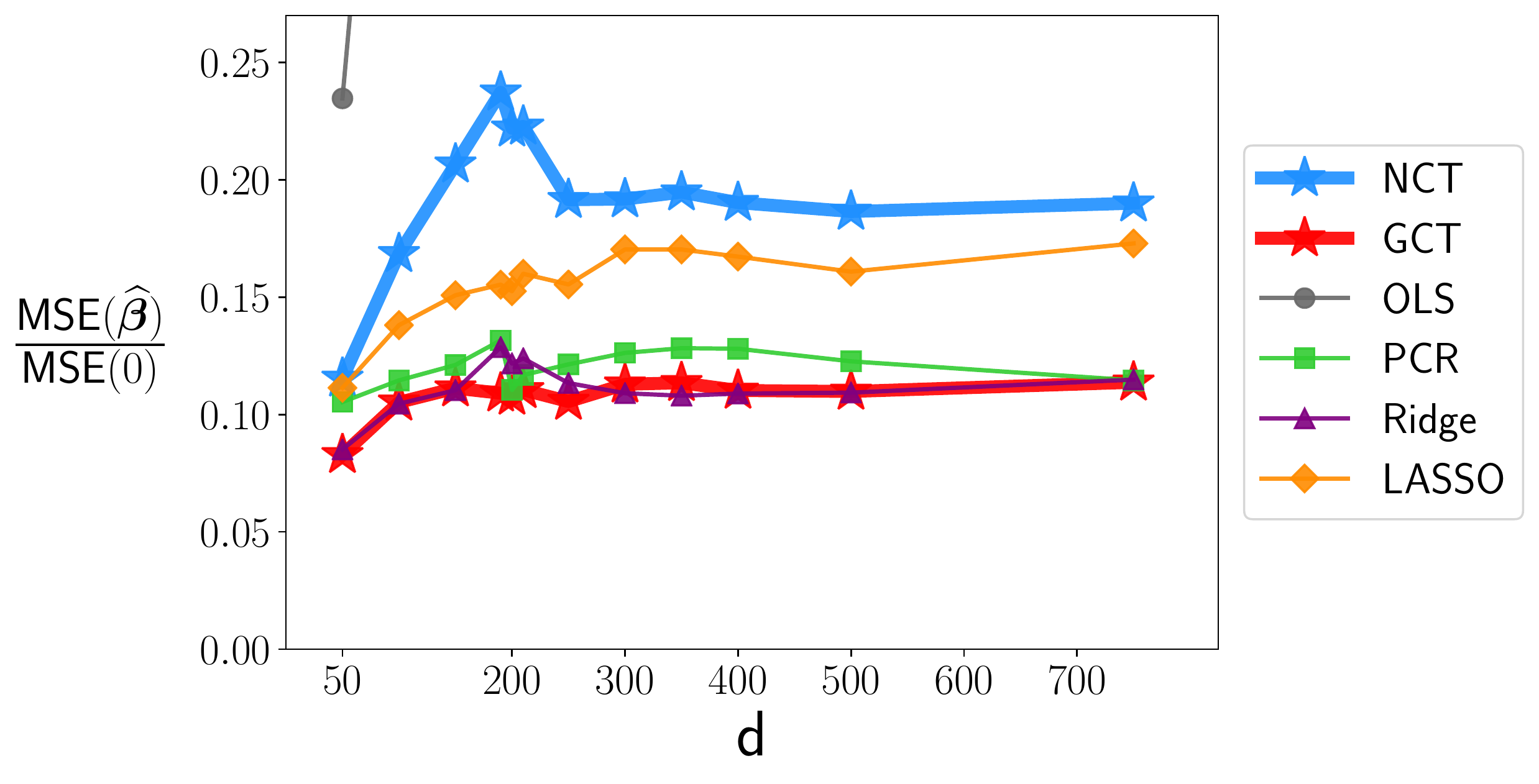}
         \hfill\hspace{0.3cm}
         \includegraphics[width=0.49\textwidth]{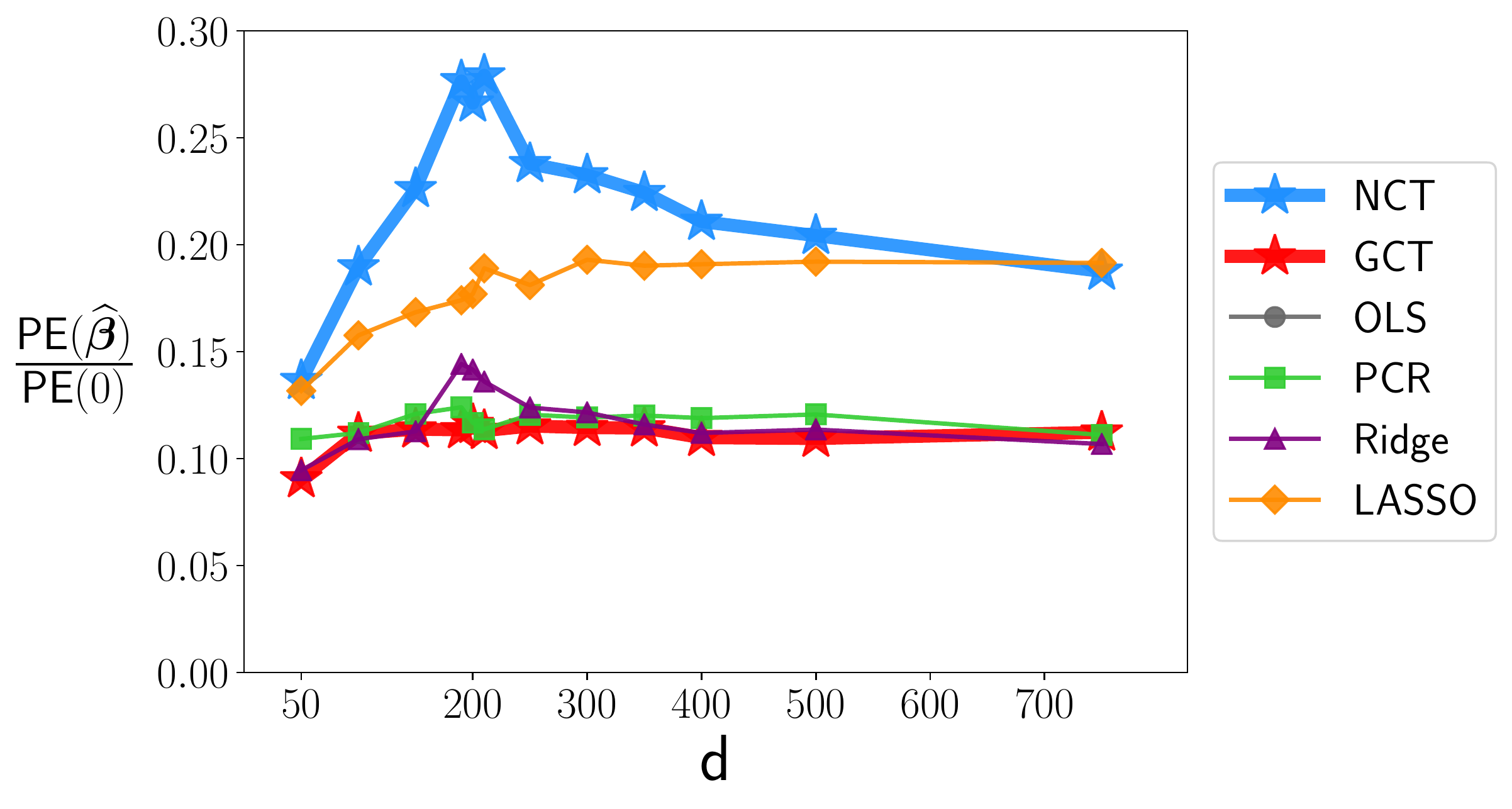}
         \vspace{-0.5cm}
         \caption{$a = 1$, $\,b = 0.5$;}
     \end{subfigure}
     \vfill
     \vspace{1cm}
     \begin{subfigure}{\textwidth}
         \includegraphics[width=0.49\textwidth]{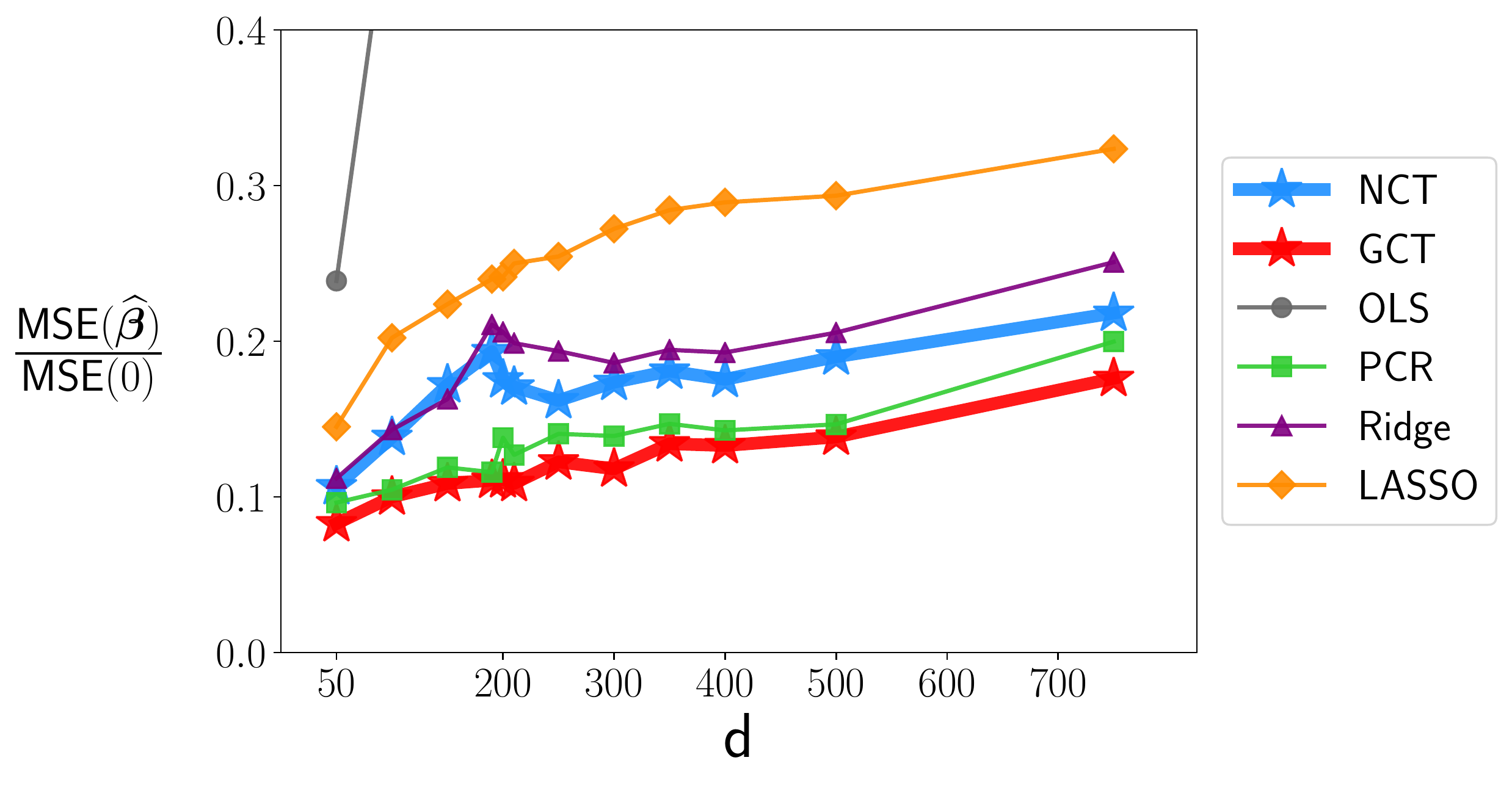}
         \hfill\hspace{0.3cm}
         \includegraphics[width=0.49\textwidth]{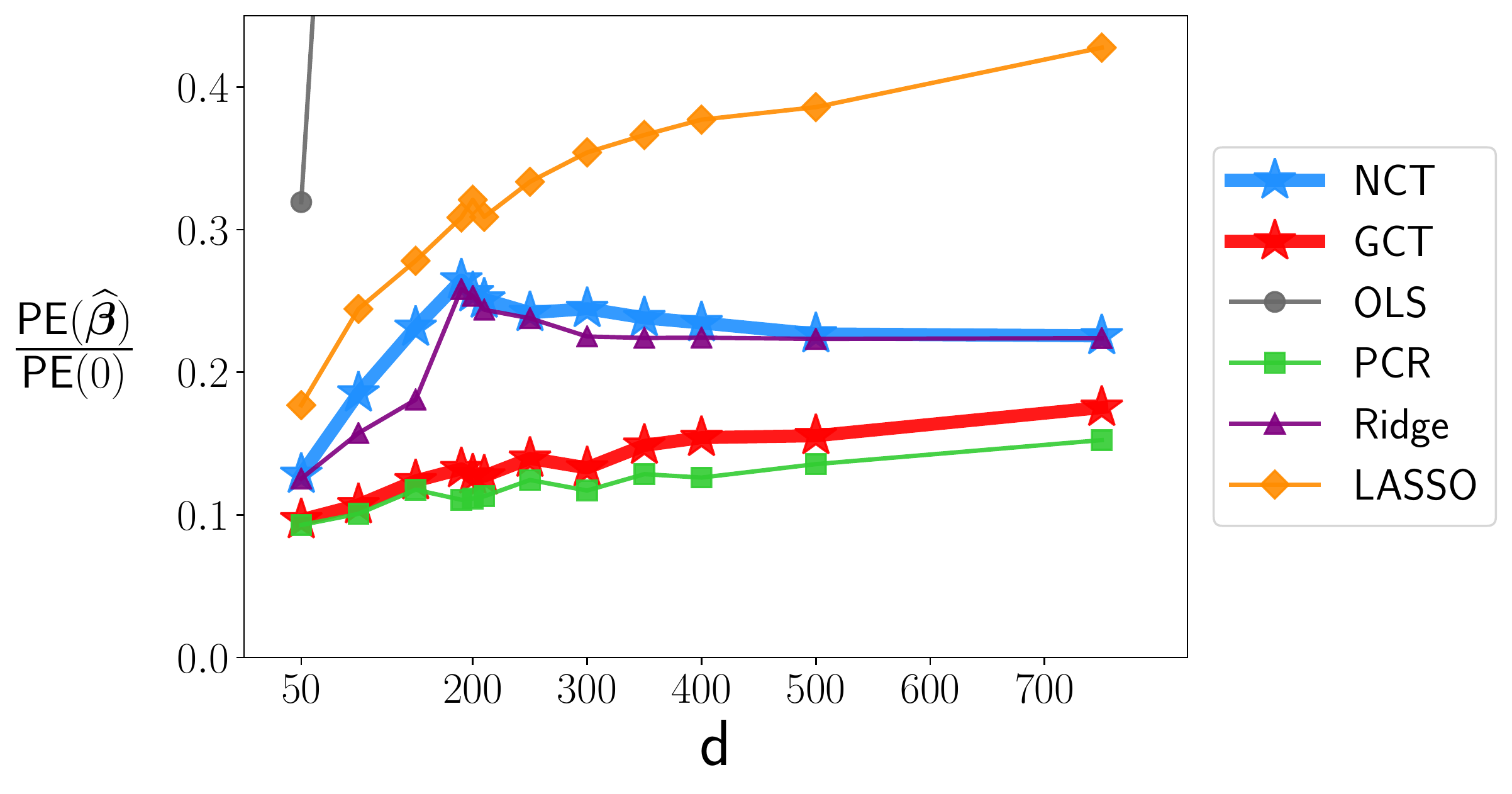}
         \vspace{-0.5cm}
         \caption{$a = 0.5$, $\,b = 1$;}
     \end{subfigure}
        \caption{The relative errors $\,\mathsf{MSE}(\ebeta)/\mathsf{MSE}(0)\,$ (left) and $\,\mathsf{PE}(\ebeta)/\mathsf{PE}(0)\,$ (right) for different estimators with $\,n=200$, $\,\snr=1$. Polynomial decay of eigenvalues and coefficients in eigenbasis: $\,\lambda_j = j^{-a}$, $\,\bu_j^\T \bbeta = j^{-b}$. }
        \label{plots3}
\end{figure}

\begin{figure}
     \begin{subfigure}{\textwidth}
         \includegraphics[width=0.49\textwidth]{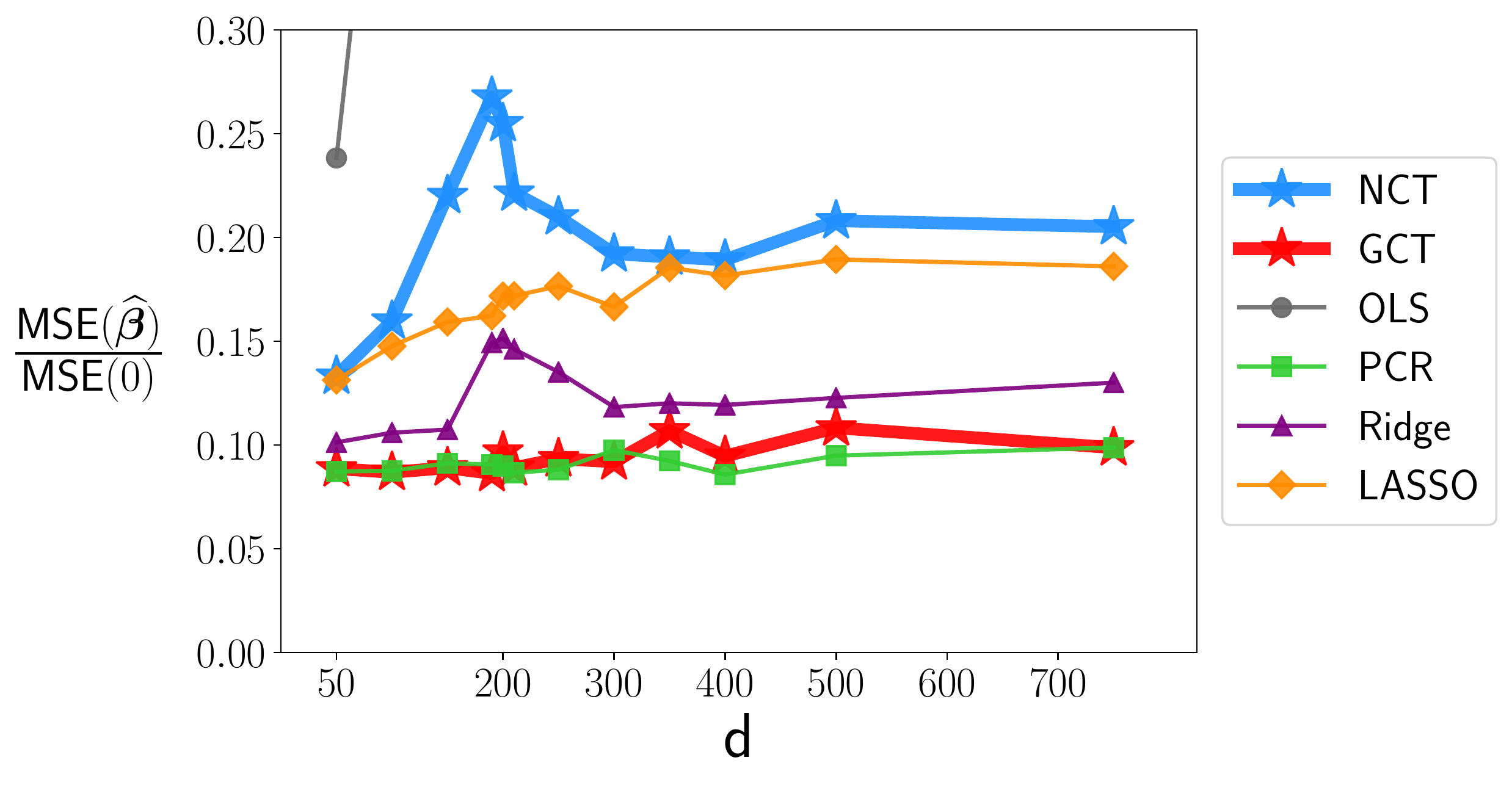}
         \hfill\hspace{0.3cm}
         \includegraphics[width=0.49\textwidth]{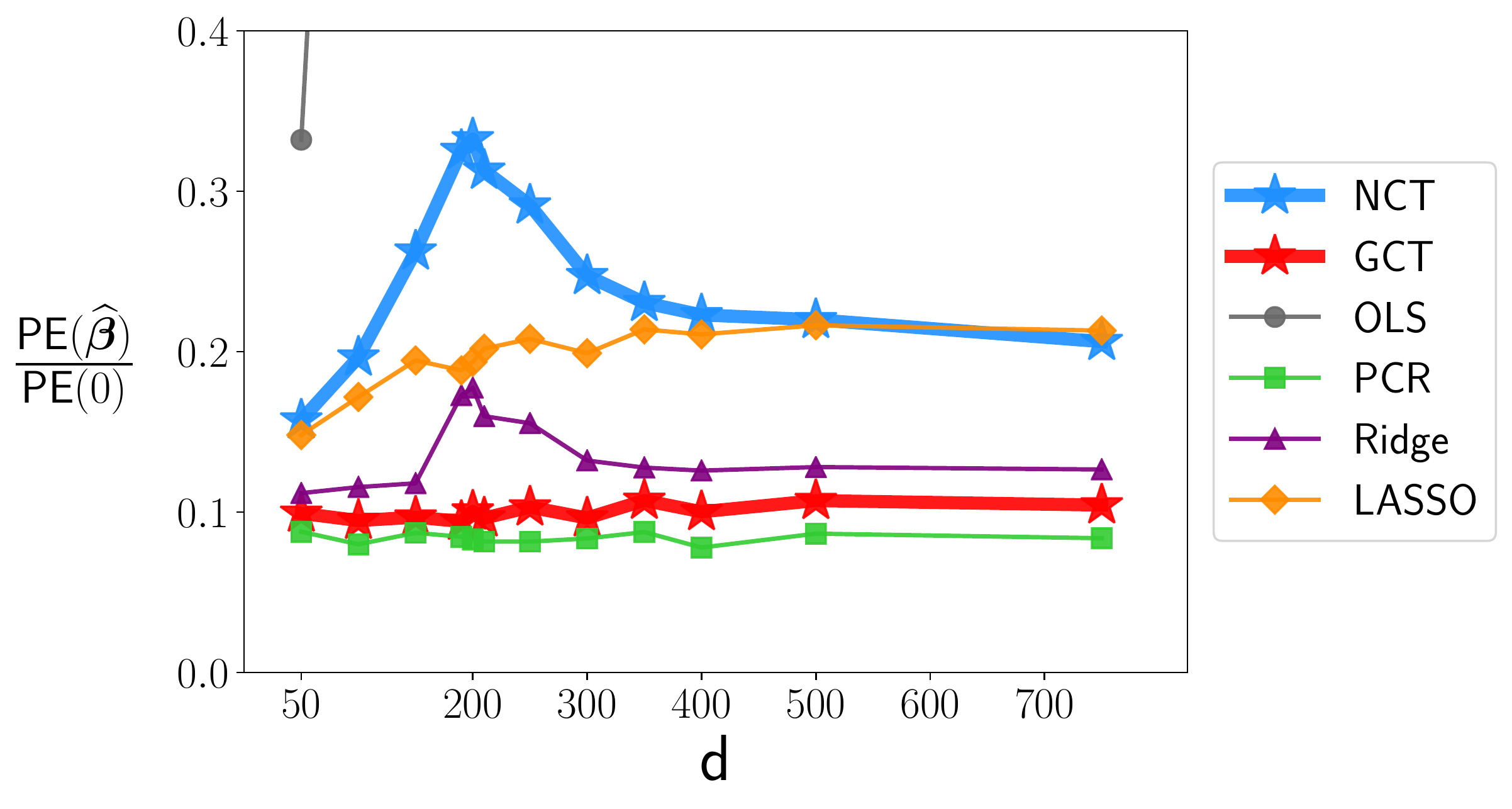}
         \vspace{-0.5cm}
         \caption{$a = 1$, $\,\bu_j^\T \bbeta = 1\,$ for $\,j=1,\ldots, 10\,$ and $0$ otherwise;}
     \end{subfigure}
     \vfill
     \vspace{1cm}
     \begin{subfigure}{\textwidth}
         \includegraphics[width=0.49\textwidth]{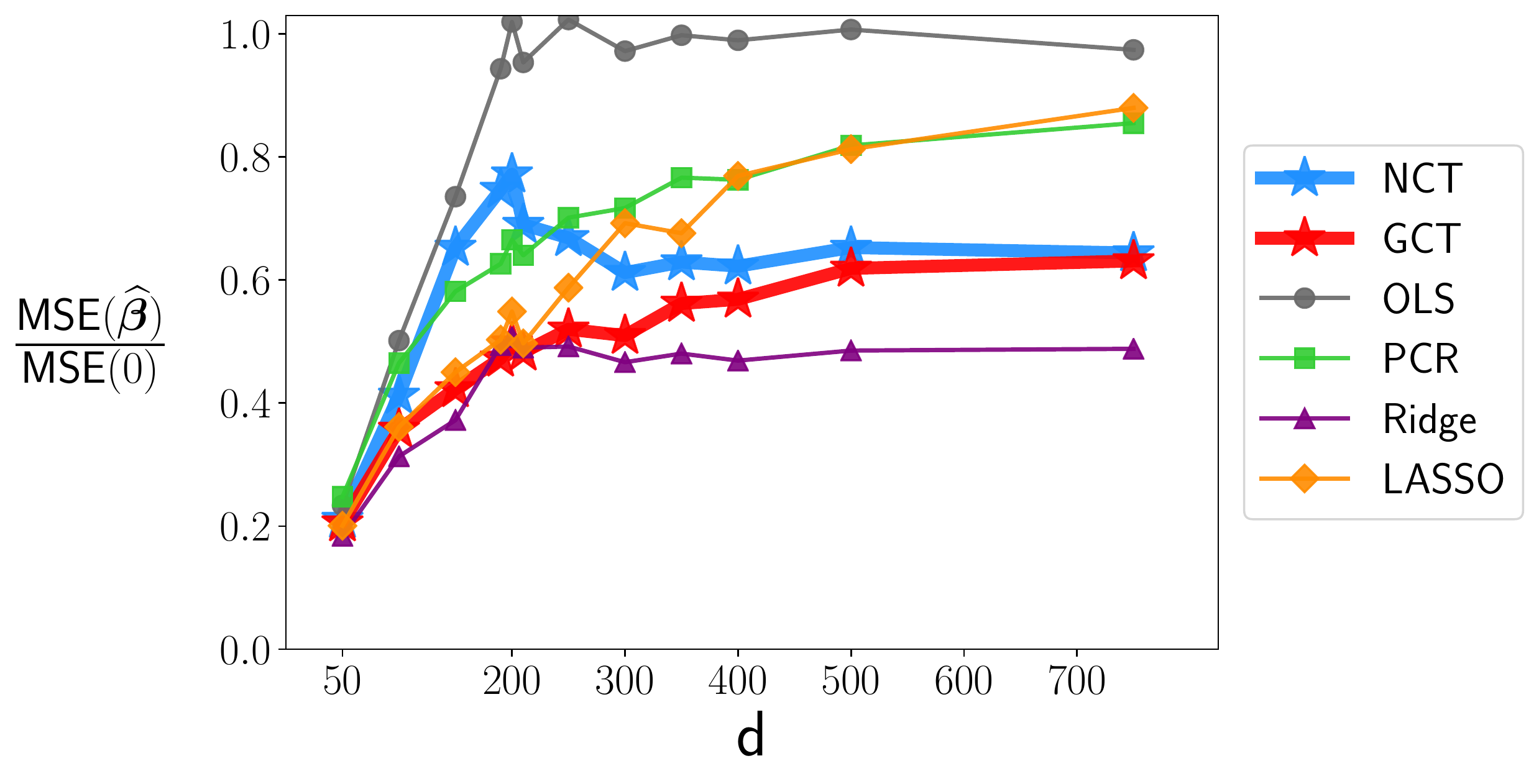}
         \hfill\hspace{0.3cm}
         \includegraphics[width=0.49\textwidth]{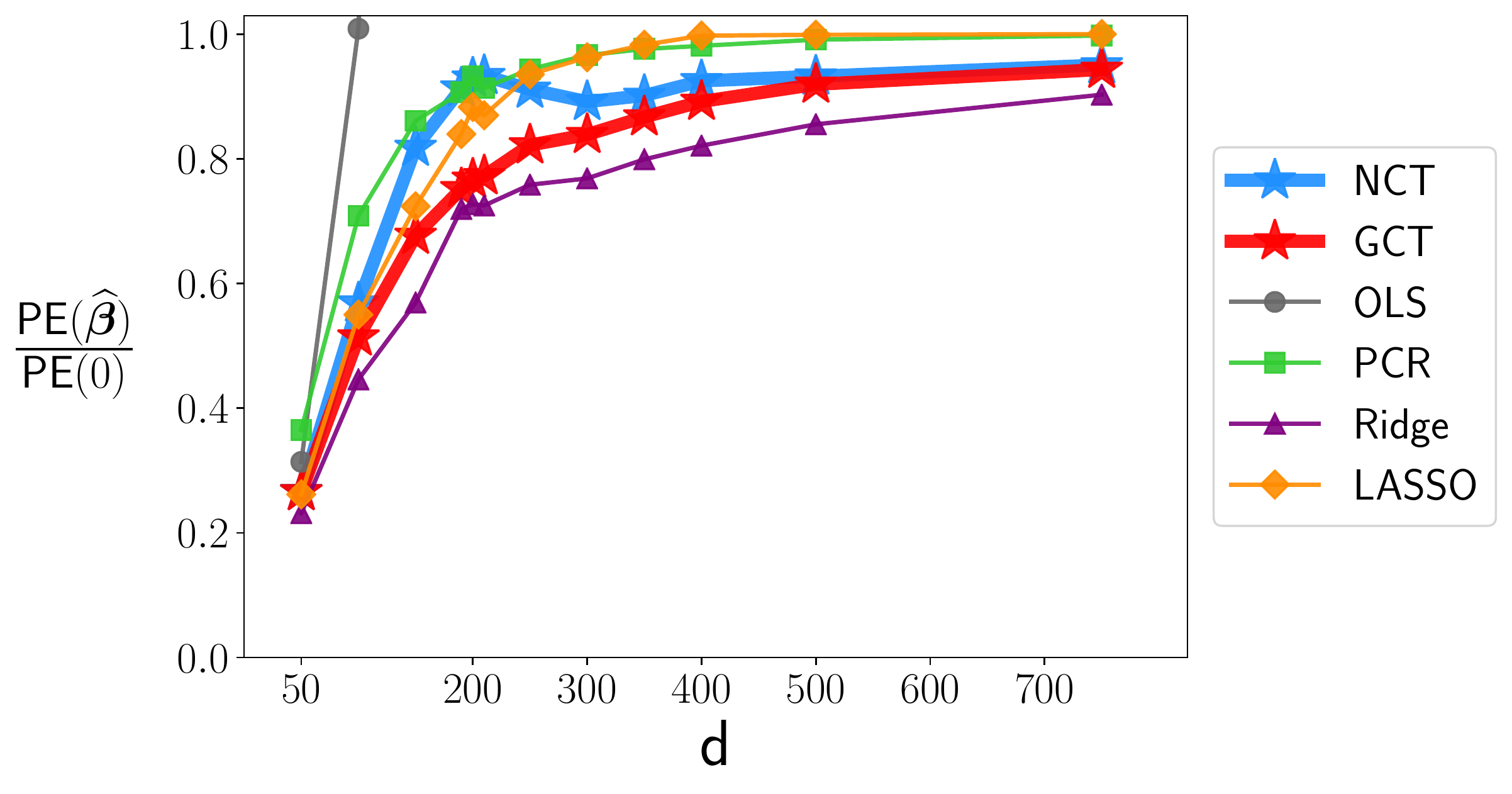}
         \vspace{-0.5cm}
         \caption{$a = 0.1$, $\,\bu_j^\T \bbeta = 1\,$ for 10 randomly chosen $\,j\in\{ d-25,\ldots, d\}$, and the rest components are i.i.d. $\,\mathcal{N}(0, d^{-1})$;}
     \end{subfigure}
     \vfill
     \vspace{1cm}
     \begin{subfigure}{\textwidth}
         \includegraphics[width=0.49\textwidth]{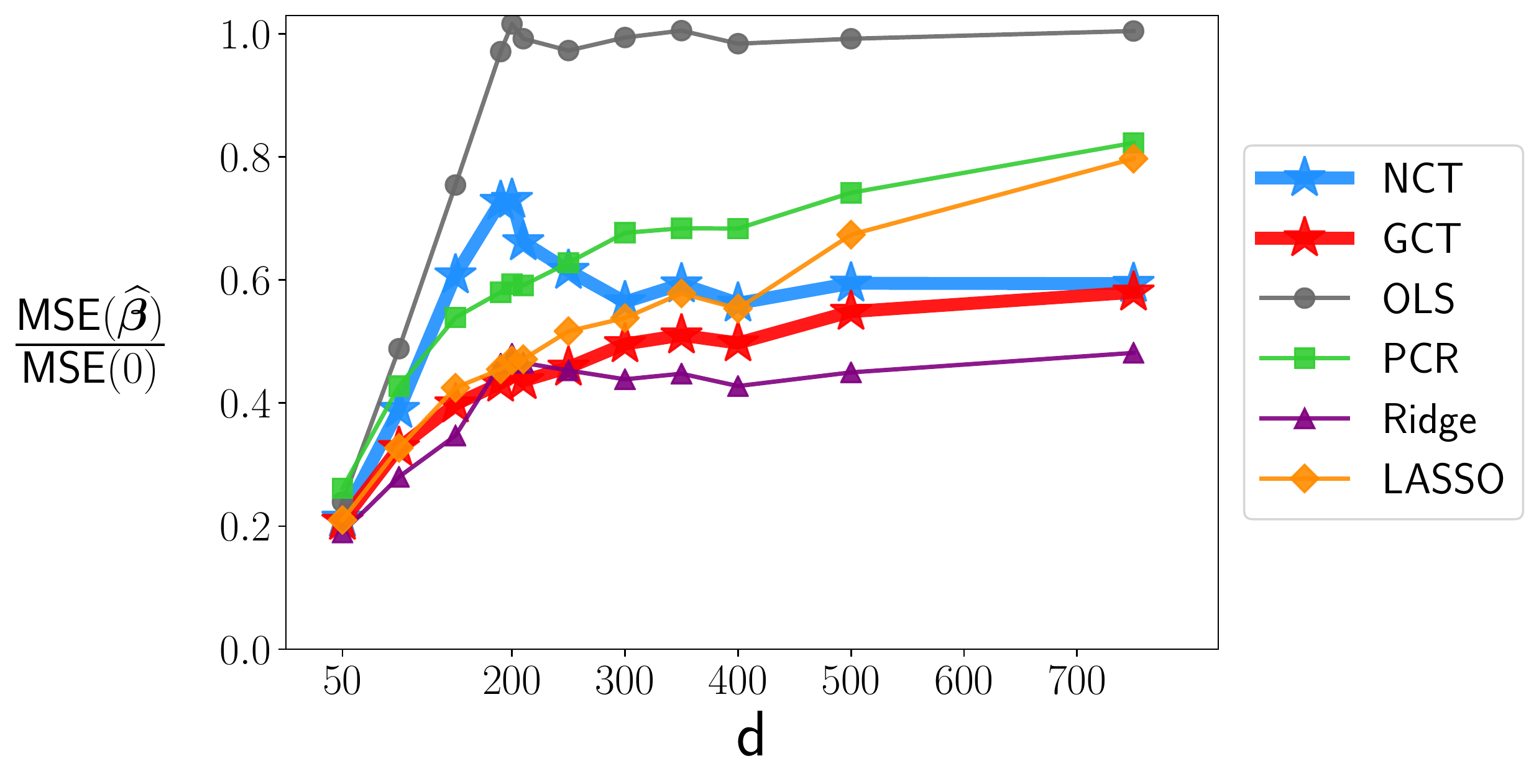}
         \hfill \hspace{0.3cm}
         \includegraphics[width=0.49\textwidth]{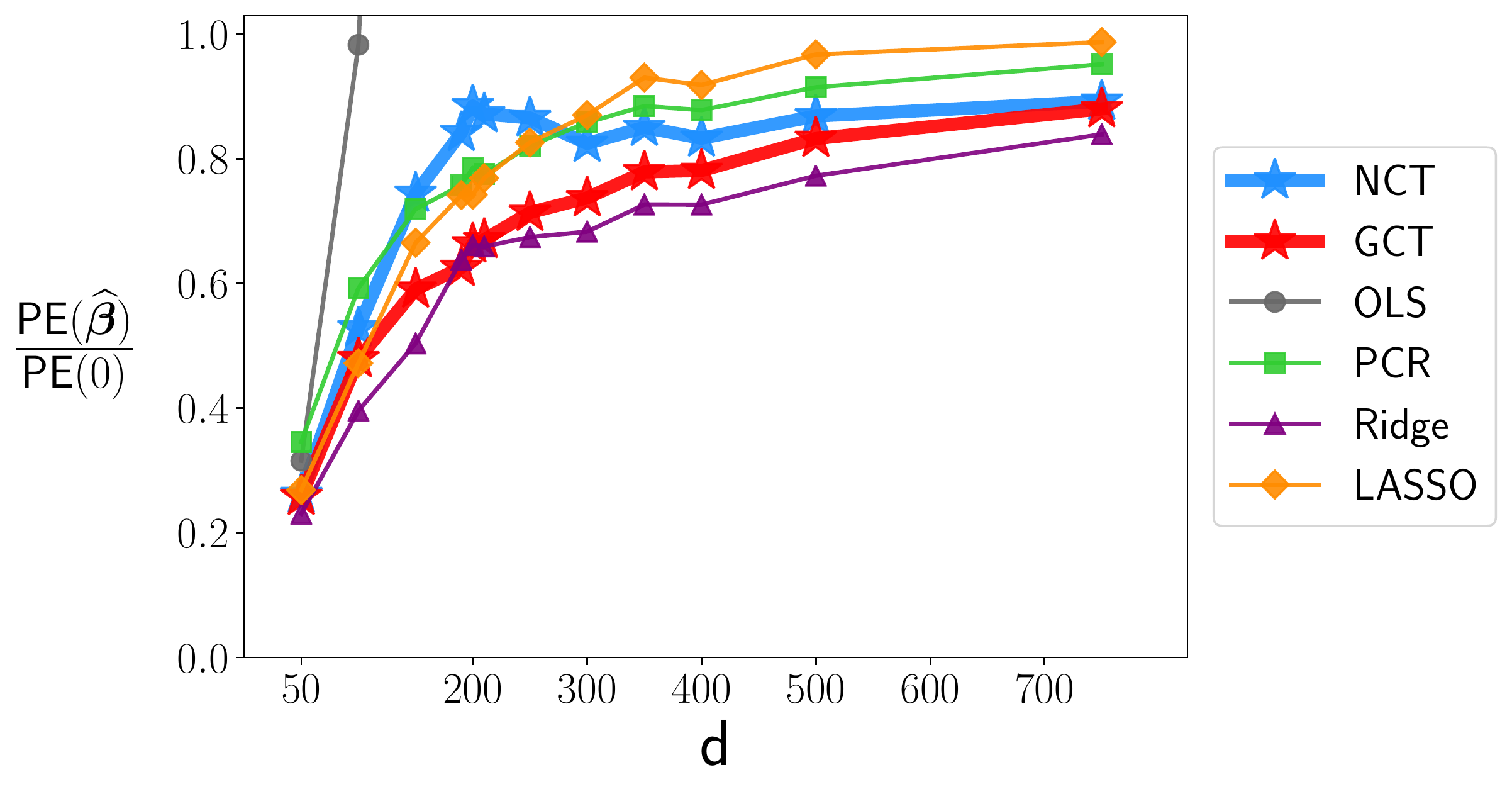}
         \vspace{-0.5cm}
         \caption{$a = 0.5$, $\,\bU^\T \bbeta \sim \mathcal{N}(0, \Id_d)$;}
     \end{subfigure}
        \caption{The relative errors $\,\mathsf{MSE}(\ebeta)/\mathsf{MSE}(0)\,$ (left) and $\,\mathsf{PE}(\ebeta)/\mathsf{PE}(0)\,$ (right) for different estimators with $\,n=200$, $\,\snr=1$. Polynomial decay of eigenvalues $\,\lambda_j = j^{-a}\,$ and different regimes of coefficients in eigenbasis $\,\bu_j^\T \bbeta$.}
        \label{plots4}
\end{figure}

We fix $\,n=200$, $\,\snr = 10\,$ or $\,\snr = 1$, and focus on how the relative errors $\,\mathsf{MSE}(\ebeta)/\mathsf{MSE}(0)\,$ and $\,\mathsf{PE}(\ebeta)/\mathsf{PE}(0)\,$ of these methods behave when the dimension $\,d\,$ grows.
The covariates $\,\bx_1,\ldots,\bx_n \sim \mathcal{N}(0, \St)$, where $\,\St\,$ depends on the eigenvalue scenario, and the noise vector $\,\beps \sim \mathcal{N}(0, \sigma^2\Id_n)\,$ (where $\sigma^2$ is chosen to ensure $\snr = 10$ or $\snr = 1$ for given $\,\St\,$ and $\,\bbeta$).
Without loss of generality we take $\,\St\,$ a diagonal matrix, or equivalently $\,\bU = \Id_d$.  The results are presented in Figure~\ref{plots1}--\ref{plots4}.
Figure~\ref{plots1} and Figure~\ref{plots2} correspond to $\,\snr = 10$, Figure~\ref{plots3} and Figure~\ref{plots4} correspond to $\,\snr = 1$.
Figure~\ref{plots1} and Figure~\ref{plots3} cover the scenarios of polynomial decay of the eigenvalues $\,\lambda_j = j^{-a}\,$ and  the coefficients $\,\bu_j^\T \bbeta = j^{-b}\,$ with
\begin{enumerate}[(a)]
	\item $a = 2$, $\,b = 2$;
	\item $a = 1$, $\,b = 0.5$;
	\item $a = 0.5$, $\,b = 1$;
\end{enumerate}
while Figure~\ref{plots2} and Figure~\ref{plots4} also consider polynomial decay of the eigenvalues $\,\lambda_j = j^{-a}\,$ but $\,\bU^\T \bbeta\,$ is different:
\begin{enumerate}[(a)]
	\item $a = 1$; $\,\bu_j^\T \bbeta = 1$ for $\,j=1,\ldots, 10\,$ and $0$ otherwise;
	\item $a = 0.1$; $\,\bu_j^\T \bbeta = 1\,$ for 10 randomly chosen $\,j\in\{ d-25,\ldots, d\}$, and the rest components are i.i.d. $\,\mathcal{N}(0, d^{-1})$;
	\item $a = 0.5$; $\,\bU^\T \bbeta \sim \mathcal{N}(0, \Id_d)$.
\end{enumerate}
 In each scenario, for each method and dimension we run the corresponding experiment 100 times and plot the median errors.

We notice that the NCT estimator (among some others) in some settings suffer around $\,d = n$. This is so called ``interpolation threshold'' -- when the dimension exceeds the number of data points, a model has enough features to interpolate training points. The behavior around this point and the associated ``double descent'' phenomenon has been an active area of research for the last couple of years. We do not focus on this in our work.

Otherwise, from the plots it is clear that in the presented settings the proposed procedure performs quite good compared to the other methods. In particular, the persistent performance of GCT suggests the benefit of varying thresholding to better adapt to various scenarios with different priors.   However, it is worth mentioning that other methods also perform quite unexpectedly well in a variety of settings, though previous theoretical results for them do not predict such performance. This may engender an interest in more thorough study of classical linear regression methods in high-dimensional setting under different structural assumptions.

\section{Main proofs} \label{S:mainproofs}

We start with the following lemma that allows to bound the properly scaled noise vector in $\ell_\infty$-norm. The lemma simultaneously deals with both fixed and random design settings.
\begin{lemma} \label{L:noise}
	Suppose Assumption~\ref{A:Noise} is fulfilled. Let $\,\rho\,$ be as in \eqref{def:rho}. Define the event
	\begin{equation}
	\begin{aligned}
		\Omega_1 \eqdef \left\{ \left\| \bxi \right\|_{\infty} \leq \frac{\sigma\rho}{2}\right\} \;\;\text{ with }\;\; \bxi \eqdef  \frac{\ZZ^\T\beps}{n} = \frac{\eL^{-1} \eU^\T \XX^\T \beps}{n} .
	\nonumber
	\end{aligned}
	\end{equation}
	Then
	\begin{equation}
	\begin{aligned}
		\Prob[\Omega_1] \geq 1-\delta.
	\nonumber
	\end{aligned}
	\end{equation}
\end{lemma}
\subsection{Proof of  Theorem~\ref{Th1}}
Using $\,\Se = \eU\eL^2\eU^\T\,$ we write for $\,\ebeta\,$ from~\eqref{estimator}
\begin{equation}
            \begin{aligned}
            	\mathsf{MSE}(\ebeta) = (\ebeta-\bbeta)^\T \Se (\ebeta-\bbeta) = \| \eL \eU^\T \ebeta - \eL\eU^\T \bbeta\|_2^2.
                \nonumber
            \end{aligned}
        \end{equation}
Now we plug our estimator $\,\ebeta\,$ and $\,\YY = \XX\bbeta + \beps\,$ in to get
\begin{equation}
            \begin{aligned}
            	\mathsf{MSE}(\ebeta) &=
            	\left\| \mathsf{SOFT}_\tau\left[ \eL^{-1}\eU^\T \frac{\XX^\T \YY}{n} \right] - \eL\eU^\T \bbeta\right\|_2^2\\
            	&= \left\| \mathsf{SOFT}_\tau\left[ \eL\eU^\T\bbeta +  \frac{\eL^{-1}\eU^\T\XX^\T \beps}{n} \right] - \eL\eU^\T \bbeta\right\|_2^2
            	= \left\| \mathsf{SOFT}_\tau\left[ \btheta +  \bxi \right] - \btheta\right\|_2^2,
                \nonumber
            \end{aligned}
        \end{equation}
 where we recall the canonical coefficients $\,\btheta = \eL \eU^\T \bbeta\,$ from Definition~\ref{canonical} and $\,\bxi\,$ from Lemma~\ref{L:noise}. From now on, the proof basically repeats the classical derivation for the soft and hard thresholding. Let us analyze its $j$-th component on $\,\Omega_1\,$ from Lemma~\ref{L:noise} of probability at least $\,1-\delta$.

\begin{itemize}
\item If $\,|\theta_j+\xi_j| > \tau$, then $\,|\theta_j| \geq \tau - |\xi_j| \geq \tau/2\,$ and
         \begin{equation}
            \begin{aligned}
            	|\mathsf{SOFT}_\tau\left[ \theta_j +  \xi_j \right] - \theta_j | = |\theta_j +\xi_j \pm \tau - \theta_j| = |\xi_j \pm \tau| \leq |\xi_j| + \tau \leq \frac{3\tau}{2} \leq 3\min\left( \tau, |\theta_j|\right),
                \nonumber
            \end{aligned}
        \end{equation}
where $\pm$ means that we take either $+$ or $-$ depending on the sign of $\,(\theta_j + \xi_j)$, but this doesn't play any role. For the lower bound,
\begin{equation}
            \begin{aligned}
            	|\mathsf{SOFT}_\tau\left[ \theta_j +  \xi_j \right] - \theta_j | = |\xi_j \pm \tau| \geq \tau - |\xi_j| \geq \tau/2 \geq \frac{1}{2}\min\left( \tau, |\theta_j|\right).
                \nonumber
            \end{aligned}
        \end{equation}
 \item If $\,|\theta_j+\xi_j| \leq \tau$, then $\,|\theta_j| \leq \tau + |\xi_j| \leq 3\tau/2\,$ and
         \begin{equation}
            \begin{aligned}
            	|\mathsf{SOFT}_\tau\left[ \theta_j +  \xi_j \right] - \theta_j | = |0 -\theta_j| = |\theta_j| \leq 3\min\left( \tau, |\theta_j|\right).
                \nonumber
            \end{aligned}
        \end{equation}
        For the lower bound,
        \begin{equation}
            \begin{aligned}
            	|\mathsf{SOFT}_\tau\left[ \theta_j +  \xi_j \right] - \theta_j | = |\theta_j| \geq \frac{1}{2}\min\left( \tau, |\theta_j|\right).
                \nonumber
            \end{aligned}
        \end{equation}
       \end{itemize}
       Hence, on $\,\Omega_1$
       \begin{equation}
            \begin{aligned}
            	\frac{1}{4}\sum\limits_{j=1}^r \min\left( \tau, |\theta_j|\right)^2 \leq \left\| \mathsf{SOFT}_\tau\left[ \btheta +  \bxi \right] - \btheta\right\|_2^2
            	\leq 9\sum\limits_{j=1}^r \min\left( \tau, |\theta_j|\right)^2.
                \nonumber
            \end{aligned}
        \end{equation}
Continuing the upper bound, note that for any $\,0 \leq q \leq 2\,$ we have $\,\min\left( \tau, |\theta_j|\right) \leq \tau^{1-q/2}\,|\theta_j|^{q/2}\,$ (we use convention $\,0^0 = 0$).
Thus, 
\begin{equation}
            \begin{aligned}
            	\sum\limits_{j=1}^r \min\left( \tau, |\theta_j|\right)^2
            	\leq \tau^{2-q} \sum\limits_{j=1}^r |\theta_j|^q =
            	\tau^{2-q} \| \btheta \|_q^q,
                \nonumber
            \end{aligned}
        \end{equation}
using convention $\,\| \cdot\|_0^0 = \| \cdot\|_0$. Taking infimum over $\,q\in [0, 2]$, extracting $\,\bbeta^\T\Se\bbeta = \mathsf{MSE}(0)\,$ and recalling the definitions of $\,\snr\,$ and $\,\eff_{q,d}(\Se,\bbeta)$,  we conclude the proof.

\subsection{Proof of Theorem~\ref{Th2}}
To begin with, we state the following well-known result on the concentration of the sample covariance around the true covariance in terms of the effective rank.
See \cite{Koltchinskii_CIAMBFSCO}, Theorem 9; also, \cite{Vershynin}, Theorem 9.2.4 and Exercise 9.2.5.
\begin{lemma} \label{L:covconc}
Suppose Assumption~\ref{A:X} is fulfilled. Then, with probability $1-\delta$
	\begin{equation}
            \begin{aligned}
            	\| \Se - \St \| \leq C\| \St\| \left( \sqrt{\frac{\reff[\St] + \log(1/\delta)}{n}} + \frac{\reff[\St] + \log(1/\delta)}{n} \right).
                \nonumber
            \end{aligned}
        \end{equation}
\end{lemma}
\noindent Using Assumption~\ref{A:tech} we can leave only the first term in the bound above. Let $\Omega_2$ be the event on which this bound holds.

Our main tools to prove the main result is the beautiful work by \cite{Wahl} that develops tight \textit{relative perturbation bounds for eigenvalues and eigenvectors} of covariance matrix. Let us describe the framework of that paper. By Assumption~\ref{A:cvxdecay} we consider the case of simple eigenvalues of $\St$. The following quantities play important role: the \textit{relative rank}
	\begin{equation}
            \begin{aligned}
            	\br_j(\St) \eqdef \sum\limits_{\substack{l=1\\l\neq j}}^d \frac{\lambda_l}{|\lambda_j - \lambda_l|} + \frac{\lambda_j}{\min(\lambda_{j-1} - \lambda_j, \lambda_j - \lambda_{j+1})} \;\;\;\text{ for }j\in[d]
                \nonumber
            \end{aligned}
        \end{equation}
(here $\lambda_0 = +\infty$ and $\lambda_{d+1} = 0$ for convenience) and the entries of $\St^{-1/2} (\Se-\St)\St^{-1/2}$
\begin{equation}
            \begin{aligned}
            	\overline{\eta}_{ll^\prime} \eqdef \frac{\bu_l^\T (\Se-\St) \bu_{l^\prime}}{\sqrt{\lambda_l\lambda_{l^\prime}}}\;\;\;\text{ for }l,l^\prime\in[d].
                \nonumber
            \end{aligned}
        \end{equation}
Relative perturbation bounds for $j$-th eigenvalue and eigenvector hold under the condition that there exist $x$ such that
 \begin{equation}
            \begin{aligned}
            	|\overline{\eta}_{ll^\prime}| \leq x\;\;\;\text{ for all }l,l^\prime\in[d], \;\;\text{ and }\;\;\br_j(\St) \leq \frac{1}{3x}.
                \nonumber
            \end{aligned}
        \end{equation}
The following lemma helps to control the first condition.
\begin{lemma} \label{L:psi_n}
	Suppose Assumption~\ref{A:X} and Assumption~\ref{A:tech} hold. Then, with probability $1-\delta$
	 \begin{equation}
            \begin{aligned}
            	\max\limits_{l,l^\prime\in[d]} |\overline{\eta}_{ll^\prime}| \leq \epsilon,
                \nonumber
            \end{aligned}
        \end{equation}
        where
         \begin{equation}
            \begin{aligned}
            	\epsilon = \epsilon_{n,d,\delta} \eqdef C \sqrt{\frac{\log(d/\delta)}{n}}. 
                \nonumber
            \end{aligned}
        \end{equation}
        for some $C$.
\end{lemma}
\noindent Define $\Omega_3$ to be the event where the inequality from the previous lemma holds.

So, the relative perturbation bounds hold true on $\Omega_3$ for indices $j$ for which $\br_j(\St) \leq 1/(3\epsilon)$. We would like to have this property for as many indices as possible.
Under Assumption~\ref{A:cvxdecay} we have (see \cite{Wahl}, inequalities (3.30); \cite{Jirak}, Lemma 7.13; \cite{Cardot}, Lemma 6.1)
	\begin{equation}
            \begin{aligned}
            	\br_j(\St) \leq 1+ 2 \sum\limits_{\substack{l=1\\l\neq j}}^d \frac{\lambda_l}{|\lambda_j - \lambda_l|} \leq 1 + 2C j \log(j).
                \nonumber
            \end{aligned}
        \end{equation}
Note that with
	\begin{equation}
            \begin{aligned}
            	 k^* \eqdef (\epsilon \log(1/\epsilon))^{-2/3}
                \nonumber
            \end{aligned}
        \end{equation}
 we indeed have $\br_j(\St) \leq 1/(3\epsilon)$ for all $j\in[k^*]$ due to Assumption~\ref{A:tech}. Hence, the following relative perturbation bounds from \cite{Wahl} hold true.
\begin{lemma} \label{L:relbounds}
	For all $j\in[k^*]$ on $\Omega_3$ holds
	\begin{equation}
            \begin{aligned}
            	 |\widehat{\lambda}_j - \lambda_j| \leq C\epsilon\lambda_j \;\;\;\text{ and }\;\;\;
            	 \| \widehat{\bu}_j - \bu_j\|_2 \leq C\epsilon\sqrt{\sum\limits_{\substack{l=1\\l\neq j}}^d \frac{\lambda_j\lambda_l}{(\lambda_j-\lambda_l)^2}}\,.
                \nonumber
            \end{aligned}
        \end{equation}
Furthermore,
	for all $j\in[k^*]$ and $l\in[d], l\neq j$ on $\Omega_3$ holds
	\begin{equation}
            \begin{aligned}
            	 |\widehat{\bu}_j^\T \bu_l| \leq C\epsilon \frac{\sqrt{\lambda_j\lambda_l}}{|\lambda_j-\lambda_l|}\,.
                \nonumber
            \end{aligned}
        \end{equation}
\end{lemma}
\noindent Note that the bounds from the previous lemma apply even for larger indices $j$, which can be up $(\epsilon\log(1/\epsilon))^{-1}$ (of order $n^{1/2}$), while we restrict $k^*$ to be of order $n^{1/3}$. Later in the proof it will be clear how this specific $k^*$ arises.

Now we are ready to proceed to the main part of the proof.
\begin{proof}[Proof of Theorem~\ref{Th2}]
	We prove the theorem for $k = k^*$, and it will be clear that the same proof works with any $k < k^*$.
	The proof for part (i) and part (ii) coincides up to the last step. Denote $\tau^\prime$ to be the general thresholding level, which is $\tau$ for part (i) and $\overline{\tau}$ for part(ii).
	Using the definition of $\ebeta$ given in \eqref{estimator}, the eigendecompositions $\St  = \bU \bLambda^2 \bU^\T$, $\Se  = \eU \eL^2 \eU^\T$ and the model $\YY = \XX\bbeta+\beps$, write the prediction error as
	\begin{equation}
            \begin{aligned}
            	\mathsf{PE}(\ebeta) &= (\ebeta-\bbeta)^\T \St (\ebeta-\bbeta)
            	\\&= \| \bLambda \bU^\T \eU\eL^{-1} \mathsf{SOFT}_{\tau^\prime}\left[\eL^{-1}\eU^\T\frac{\XX^\T\YY}{n} \right] - \bLambda\bU^\T\bbeta \|_2^2
            	\\&= \| \bLambda \bU^\T \eU\eL^{-1} \mathsf{SOFT}_{\tau^\prime}[\eL\eU^\T\bbeta + \bxi] - \bLambda\bU^\T\bbeta \|_2^2
                \nonumber
            \end{aligned}
        \end{equation}
        with $\bxi$ from Lemma~\ref{L:noise}.
 Let us add and subtract $\bLambda \bU^\T \eU\eU^\T \bbeta$ inside the norm and apply $\|\ba+\bbb\|_2^2 \leq 2\|\ba\|_2^2 + 2\|\bbb\|_2^2$:
 \begin{equation}
            \begin{aligned}
            	\mathsf{PE}(\ebeta) &\leq
            	2\|  \bLambda\bU^\T\bbeta - \bLambda\bU^\T\eU\eU^\T\bbeta\|_2^2
            	+
            	\\&\qquad+ 2\| \bLambda \bU^\T \eU\eL^{-1} \mathsf{SOFT}_{\tau^\prime}\left[\eL\eU^\T\bbeta+\bxi \right] - \bLambda\bU^\T\eU\eU^\T\bbeta \|_2^2  =: I_1+I_2.
                \nonumber
            \end{aligned}
        \end{equation}
We first deal with $I_1$:
 \begin{equation}
            \begin{aligned}
            	\frac{I_1}{2} &=
            	\|  \bLambda\bU^\T\bbeta - \bLambda\bU^\T\eU\eU^\T\bbeta\|_2^2
            	= \bbeta^\T (\Id_d - \eU\eU^\T) \,\St\,(\Id_d - \eU\eU^\T) \bbeta
            	\\&=\bbeta^\T (\Id_d - \eU\eU^\T) \,(\St - \Se)\,(\Id_d - \eU\eU^\T) \bbeta
            	\leq \| \Se - \St\|\,\|(\Id_d - \eU\eU^\T) \bbeta\|_2^2
            	\\&\leq \| \Se - \St\|\,\|\bbeta\|_2^2
            	\leq C\|\St\|\|\bbeta\|_2^2 \sqrt{\frac{\reff[\St] + \log(1/\delta)}{n}} \,,
            	\nonumber
            \end{aligned}
        \end{equation}
        where  the last inequality holds on $\Omega_2$ due to Lemma~\ref{L:covconc} and Assumption~\ref{A:tech}.

Next, we focus on $I_2$. We will decompose it into two parts: one will correspond to the first $k^*$ eigenvectors and eigenvalues, while the other will correspond to the rest $(r-k^*)$. Let us split
 \begin{equation}
            \begin{aligned}
            	\eU = [ \eU_{\leq k^*} \eU_{>k^*}] \;\;\;\text{ and }\;\;\;\eL = \begin{bmatrix} \eL_{\leq k^*} & \Oo_{k^*\times (r-k^*)} \\ \Oo_{(r-k^*)\times k^*} & \eL_{>k^*} \end{bmatrix},
            	\nonumber
            \end{aligned}
        \end{equation}
        where $\eL_{\leq k^*} \in \R^{k^*\times k^*}$, $\eU_{\leq k^*} \in\R^{d\times k^*}$ correspond to the first $k^*$ eigenvalues and eigenvectors, while $\eL_{> k^*} \in \R^{(r-k^*)\times (r-k^*)}$, $\eU_{>k^*} \in\R^{d\times (r-k^*)}$ correspond to the rest. Also let $\bxi = \left[\bxi_{\leq k^*}^\T \;\bxi_{>k^*}^\T\right]^\T$ with $\bxi_{\leq k^*} \in \R^{k^*}$ and $\bxi_{> k^*} \in \R^{r-k^*}$. Then
 \begin{equation}
            \begin{aligned}
            	 &\eU\eL^{-1} \mathsf{SOFT}_{\tau^\prime}\left[\eL\eU^\T\bbeta + \bxi\right] -\eU\eU^\T\bbeta
            	 =
            	 \\&\qquad =\eU_{\leq k^*}\eL_{\leq k^*}^{-1} \mathsf{SOFT}_{\tau^\prime}\left[\eL_{\leq k^*}\eU_{\leq k^*}^\T\bbeta + \bxi_{\leq k^*}\right] -\eU_{\leq k^*}\eU_{\leq k^*}^\T\bbeta
            	 \\&\qquad\qquad+
            	 \eU_{>k^*}\eL_{>k^*}^{-1} \mathsf{SOFT}_{\tau^\prime}\left[\eL_{>k^*}\eU_{>k^*}^\T\bbeta + \bxi_{>k^*}\right] -\eU_{>k^*}\eU_{>k^*}^\T\bbeta.
                \nonumber
            \end{aligned}
        \end{equation}
Again applying $\|\ba+\bbb\|_2^2 \leq 2\|\ba\|_2^2 + 2\|\bbb\|_2^2$ we obtain
 \begin{equation}
            \begin{aligned}
            	 &\frac{I_2}{2} = \| \bLambda\bU^\T \eU\eL^{-1} \mathsf{SOFT}_{\tau^\prime}\left[\eL\eU^\T\bbeta + \bxi\right] - \bLambda\bU^\T \eU\eU^\T\bbeta \|_2^2
            	 \\&\qquad \leq 2\| \bLambda\bU^\T \eU_{\leq k^*}\eL_{\leq k^*}^{-1} \mathsf{SOFT}_{\tau^\prime}\left[\eL_{\leq k^*}\eU_{\leq k^*}^\T\bbeta + \bxi_{\leq k^*}\right] - \bLambda\bU^\T \eU_{\leq k^*}\eU_{\leq k^*}^\T\bbeta\|_2^2
            	 \\&\qquad\qquad+
            	 2\| \bLambda\bU^\T \eU_{>k^*}\eL_{>k^*}^{-1} \mathsf{SOFT}_{\tau^\prime}\left[\eL_{>k^*}\eU_{>k^*}^\T\bbeta + \bxi_{>k^*}\right] - \bLambda\bU^\T \eU_{>k^*}\eU_{>k^*}^\T\bbeta \|_2^2 
            	 \\&\qquad=: I_3 + I_4.
                \nonumber
            \end{aligned}
        \end{equation}
        So, to upper bound $I_2$ we will upper bound $I_3$ and $I_4$ separately.

        Consider $I_4$.
        Denote
        \begin{equation}
        	\begin{aligned}
            	 \bgamma \eqdef \eL_{>k^*}^{-1} \mathsf{SOFT}_{\tau^\prime} \left[ \eL_{>k^*} \eU_{>k^*}^\T\bbeta + \bxi_{>k^*} \right] - \eU_{>k^*}^\T \bbeta \in \R^{r-k^*}.
                \nonumber
            \end{aligned}
        \end{equation}
        Let us analyze $j$-th component $\gamma_j$, for $j\in[r-k^*]$, using the definition of $\mathsf{SOFT}_{\tau^\prime}[\,\cdot\,]$. We have two cases:
        \begin{itemize}
        	\item If $|\widehat{\lambda}_{j+k^*}^{1/2} \widehat{\bu}_{j+k^*}^\T \bbeta + \xi_{j+k^*}| > \tau^\prime$, then $\gamma_j = \widehat{\lambda}_{j+k^*}^{-1/2} (\xi_{j+k^*} \pm \tau^\prime)$ (the actual sign will play no role).
        	Since by Lemma~\ref{L:noise} on $\Omega_1$
        	\begin{equation}
        	\begin{aligned}
        	\widehat{\lambda}_{j+k^*}^{1/2} |\widehat{\bu}_{j+k^*}^\T \bbeta| \geq \tau^\prime - |\xi_{j+k^*}| \geq \tau^\prime/2,\nonumber
            \end{aligned}
        \end{equation}
         we have
        	\begin{equation}
        	\begin{aligned}
            	 |\gamma_j| \leq \widehat{\lambda}_{j+k^*}^{-1/2} (|\xi_{j+k^*}| + \tau^\prime) \leq \widehat{\lambda}_{j+k^*}^{-1/2} \cdot \frac{3\tau^\prime}{2}\leq 3|\widehat{\bu}_{j+k^*}^\T \bbeta|.
                \nonumber
            \end{aligned}
        \end{equation}
        	\item If $|\widehat{\lambda}_{j+k^*}^{1/2} \widehat{\bu}_{j+k^*}^\T \bbeta + \xi_{j+k^*}| \leq \tau^\prime$, then $\gamma_j = -\widehat{\bu}_{j+k^*}^\T\bbeta$, and we directly get
        	\begin{equation}
        	\begin{aligned}
            	 |\gamma_j| = |\widehat{\bu}_{j+k^*}^\T \bbeta|.
                \nonumber
            \end{aligned}
        \end{equation}
        \end{itemize}
        In any case, $ |\gamma_j| \leq 3|\widehat{\bu}_{j+k^*}^\T \bbeta|$ for all $j\in[r-k^*]$, and therefore $\|\bgamma\|_2^2 \leq 9\|\bbeta\|_2^2$ on $\Omega_1$.
        Hence,
         \begin{equation}
        	\begin{aligned}
            	 \frac{I_4}{2} &= \| \bLambda\bU^\T \eU_{>k^*} \bgamma\|_2^2 = \bgamma^\T \eU_{>k^*}^\T\St\eU_{>k^*} \bgamma
            	 =\bgamma^\T \eU_{>k^*}^\T (\St-\Se)  \eU_{>k^*} \bgamma
            	 + \bgamma^\T \eU_{>k^*}^\T \Se  \eU_{>k^*} \bgamma
            	 \\&\leq \| \Se-\St\| \|\eU_{>k^*}\bgamma\|_2^2 + \bgamma^\T\eL_{>k^*}^2   \bgamma
            	 \leq \| \Se-\St\| \|\bgamma\|_2^2 + \widehat{\lambda}_{k^*+1}\|  \bgamma\|_2^2.
                \nonumber
            \end{aligned}
        \end{equation}
        We bound the first term on the right-hand side on $\Omega_2$ by Lemma~\ref{L:covconc}, and for the second term on $\Omega_3$ holds $\widehat{\lambda}_{k^*+1} \leq \widehat{\lambda}_{k^*} \leq (1+C\epsilon)\lambda_{k^*} \leq C^\prime\lambda_{k^*} $ due to Lemma~\ref{L:relbounds} and Assumption~\ref{A:tech}. Taking into account $\| \bgamma\|_2 \leq C\|\bbeta\|_2$ on $\Omega_1$, we get on $\Omega_1 \cap \Omega_2 \cap \Omega_3$
        \begin{equation}
        	\begin{aligned}
            	 I_4&\leq C \|\St\|\|\bbeta\|_2^2\left(\sqrt{\frac{\reff[\St] + \log(1/\delta)}{n}} + \frac{\lambda_{k^*}}{\lambda_1} \right).
                \nonumber
            \end{aligned}
        \end{equation}

        Finally, it is left to bound $I_3$.
        Denote
        \begin{equation}
        	\begin{aligned}
            	 \widehat{\bomega} \eqdef  \mathsf{SOFT}_{\tau^\prime} \left[ \eL_{\leq k^*} \eU_{\leq k^*}^\T\bbeta + \bxi_{\leq k^*} \right] - \eL_{\leq k^*} \eU_{\leq k^*}^\T \bbeta \in \R^{k^*}.
                \nonumber
            \end{aligned}
        \end{equation}
        Then,
        \begin{equation}
        	\begin{aligned}
            	 \frac{I_3}{2} &= \| \bLambda\bU^\T \eU_{\leq k^*} \eL_{\leq k^*}^{-1}\widehat{\bomega}\|_2^2
            	 \leq \| \bLambda\bU^\T \eU_{\leq k^*} \eL_{\leq k^*}^{-1}\|^2 \| \widehat{\bomega}\|_2^2.
                \nonumber
            \end{aligned}
        \end{equation}
         An upper bound on $\| \bLambda\bU^\T \eU_{\leq k^*} \eL_{\leq k^*}^{-1}\|$ is provided in the next lemma.
        \begin{lemma} \label{L:opnorm}
        	Suppose Assumption~\ref{A:X} holds. Then on $\,\Omega_3\,$ holds
        	\begin{equation}
        	\begin{aligned}
            	 \| \bLambda\bU^\T \eU_{\leq k^*} \eL_{\leq k^*}^{-1} \| \leq C.
                \nonumber
            \end{aligned}
        \end{equation}
        \end{lemma}
        \begin{remark}
        	The previous lemma is the only place where we use $k^* = (\epsilon\log(1/\epsilon))^{-2/3}$. The rest of the proof would go through if $k^*$ was defined as $(\epsilon\log(1/\epsilon))^{-1}$.
        \end{remark}
        \begin{remark}
        	 Interestingly, a closely related to $\bLambda\bU^\T \eU_{\leq k^*} \eL_{\leq k^*}^{-1} $ matrix appears also in \cite{Bartlett}. The main difficulty of their proof is to find regimes of eigenvalues such that for $\bC$ defined as
        	 \begin{align}
        	 	\bC \eqdef (\XX\XX^\T)^{-1} \XX\St\XX^\T(\XX\XX^\T)^{-1}
        	 	\nonumber
        	 \end{align}
        	 holds $\Tr[\bC] = o(1)$ as $n\to\infty$. Using SVD $n^{-1/2} \XX = \eV \eL \eU^\T$ one can show
        	 \begin{align}
        	 	\Tr[\bC] = \frac{1}{n} \Tr[\eL^{-1} \eU^\T \bU\bLambda^2\bU^\T \eU\eL^{-1}],
        	 \nonumber
        	 \end{align}
        	while in Lemma~\ref{L:opnorm} we essentially upper bound the operator norm of somewhat simpler (in a sense that we truncate the sample eigenvalues and eigenvector beyond $k^*$-th) matrix \\$\eL_{\leq k^*}^{-1} \eU_{\leq k^*}^\T \bU\bLambda^2\bU^\T \eU_{\leq k^*}\eL_{\leq k^*}^{-1}$. The latter task turns out to be much easier and does not require specific regimes of eigenvalues, unlike the former one.
        \end{remark}
        To deal with $\| \widehat{\bomega} \|_2^2$, we first state the following lemma which has two parts, one of which will help to conclude the proof of claim (i), and the other one will be useful for claim (ii).
        \begin{lemma} \label{L:eigen}
         On $\Omega_3$ holds
         \begin{enumerate}[(i)]
         \item
         \begin{equation}
        	\begin{aligned}
            	 \| \eL_{\leq k^*} \eU_{\leq k^*}^\T \bbeta - \bLambda_{\leq k^*} \bU_{\leq k^*}^\T \bbeta \|_2^2 \leq C \|\St\| \|\bbeta\|_2^2 \,\epsilon\left( \reff[\St] + \epsilon\sum\limits_{j=1}^{k^*}  \frac{\lambda_j\,j^2}{\lambda_1} \right).
                \nonumber
            \end{aligned}
        \end{equation}
         \item
         \begin{equation}
        	\begin{aligned}
            	 \| \eL_{\leq k^*} \eU_{\leq k^*}^\T \bbeta - \bLambda_{\leq k^*} \bU_{\leq k^*}^\T \bbeta \|_\infty \leq \frac{\tau^\infty}{2} \eqdef C \|\St\|^{1/2} \|\bbeta\|_2 \,\epsilon^{1/2}\max\limits_{j\in[k^*]} \left( \frac{\lambda_j\,(1+\epsilon j^2)}{\lambda_1} \right)^{1/2}.
                \nonumber
            \end{aligned}
        \end{equation}
        \end{enumerate}
        \end{lemma}
        \noindent
        \noindent So, to deal with part (ii) of Theorem~\ref{Th2}, we notice that our thresholding level $\tau^\prime = \overline{\tau} = \tau + \tau^\infty$, where $\tau$ from Lemma~\ref{L:noise} is responsible for the noise and $\tau^\infty$ from Lemma~\ref{L:eigen} is responsible for the estimation of the eigenvalues and eigenvectors. Now we analyze the components $\widehat{\omega}_j$ of $\widehat{\bomega}$ for $j\in[k^*]$.
        \begin{itemize}
        	\item If $|\widehat{\lambda}_j^{1/2} \widehat{\bu}_j^\T \bbeta + \xi_j| > \overline{\tau}$, then $\widehat{\omega}_j = \xi_j \pm \overline{\tau}$ (the sign again doesn't matter). Moreover, due to Lemma~\ref{L:noise} and Lemma~\ref{L:eigen} (ii) on $\Omega_1\cap\Omega_3$
        	 \begin{equation}
        	\begin{aligned}
            	 |\lambda_j^{1/2}\bu_j^\T \bbeta| \geq \overline{\tau} - |\xi_j| - |\widehat{\lambda}_j^{1/2}\widehat{\bu}_j^\T \bbeta-\lambda_j^{1/2}\bu_j^\T \bbeta| \geq \overline{\tau} - \frac{\tau}{2} - \frac{\tau^\infty}{2} = \frac{\overline{\tau}}{2}.
                \nonumber
            \end{aligned}
        \end{equation}
        Hence,
        \begin{equation}
        	\begin{aligned}
            	 |\widehat{\omega}_j| \leq 3\min(\overline{\tau}, |\lambda_j^{1/2}\bu_j^\T \bbeta|).
                \nonumber
            \end{aligned}
        \end{equation}
        	\item If $|\widehat{\lambda}_j^{1/2} \widehat{\bu}_j^\T \bbeta + \xi_j| \leq \overline{\tau}$, then $\widehat{\omega}_j = -\widehat{\lambda}_j^{1/2} \widehat{\bu}_j^\T \bbeta$. Furthermore, again by Lemma~\ref{L:noise} and Lemma~\ref{L:eigen} (ii) on $\Omega_1\cap\Omega_3$
        	 \begin{equation}
        	\begin{aligned}
            	 |\lambda_j^{1/2}\bu_j^\T \bbeta| \leq \overline{\tau} + |\xi_j| + |\widehat{\lambda}_j^{1/2}\widehat{\bu}_j^\T \bbeta-\lambda_j^{1/2}\bu_j^\T \bbeta| \leq \overline{\tau} + \frac{\tau}{2} + \frac{\tau^\infty}{2} = \frac{3\overline{\tau}}{2}.
                \nonumber
            \end{aligned}
        \end{equation}
        Thus,
        \begin{equation}
        	\begin{aligned}
            	 |\widehat{\omega}_j| \leq 3\min(\overline{\tau}, |\lambda_j^{1/2}\bu_j^\T \bbeta|).
                \nonumber
            \end{aligned}
        \end{equation}
        \end{itemize}
        In both cases, $|\widehat{\omega}_j| \leq 3\min(\overline{\tau}, |\lambda_j^{1/2}\bu_j^\T \bbeta|)$, and based on the same derivation as in the proof of Theorem~\ref{Th1}, we obtain on $\Omega_1\cap\Omega_3$
        \begin{equation}
        	\begin{aligned}
            	 \|\widehat{\bomega}\|_2^2 \leq 9\inf\limits_{q\in[0, 2]} \left\{ \overline{\tau}^{2-q} \|\bLambda_{\leq k^*}\bU_{\leq k^*}^\T \bbeta\|_q^q \right\}.
                \nonumber
            \end{aligned}
        \end{equation}
       For part (i) we act slightly differently. Now $\tau^\prime = \tau$. We decompose
       \begin{align}
       	\widehat{\bomega} = \bomega + \Delta\bomega,
       	\nonumber
       \end{align}
       where\begin{equation}
        	\begin{aligned}
            	 \bomega \eqdef  \mathsf{SOFT}_{\tau} \left[ \bLambda_{\leq k^*} \bU_{\leq k^*}^\T\bbeta + \bxi_{\leq k^*} \right] - \bLambda_{\leq k^*} \bU_{\leq k^*}^\T \bbeta \in \R^{k^*}.
                \nonumber
            \end{aligned}
        \end{equation}
        By Lemma~\ref{L:eigen} (i) it is easy to bound on $\Omega_3$
        \begin{equation}
        	\begin{aligned}
            	 \| \Delta\bomega\|_2^2 = \| \widehat{\bomega} - \bomega\|_2^2\leq C \|\St\| \|\bbeta\|_2^2 \,\epsilon\left( \reff[\St] + \epsilon\sum\limits_{j=1}^{k^*}  \frac{\lambda_j\,j^2}{\lambda_1} \right).
                \nonumber
            \end{aligned}
        \end{equation}
        The norm $\| \bomega \|_2^2$ with thresholding at level $\tau$ can be bounded as in the proof of Theorem~\ref{Th1}: on $\Omega_1$
        \begin{equation}
        	\begin{aligned}
            	 \|\bomega\|_2^2 \leq 9\inf\limits_{q\in[0, 2]} \left\{ \tau^{2-q} \|\bLambda_{\leq k^*}\bU_{\leq k^*}^\T \bbeta\|_q^q \right\},
                \nonumber
            \end{aligned}
        \end{equation}
        which, together with the bound on $\|\Delta\bomega\|_2^2$, gives bound on $\|\widehat{\bomega}\|_2^2$ on $\Omega_1 \cap \Omega_3$.

        Putting all the bounds for $I_1$ and $I_2$ (in particular, for $I_3$ and $I_4$) together on the intersection of high probability events $\,\Omega_1\cap\Omega_2\cap \Omega_3$, adjusting $\delta \rightarrow \delta/3$ so that the intersection has probability at least $1-\delta$, we conclude the proof.
\end{proof}

\subsection{Proof of Theorem~\ref{Th:minimax}}
We first reduce the general linear regression model to the Gaussian sequence model, and then apply some classical results from the literature. Our original linear regression problem (restricted to the Gaussian noise case)
\begin{align}
	\YY = \XX\bbeta + \beps,\;\;\;\beps \sim \mathcal{N}(0, \sigma^2\Id_r)
\nonumber
\end{align}
can be rewritten in the canonical form
\begin{align}
	\YY = \sqrt{n} \eV \btheta + \beps,\;\;\;\beps \sim \mathcal{N}(0, \sigma^2\Id_r)
\nonumber
\end{align}
with $\btheta = \eL \eU^\T \bbeta \in \R^r$, and then as the Gaussian sequence model
\begin{align}
	\frac{\eV^\T\YY}{\sqrt{n}} = \btheta + \beps,\;\;\;\beps \sim \mathcal{N}(0, \sigma^2\Id_r).
\nonumber
\end{align}
Recall that the joint effective dimension and the signal-to-noise ratio can be expressed in terms of the canonical parameter as $\,\eff_{q,r}(\Se, \bbeta) = \|\btheta\|_q^q/\|\btheta\|_2^q\,$ and $\,\snr = \| \btheta\|_2/\sigma$. 
Hence, the parameter space $\,\mathcal{P}^{\XX}(q, \mathsf{D}, \mathsf{S})$ for the initial model translates into the parameter space
\begin{align}
	\mathcal{Q}_{r, n}(q, \mathsf{D}, \mathsf{S}) \eqdef \left\{ (\btheta, \sigma) \in \R^r\times \R_+:\;\;\; \|\btheta\|_q^q/\|\btheta\|_2^q \leq \mathsf{D},\; \| \btheta\|_2/\sigma\geq \mathsf{S}  \right\}
\nonumber
\end{align}
for the Gaussian sequence model. Also, any estimator $\,\widetilde{\bbeta} = \widetilde{\bbeta}(\XX, \YY)$ in the original problem corresponds to an estimator $\,\widetilde{\btheta} = \eL \eU^\T \widetilde{\bbeta}(\XX, \YY)$, and since $\XX$ is a fixed known design, we can write $\widetilde{\btheta} = \widetilde{\btheta}(\eV^\T \YY/\sqrt{n})$, so that it is indeed an estimator in the Gaussian sequence model (the reverse is also true). Therefore,
\begin{align}
	\inf\limits_{\widetilde{\bbeta}} \sup\limits_{\mathcal{P}^{\XX}(q, \mathsf{D}, \mathsf{S})} \E\left[ \frac{\mathsf{MSE}(\widetilde{\bbeta})}{\mathsf{MSE}(0)}\right] \;=\; 
		\inf\limits_{\widetilde{\btheta}} \sup\limits_{\mathcal{Q}_{r, n}(q, \mathsf{D}, \mathsf{S})} \E\left[ \frac{\| \widetilde{\btheta} - \btheta\|_2^2}{\|\btheta\|_2^2}\right],
	\nonumber
\end{align}
and to establish the desired minimax lower bounds for the general problem it is enough to study the minimax lower bound for the Gaussian sequence model in the right-hand side, which we will do next.
\\
\\
(i) We take the following subset of $\,\mathcal{Q}_{r, n}(q, \mathsf{D}, \mathsf{S})$ to prove the minimax lower bound:
\begin{align}
	\mathcal{Q}^{poly}_{r, n}(q, \mathsf{S}) \eqdef \left\{ (\btheta, \sigma) \in \R^r\times \R_+:\;\;\; |\theta_j| = j^{-1/q},\; \sigma =  \| \btheta\|_2/\mathsf{S}  \right\}.
\nonumber
\end{align}
It is easy to check that indeed $\,\mathcal{Q}^{poly}_{r, n}(q, \mathsf{S}) \subseteq \mathcal{Q}_{r, n}(q, \mathsf{D}, \mathsf{S})$ for large enough constant $\,\mathsf{D}$. Also, all $\,\btheta\,$ from this new set of parameters have the same $\ell_2$-norm of constant order, which we denote $\,h$ for concreteness, even though its value will not play a role.  We can write
\begin{align}
	\inf\limits_{\widetilde{\btheta}} \sup\limits_{\mathcal{Q}_{r, n}(q, \mathsf{D}, \mathsf{S})} \E\left[ \frac{\| \widetilde{\btheta} - \btheta\|_2^2}{\|\btheta\|_2^2}\right] \,\geq\, 
		\inf\limits_{\widetilde{\btheta}} \sup\limits_{\mathcal{Q}^{poly}_{r, n}(q, \mathsf{S})} \E\left[ \frac{\| \widetilde{\btheta} - \btheta\|_2^2}{\|\btheta\|_2^2}\right] \,=\,
		\frac{1}{h^2} \inf\limits_{\widetilde{\btheta}} \sup\limits_{\mathcal{Q}^{poly}_{r, n}(q, \mathsf{S})} \E\left[ \| \widetilde{\btheta} - \btheta\|_2^2\right].
	\nonumber
\end{align}

The minimax risk on the right-hand side is easy to deal with using, for instance, \cite{Draft}. In particular, by Proposition 4.16 and (4.47) of \cite{Draft} the minimax risk decomposes into the sum of univariate minimax risks, which are given in (4.40) of \cite{Draft}. Thus,
\begin{align}
	 \inf\limits_{\widetilde{\btheta}} \sup\limits_{\mathcal{Q}^{poly}_{r, n}(q, \mathsf{S})} \E\left[ \| \widetilde{\btheta} - \btheta\|_2^2\right] \,\gtrsim\, \sum\limits_{j=1}^r \min\left( \frac{\sigma}{\sqrt{n}},\,j^{-1/q}\right)^2.
	\nonumber
\end{align}
One subtlety is that the results of \cite{Draft} that we used are derived for hyperrectangles, i.e. in the definition of $\,\mathcal{Q}^{poly}_{r, n}(q, \mathsf{S})\,$ we should have $\,|\theta_j| \leq j^{-1/q}\,$ instead of $\,|\theta_j| = j^{-1/q}$. (It is important for us to use equality here, because otherwise $\,\| \btheta\|_2\,$ cannot be bounded from below, and it is not clear how to get the minimax lower bound for the relative error.) However, the analysis of their proof shows that the lower bound holds also for $\,\mathcal{Q}^{poly}_{r, n}(q, \mathsf{S})\,$ defined in our way, since the underlying least favorable prior used to obtain the lower bound for the univariate minimax risk puts mass $\,1/2\,$ at the extremes of the interval, forcing its support to be contained in our parametric set. 

The right-hand side of the above display can be computed similarly to the proof of Proposition~\ref{Prop:tight} (iii), and we have
\begin{align}
	 \sum\limits_{j=1}^r \min\left( \frac{\sigma}{\sqrt{n}},\,j^{-1/q}\right)^2
	 \,\gtrsim\,\left(\frac{\sigma}{\sqrt{n}}\right)^{2-q}\,.
	\nonumber
\end{align}
Putting this all together, we obtain
\begin{align}
	\inf\limits_{\widetilde{\btheta}} \sup\limits_{\mathcal{Q}_{r, n}(q, \mathsf{D}, \mathsf{S})} \E\left[ \frac{\| \widetilde{\btheta} - \btheta\|_2^2}{\|\btheta\|_2^2}\right] \,\gtrsim\, 
		\frac{1}{h^2} \left(\frac{\sigma}{\sqrt{n}}\right)^{2-q} =
		\frac{1}{h^q} \frac{1}{(\mathsf{S}^2\, n)^{1-q/2}} \asymp \frac{1}{(\mathsf{S}^2\, n)^{1-q/2}} \,.
	\nonumber
\end{align}
\\
\\
(ii) The classical minimax lower bound for sparse linear regression are derived via a reduction to multiple hypothesis testing. In particular, one constructs a specific finite set of hypotheses $\,\{ \btheta^{(1)}, \ldots, \btheta^{(M)}\}\,$ (each of which has the desired sparsity) and using techniques from \cite{Tsybakov} shows (see, for instance, \cite{Rigollet}, Corollary 4.15 together with equivalence of Definition 4.1 and Definition 4.2 by (4.5))
\begin{align}	
	\inf\limits_{\widetilde{\btheta}} \sup\limits_{\substack{\btheta \in \R^r \\ \| \btheta \|_0 \leq \mathsf{D}}} \E\left[\| \widetilde{\btheta} - \btheta\|_2^2\right] \,\gtrsim\,
	\inf\limits_{\widetilde{\btheta}} \sup\limits_{\btheta \in \{ \btheta^{(l)} \}_{l=1}^M} \E\left[\| \widetilde{\btheta} - \btheta\|_2^2\right] \,\gtrsim\, \sigma^2 \frac{\mathsf{D} \log(er/\mathsf{D})}{n}.
	\nonumber
\end{align}
To get a minimax lower bound for the relative error, we need to make one slight modification. The construction of $\{ \btheta^{(1)}, \ldots, \btheta^{(M)} \}$ is based on the sparse Varshamov-Gilbert lemma (e.g. Lemma 4.14 in \cite{Rigollet}): it produces binary vectors  $\,\bomega^{(l)} \in \{ 0, 1\}^r$, $\,l\in[M]$, satisfying some properties, and then one sets $\btheta^{(l)} = C(r, n, \sigma, ...)\bomega^{(l)}\,$ for all $\,l\in[m]$, where $C(r,n,\sigma, ...)$ is carefully chosen and may depend on $r,n, \sigma,$ etc. With a simple linear transform $\bomega \to 2\bomega - \be$ (here $\be = [1, \ldots, 1]^\T\in \R^r$) we modify the binary vectors produced by the Varshamov-Gilbert lemma so that they belong to $\{ -1, 1\}^r$ after this modification. This forces the transformed vectors $\,\{\btheta^{(1)}, \ldots, \btheta^{(M)} \}\,$ to have the same $\ell_2$-norm without changing the essence of the argument. Now we can deal with the relative errors. Define for concreteness $\,h = \| \btheta^{(1)} \|_2 =\ldots = \|\btheta^{(M)}\|_2$, and write
\begin{align}	
	&\inf\limits_{\widetilde{\btheta}} \sup\limits_{\mathcal{Q}_{r, n}(0, \mathsf{D}, \mathsf{S})} \E\left[\frac{\| \widetilde{\btheta} - \btheta\|_2^2}{\|\btheta\|_2^2}\right] \,\gtrsim\,
	\inf\limits_{\widetilde{\btheta}} \sup\limits_{\substack{\btheta \in \{ \btheta^{(l)} \}_{l=1}^M\\\sigma = \|\btheta\|_2/\mathsf{S}}} \E\left[\frac{\| \widetilde{\btheta} - \btheta\|_2^2}{\|\btheta\|_2^2}\right] 
	\nonumber\\ &\qquad\gtrsim
	\frac{1}{h^2}\,\inf\limits_{\widetilde{\btheta}} \sup\limits_{\substack{\btheta \in \{ \btheta^{(l)} \}_{l=1}^M\\\sigma = h/\mathsf{S}}} \E\left[\| \widetilde{\btheta} - \btheta\|_2^2\right] 
	\,\gtrsim\, \frac{1}{h^2}\frac{h^2}{\mathsf{S}^2}\, \frac{\mathsf{D} \log(er/\mathsf{D})}{n} = \frac{\mathsf{D} \log(er/\mathsf{D})}{\mathsf{S}^2\,n}\,,
	\nonumber
\end{align}
as desired.
\subsection{Proof of Theorem~\ref{Th1a}}

Similarly to the proof of Theorem~\ref{Th1}, we write for $\,\ebeta\,$ from~\eqref{estimator2}
\begin{align}
 \mathsf{MSE}(\ebeta) = \left\| \eL^{-\varphi}\;\mathsf{T}_\tau\left[ \eL^\varphi (\btheta +  \bxi) \right] - \btheta\right\|_2^2 = \sum\limits_{j=1}^r \left| \widehat\lambda_j^{-\varphi/2} \; \mathsf{T}_{\tau}[\widehat{\lambda}_j^{\varphi/2}(\theta_j + \xi_j)] - \theta_j\right|^2.
\nonumber
\end{align}
For each individual term we apply the following two bounds. On one hand,
\begin{align}
	&\left| \widehat\lambda_j^{-\varphi/2} \; \mathsf{T}_{\tau}[\widehat{\lambda}_j^{\varphi/2}(\theta_j + \xi_j)] - \theta_j\right| =
	\left| \widehat\lambda_j^{-\varphi/2} \; \left( \mathsf{T}_{\tau}[\widehat{\lambda}_j^{\varphi/2}(\theta_j + \xi_j)] - \widehat{\lambda}_j^{\varphi/2}(\theta_j + \xi_j) \right) + \xi_j\right|
	\nonumber \\
	&\qquad \leq \widehat{\lambda}_j^{-\varphi/2}\tau + |\xi_j|
\leq \left( \frac{\widehat\lambda_1^{\varphi/2}}{\widehat\lambda_j^{\varphi/2}}+\frac{1}{2}\right) \sigma\rho \leq \frac{3}{2}\frac{\widehat\lambda_1^{\varphi/2}}{\widehat\lambda_j^{\varphi/2}} \,\sigma \rho,
\nonumber
\end{align}
where the first inequality uses property (ii) from Definition~\ref{thres} and the triangle inequality, and the second inequality holds on $\,\Omega_1$ by Lemma~\ref{L:noise}.
On the other hand,
\begin{align}
	&\left| \widehat\lambda_j^{-\varphi/2} \; \mathsf{T}_{\tau}[\widehat{\lambda}_j^{\varphi/2}(\theta_j + \xi_j)] - \theta_j\right| \leq \widehat\lambda_j^{-\varphi/2}\left| \mathsf{T}_{\tau}[\widehat{\lambda}_j^{\varphi/2}(\theta_j + \xi_j)] \right| + |  \theta_j |     \nonumber \\        &\qquad
	\leq \widehat\lambda_j^{-\varphi/2}\cdot c | \widehat\lambda_j^{\varphi/2} \theta_j| + |  \theta_j | 
	= (c+1)\,|\theta_j|,
\nonumber
\end{align}
where the second inequality is due to property (i) from Definition~\ref{thres} applied to $\,z = \widehat{\lambda}_j^{\varphi/2}(\theta_j + \xi_j)\,$ and $\,z^\prime = \widehat{\lambda}_j^{\varphi/2}\theta_j\,$ satisfying $\,|z - z^\prime| = \widehat{\lambda}_j^{\varphi/2}|\xi_j| \leq \widehat{\lambda}_1^{\varphi/2} \sigma\rho/2 = \tau/2\,$ on $\,\Omega_1\,$ by Lemma~\ref{L:noise}.
Thus, on $\,\Omega_1$
\begin{equation}
            \begin{aligned}
            	\mathsf{MSE}(\ebeta) &
            	= \left\| \eL^{-\varphi}\;\mathsf{T}_\tau\left[ \eL^{\varphi}\,(\btheta +  \bxi) \right] - \btheta\right\|_2^2
            	\lesssim \sum\limits_{j=1}^r \min\left( \frac{\widehat\lambda_1^{\varphi/2}}{\widehat\lambda_j^{\varphi/2}} \,\sigma \rho,\, |\theta_j|\right)^2.
                \nonumber
            \end{aligned}
        \end{equation}
 
  Now we choose $\,\mathsf{T}_\tau[\,\cdot\,] = \mathsf{SOFT}_\tau[\,\cdot\,]\,$ to prove the desired lower bound. We again consider two cases:
 \begin{itemize}
 	\item If $\left|\widehat{\lambda}_j^{\varphi/2}(\theta_j + \xi_j)\right| > \tau\,$, then on $\,\Omega_1\,$ using Lemma~\ref{L:noise}
 	\begin{align}
	&\left| \widehat\lambda_j^{-\varphi/2} \; \mathsf{SOFT}_{\tau}[\widehat{\lambda}_j^{\varphi/2}(\theta_j + \xi_j)] - \theta_j\right| \nonumber \\
	&\qquad =
	\left| \widehat\lambda_j^{-\varphi/2} \; \left( \mathsf{SOFT}_{\tau}[\widehat{\lambda}_j^{\varphi/2}(\theta_j + \xi_j)] - \widehat{\lambda}_j^{\varphi/2}(\theta_j + \xi_j) \right) + \xi_j\right|
	\nonumber \\
	&\qquad \geq \widehat{\lambda}_j^{-\varphi/2}\tau - |\xi_j|
\geq \left( \frac{\widehat\lambda_1^{\varphi/2}}{\widehat\lambda_j^{\varphi/2}}-\frac{1}{2}\right) \sigma\rho \geq \frac{1}{2}\frac{\widehat\lambda_1^{\varphi/2}}{\widehat\lambda_j^{\varphi/2}} \,\sigma \rho.
\nonumber
\end{align}
 	\item If $\left|\widehat{\lambda}_j^{\varphi/2}(\theta_j + \xi_j)\right| \leq \tau\,$, then 
 	\begin{align}
	&\left| \widehat\lambda_j^{-\varphi/2} \; \mathsf{SOFT}_{\tau}[\widehat{\lambda}_j^{\varphi/2}(\theta_j + \xi_j)] - \theta_j\right| = |\widehat\lambda_j^{-\varphi/2} \cdot 0 - \theta_j| = |\theta_j|
\nonumber
\end{align}
 \end{itemize}
In any case,
\begin{align}
	&\left| \widehat\lambda_j^{-\varphi/2} \; \mathsf{SOFT}_{\tau}[\widehat{\lambda}_j^{\varphi/2}(\theta_j + \xi_j)] - \theta_j\right| \gtrsim \min\left( \frac{\widehat\lambda_1^{\varphi/2}}{\widehat\lambda_j^{\varphi/2}} \,\sigma \rho,\, |\theta_j|\right),
\nonumber
\end{align}
implying the matching lower bound.

 \subsection{Proof of Corollary~\ref{Cor1}}
 We need to upper bound the right-hand side of the inequality obtained in Theorem~\ref{Th1a}.
 Let us define the following auxiliary set:
 \begin{align}
 	&\mathcal{P} \eqdef \left\{ j\in [r]\;\Big|\;\; \widehat{\lambda}_j \geq \widehat{\lambda}_1\,\rho^{2/(2+\varphi)}\right\}.
 	\nonumber
 \end{align}
  We first bound
  \begin{align}
  	\sum\limits_{j\notin \mathcal{P}} \min\left( \frac{\widehat\lambda_1^{\varphi/2}}{\widehat\lambda_j^{\varphi/2}} \,\sigma \rho, \; |\theta_j|\right)^2
  	\leq
  	\sum\limits_{j\notin \mathcal{P}} |\theta_j|^2 = \sum\limits_{j\notin \mathcal{P}}  \widehat\lambda_j  (\widehat\bu_j^\T\bbeta)^2 \leq \widehat\lambda_1 \|\bbeta\|_2^2 \,\rho^{2/(2+\varphi)},
  \nonumber
  \end{align}
  where we used only the definition of $\,\mathcal{P}\,$ (more specifically its complement).
  Then, we bound
  \begin{align}
  	\sum\limits_{j\in \mathcal{P}} \min\left( \frac{\widehat\lambda_1^{\varphi/2}}{\widehat\lambda_j^{\varphi/2}} \,\sigma \rho, \; |\theta_j|\right)^2
  	\leq
  	\sum\limits_{j\in \mathcal{P}} \frac{\widehat\lambda_1^{\varphi/2}}{\widehat\lambda_j^{\varphi/2}} \,\sigma \rho \;|\theta_j|
  	\leq
  	\sigma \rho^{-\frac{2}{2+\varphi}\cdot\frac{\varphi}{2} + 1}\,
  	\sum\limits_{j\in \mathcal{P}} |\theta_j|
  	\leq \sigma \rho^{2/(2+\varphi)} \,\| \btheta\|_1,
  \nonumber
  \end{align}
  where we again used the definition of $\,\mathcal{P}$. Now we apply $\,\| \btheta\|_1 \leq \| \Se\|^{1/2} \| \bbeta\|_2 \,\reff[\Se]^{1/2}$, and adding the above two inequalities yields the desired statement.

\subsection{Proof of Theorem~\ref{Th:CV}}
	Using standard ``empirical risk minimization'' reasoning, we write
	\begin{align}
		&\frac{1}{L}\sum\limits_{l=1}^L \E\left[ \left( y - \bx^\T \ebeta^{(l)}_{\tau^{cv}}\right)^2\right] 
		\leq 
		\frac{1}{L}\sum\limits_{l=1}^L \frac{1}{|\mathcal{B}_l|} \sum\limits_{i\in\mathcal{B}_l} \left( y_i - \bx_i^\T \ebeta^{(l)}_{\tau^{cv}}\right)^2 
		\nonumber \\ &\qquad\qquad+
		\frac{1}{L}\sum\limits_{l=1}^L \sup\limits_{\bbeta^\prime \in \{ \ebeta^{(l)}_\tau\}_{\tau \geq 0}} \left| \frac{1}{|\mathcal{B}_l|} \sum\limits_{i\in\mathcal{B}_l} \left( y_i - \bx_i^\T \bbeta^\prime\right)^2  -\E\left[ \left( y - \bx^\T \bbeta^\prime\right)^2\right]  \right|
		\nonumber \\ &\qquad \leq 
		\frac{1}{L}\sum\limits_{l=1}^L \frac{1}{|\mathcal{B}_l|} \sum\limits_{i\in\mathcal{B}_l} \left( y_i - \bx_i^\T \ebeta^{(l)}_{\tau^{oracle}}\right)^2 
		\nonumber \\ &\qquad\qquad+
		\frac{1}{L}\sum\limits_{l=1}^L \sup\limits_{\bbeta^\prime \in \{ \ebeta^{(l)}_\tau\}_{\tau \geq 0}} \left| \frac{1}{|\mathcal{B}_l|} \sum\limits_{i\in\mathcal{B}_l} \left( y_i - \bx_i^\T \bbeta^\prime\right)^2  -\E\left[ \left( y - \bx^\T \bbeta^\prime\right)^2\right]  \right|
		\nonumber \\ &\qquad \leq 
		\frac{1}{L}\sum\limits_{l=1}^L \E\left[\left( y - \bx^\T \ebeta^{(l)}_{\tau^{oracle}}\right)^2 \right]
		\nonumber \\ &\qquad\qquad+
		\frac{2}{L}\sum\limits_{l=1}^L \sup\limits_{\bbeta^\prime \in \{ \ebeta^{(l)}_\tau\}_{\tau \geq 0}} \left| \frac{1}{|\mathcal{B}_l|} \sum\limits_{i\in\mathcal{B}_l} \left( y_i - \bx_i^\T \bbeta^\prime\right)^2  -\E\left[ \left( y - \bx^\T \bbeta^\prime\right)^2\right]  \right|.
		\nonumber
	\end{align}
	Here in the second inequality we used the definition of $\,\tau^{cv}\,$, namely the fact that it minimizes the cross-validation error. The expectations in the above expressions are over $\,(\bx, y)\,$ only, so using $\,\E\left[ \left( y - \bx^\T \ebeta^{(l)}_{\tau}\right)^2\right]  = \mathsf{PE}(\ebeta^{(l)}_{\tau}) + \sigma^2\,$ the only thing left is to bound the supremums in the right-hand side. We will do this for each $\,l \in [L]\,$ similarly, so from now on we fix $\,l \in [L]$. Note that due to the structure of our estimator $\,\ebeta^{(l)}_\tau$, the set $\{ \ebeta^{(l)}_\tau \}_{\tau \geq 0}$ contains at most $\,(r+1)\,$ distinct estimators, each derived from the training sample of the $\,l$-th fold $\,\{ (\bx_i, y_i) \}_{i\in [n]\setminus\mathcal{B}_l}$, and thus independent of the validation set of the $\,l$-th fold $\,\{ (\bx_i, y_i) \}_{i\in \mathcal{B}_l}$. Working conditionally on $\,\{ (\bx_i, y_i) \}_{i\in [n]\setminus\mathcal{B}_l}$, we can bound 
		\begin{align}
		 \left| \frac{1}{|\mathcal{B}_l|} \sum\limits_{i\in\mathcal{B}_l} \left( y_i - \bx_i^\T \bbeta^\prime\right)^2  -\E\left[ \left( y - \bx^\T \bbeta^\prime\right)^2\right]  \right|
		\nonumber
	\end{align}
	with high probability for each single $\bbeta^\prime \in \{ \ebeta^{(l)}_\tau \}_{\tau \geq 0}$ (which are treated as deterministic vectors), and the rest of the proof will easily follow.  Let us focus on an arbitrary $\bbeta^\prime \in \{ \ebeta^{(l)}_\tau \}_{\tau \geq 0}$.
	
	To apply some concentration results, we first show that $\,(y_i - \bx_i^\T\bbeta^\prime)\,$ are sub-Weibull random variables with parameter $\,\alpha/2$. Indeed, $\,\|\varepsilon_i \|_{\psi_\alpha} \leq \sigma$ by Assumption~\ref{A:Noise} and 
	\begin{align}
		 \| \bx_i^\T (\bbeta -\bbeta^\prime)\|_{\psi_\alpha} &\leq
		  \| \bx_i^\T (\bbeta -\bbeta^\prime)\|_{\psi_2} =
		  \| (\St^{-1/2} \bx_i)^\T  \St^{1/2}(\bbeta -\bbeta^\prime)\|_{\psi_2}
		  \nonumber \\ & =
		    \left\| (\St^{-1/2} \bx_i)^\T \frac{\St^{1/2}(\bbeta -\bbeta^\prime)}{\| \St^{1/2}(\bbeta -\bbeta^\prime)\|_2} \right\|_{\psi_2} \cdot \sqrt{(\bbeta-\bbeta^\prime)^\T \St (\bbeta-\bbeta^\prime)}
		    \nonumber \\ & \lesssim
		    \sqrt{(\bbeta-\bbeta^\prime)^\T \St (\bbeta-\bbeta^\prime)},
		\nonumber
	\end{align}
	where in the first inequality we used the monotonicity of the Orlicz norm w.r.t. parameter $\,\alpha$, and in the last inequality we used Assumption~\ref{A:X} together with the fact that the inner product of a sub-Gaussian vector with a unit vector is a sub-Gaussian random variable. By properties of the Orlicz norm we further have
	\begin{align}
		  \| y_i - \bx_i^\T \bbeta^\prime\|_{\psi_\alpha} &=
		  \| \bx_i^\T (\bbeta -\bbeta^\prime) + \varepsilon_i\|_{\psi_\alpha}
		   \nonumber \\ &\lesssim
		  \| \bx_i^\T (\bbeta -\bbeta^\prime)\|_{\psi_\alpha} + \|\varepsilon_i\|_{\psi_\alpha} \lesssim \sqrt{(\bbeta-\bbeta^\prime)^\T \St (\bbeta-\bbeta^\prime)} + \sigma.
		  \nonumber
	\end{align}
	and due to Proposition~D.2 of \cite{Weibull}
	\begin{align}
		  \left\| (y_i - \bx_i^\T \bbeta^\prime)^2\right\|_{\psi_{\alpha/2}} &\leq
		  \| y_i - \bx_i^\T \bbeta^\prime\|_{\psi_{\alpha}}^2
		  \nonumber \\ &
		  \lesssim
		    (\bbeta-\bbeta^\prime)^\T \St (\bbeta-\bbeta^\prime) + \sigma^2
		   \lesssim
		   \bbeta^\T \St \bbeta + {\bbeta^\prime}^\T \St \bbeta^\prime + \sigma^2.
		  \nonumber
	\end{align}
	After subtracting the expectation, the same bound holds for the centered random variable, but with a different hidden constant. Now we readily apply Theorem~3.1 of \cite{Weibull} with $\,X_i \eqdef (y_i - \bx_i^\T \bbeta^\prime)^2 - \E\left[ (y - \bx^\T \bbeta^\prime)^2\right]$, $\,a_i = 1/|\mathcal{B}_l|$ for all $\,i\in[|\mathcal{B}_l|]\,$ and parameter $\,\alpha/2$:
	\begin{align}
		&\Prob\Bigg[ \left| \frac{1}{|\mathcal{B}_l|} \sum\limits_{i\in \mathcal{B}_l} (y_i - \bx_i^\T \bbeta^\prime)^2 - \E\left[ (y - \bx^\T \bbeta^\prime)^2\right] \right| \gtrsim
		\label{to_int} \\
		& \qquad\gtrsim ( \bbeta^\T \St \bbeta + {\bbeta^\prime}^\T \St \bbeta^\prime + \sigma^2) \left( \sqrt{\frac{t}{|\mathcal{B}_l|}} + \frac{t^{2/\alpha}}{|\mathcal{B}_l|}\right)\;\;\Bigg|\;\; \{ (\bx_i, y_i) \}_{i\in[n]\setminus\mathcal{B}_l}\;\Bigg] \leq 2e^{-t}
		\nonumber
	\end{align}
	for all $\,t \geq 0$.  Now we need to carefully integrate out $\,\{ (\bx_i, y_i) \}_{i\in[n]\setminus\mathcal{B}_l}$, and before this, we need to bound ${\bbeta^\prime}^\T \St \bbeta^\prime$ with high probability over $\,\{ (\bx_i, y_i) \}_{i\in[n]\setminus\mathcal{B}_l}$.
	
	In this paragraph, to keep the notation light, we drop the superscript $(l)$, but keep in mind that we work with the sample $\,\{ (\bx_i, y_i) \}_{i\in[n]\setminus\mathcal{B}_l}$, and for the purposes of this paragraph the quantities $\,\Se, \,\eL,\,\eU,\, \btheta,\, \bxi$  correspond to the training sample of $l$-th fold $\,\{ (\bx_i, y_i) \}_{i\in[n]\setminus\mathcal{B}_l}$ only. With any $\tau \geq 0$, we have the following:
	\begin{align}
		{\bbeta^\prime}^\T \St \bbeta^\prime &= 
		{\ebeta_\tau^{(l)}}^\T \St \ebeta^{(l)}_\tau
		\nonumber \\ &= 
		\mathsf{SOFT}_\tau[ \btheta_{\leq k^*} + \bxi_{\leq k^*}]^\T 
		\,\eL_{\leq k^*}^{-1} \eU_{\leq k^*}^\T \bU \bLambda^2 \bU^\T\eU_{\leq k^*}\eL_{\leq k^*}^{-1}
		\,\mathsf{SOFT}_\tau[ \btheta_{\leq k^*} + \bxi_{\leq k^*}]
		\nonumber \\ &\leq
		\left\| \bLambda \bU^\T\eU_{\leq k^*}\eL_{\leq k^*}^{-1}\right\|^2\,\left\| \mathsf{SOFT}_\tau[ \btheta_{\leq k^*} + \bxi_{\leq k^*}]\right\|_2^2.
		\nonumber
		\end{align}
	By Lemma~\ref{L:opnorm}, 
	\begin{align}
		\left\| \bLambda \bU^\T\eU_{\leq k^*} \eL_{\leq k^*}^{-1}\right\| \lesssim 1
	\nonumber
	\end{align}
	on some set $\,\Omega_3^{(l)}$, defined similarly to $\,\Omega_3\,$ (which is introduced after Lemma~\ref{L:psi_n}) but for $\,\{ (\bx_i, y_i) \}_{i\in[n]\setminus\mathcal{B}_l}\,$ rather than for the whole sample. Also,
	\begin{align}
		\left\| \mathsf{SOFT}_\tau[ \btheta_{\leq k^*} + \bxi_{\leq k^*}]\right\|_2^2
		\leq \left\| \btheta_{\leq k^*} + \bxi_{\leq k^*}\right\|_2^2
		\lesssim \| \btheta_{\leq k^*}\|_2^2 + \| \bxi_{\leq k^*} \|_2^2
		\nonumber
		\end{align}
	with
	\begin{align}
		\| \btheta_{\leq k^*}\|_2^2 = \bbeta^\T \Se \bbeta \leq \bbeta^\T \St \bbeta  + \|\St\| \|\bbeta\|_2^2\sqrt{\frac{\reff[\St] + \log(1/\delta)}{n-|\mathcal{B}_l|}} \lesssim \|\St\| \|\bbeta\|_2^2
	\nonumber
	\end{align}
	by Lemma~\ref{L:covconc} on $\,\Omega_2^{(l)}\,$ (defined similarly to $\,\Omega_2$ after Lemma~\ref{L:covconc}, but for $\,\{ (\bx_i, y_i) \}_{i\in[n]\setminus\mathcal{B}_l}$) and Assumption~\ref{A:tech},  and
	\begin{align}
	\| \bxi_{\leq k^*} \|_2^2 \leq r \| \bxi_{\leq k^*} \|_\infty^2 \lesssim \sigma^2 ( \log(2r/\delta))^{2/\alpha}
	\nonumber
	\end{align}
	by Lemma~\ref{L:noise} on $\,\Omega_1^{(l)}\,$ (defined similarly to $\,\Omega_1$ after Lemma~\ref{L:noise}, but for $\,\{ (\bx_i, y_i) \}_{i\in[n]\setminus\mathcal{B}_l}$). Putting this all together,
	\begin{align}
		\bbeta^\prime \St \bbeta^\prime \lesssim \|\St\| \|\bbeta\|_2^2 + \sigma^2 ( \log(2r/\delta))^{2/\alpha}
		\nonumber
	\end{align}
	on $\,\Omega_1^{(l)} \cap \Omega_2^{(l)} \cap \Omega_3^{(l)}$. Note that $\,\Prob[\Omega_1^{(l)} \cap \Omega_2^{(l)} \cap \Omega_3^{(l)}] \geq 1-3\delta$, with the probability taken over the randomness of $\,\{ (\bx_i, y_i) \}_{i\in[n]\setminus\mathcal{B}_l}$.
	
	Taking the bound in the previous display into account, we integrate \eqref{to_int}. We split the integral into two parts: the integration over $\,\Omega_1^{(l)} \cap \Omega_2^{(l)} \cap \Omega_3^{(l)}\,$ allows to replace ${\bbeta^\prime}^\T \St \bbeta^\prime$ inside the conditional probability with its deterministic bound on this set and gives $\,2e^{-t}\,$ in the right-hand side, and the integration over $\,\left(\Omega_1^{(l)} \cap \Omega_2^{(l)} \cap \Omega_3^{(l)}\right)^c$, where the bound on $\,{\bbeta^\prime}^\T \St \bbeta^\prime\,$ may be violated, adds at most $\,3\delta$ to the probability of the bad event. Therefore,
	\begin{align}
		&\Prob\Bigg[ \left| \frac{1}{|\mathcal{B}_l|} \sum\limits_{i\in \mathcal{B}_l} (y_i - \bx_i^\T \bbeta^\prime)^2 - \E\left[ (y - \bx^\T \bbeta^\prime)^2\right] \right| \gtrsim
		\nonumber \\
		& \qquad\qquad\qquad\gtrsim \left(\|\St\| \|\bbeta\|_2^2 + \sigma^2\left( \log(2r/\delta) \right)^{2/\alpha}\right) \left( \sqrt{\frac{t}{|\mathcal{B}_l|}} + \frac{t^{2/\alpha}}{|\mathcal{B}_l|}\right)\Bigg] \leq 2e^{-t} + 3\delta.
		\nonumber
	\end{align}
	Now we apply the union bound (recall that the cardinality of $\,\{ \ebeta^{(l)}_\tau\}_{\tau \geq 0}\,$ does not exceed $\,r+1$)
	\begin{align}
		&\Prob\Bigg[ \sup\limits_{\bbeta^\prime \in \{ \ebeta^{(l)}_\tau\}_{\tau \geq 0}}\left| \frac{1}{|\mathcal{B}_l|} \sum\limits_{i\in \mathcal{B}_l} (y_i - \bx_i^\T \bbeta^\prime)^2 - \E\left[ (y - \bx^\T \bbeta^\prime)^2\right] \right| \gtrsim
		\nonumber \\
		& \qquad\qquad\qquad\gtrsim \left(\|\St\| \|\bbeta\|_2^2 + \sigma^2( \log(2r/\delta) )^{2/\alpha} \right) \left( \sqrt{\frac{t}{|\mathcal{B}_l|}} + \frac{t^{2/\alpha}}{|\mathcal{B}_l|}\right)\Bigg] \leq 2(r+1)e^{-t} + 3\delta.
		\nonumber
	\end{align}
	Finally, by yet another application of the union bound we obtain
	\begin{align}
		&\Prob\Bigg[ \frac{2}{L}\sum\limits_{l=1}^L \sup\limits_{\bbeta^\prime \in \{ \ebeta^{(l)}_\tau\}_{\tau \geq 0}}\left| \frac{1}{|\mathcal{B}_l|} \sum\limits_{i\in \mathcal{B}_l} (y_i - \bx_i^\T \bbeta^\prime)^2 - \E\left[ (y - \bx^\T \bbeta^\prime)^2\right] \right| \gtrsim
		\nonumber \\
		& \qquad\qquad\gtrsim \left(\|\St\| \|\bbeta\|_2^2 + \sigma^2( \log(2r/\delta) )^{2/\alpha}\right) \left( \sqrt{\frac{t}{|\mathcal{B}_l|}} + \frac{t^{2/\alpha}}{|\mathcal{B}_l|}\right)\Bigg] \leq 2L(r+1)e^{-t} + 3L\delta.
		\nonumber
	\end{align}
	We conclude the proof by picking $\,t = \log(2Lr/\delta)$ and adjusting the constants throughout the proof to make sure that the desired result holds with probability $\,1-\delta$.

\section{Additional proofs} \label{S:auxproofs}

\subsection{Proof of Proposition~\ref{Prop:tight}}
	(i) and (ii) follow trivially by taking $\,q = 0\,$ and $\,q = \nu$, respectively. Now we focus on (iii). Define $\,j^*$ such that $|\theta_{(j^*)}| \asymp \sigma\rho$, i.e. take $\,j^* = \lfloor (\sigma\rho/|\theta_{(1)}|)^{-1/a} \rfloor$. 
		
	If $\,j^* \geq r$, then taking $\,q = 0$ we have
	\begin{align}
		\sum\limits_{j=1}^r \min(\sigma\rho,\, |\theta_{(j)}|)^2 \asymp r(\sigma\rho)^2 \geq \inf\limits_{q\in[0,\, 2]} \left\{ \| \btheta\|_q^q (\sigma\rho)^{2-q}\right\}.
		\nonumber
	\end{align}
	For the rest of the proof assume $\,j^* < r$. Then
	\begin{align} 
		&\sum\limits_{j=1}^r \min(\sigma\rho,\, |\theta_{(j)}|)^2 
		\asymp \sum\limits_{j=1}^{j^*} (\sigma\rho)^2 + \sum\limits_{j=j^*+1}^r |\theta_{(j)}|^2
		\asymp j^*(\sigma\rho)^2 + |\theta_{(1)}| \sum\limits_{j=j^*+1}^r j^{-2a}
		\nonumber\\
		&\qquad 
		\,\asymp\, \begin{cases}
			j^*(\sigma\rho)^2 + |\theta_{(1)}|\,({j^*}^{-2a+1} - r^{-2a+1}), & \;\;a > 1/2,\\
			j^*(\sigma\rho)^2 + |\theta_{(1)}|\, \log(r/j^*), & \;\;a = 1/2,\\
			j^*(\sigma\rho)^2 + |\theta_{(1)}|\, (r^{-2a+1} - {j^*}^{-2a+1}), & \;\;a < 1/2;\\
		\end{cases}
		\nonumber\\
		&\qquad 
		\,\asymp\, \begin{cases}
			|\theta_{(1)}|^{1/a}\,(\sigma\rho)^{2-1/a}, & \;\;a > 1/2,\\
			|\theta_{(1)}|\, \log(er/j^*), & \;\;a = 1/2,\\
			|\theta_{(1)}|\, r^{-2a+1}, & \;\;a < 1/2;\\
		\end{cases}
	\nonumber
	\end{align}
	At the same time, taking $\,q = 1/a \in (0,\,2]\,$ for $\,a \geq 1/2\,$ and $\,q=2\,$ for $\,a < 1/2$, we get
	\begin{align}
		 &\inf\limits_{q\in[0,\, 2]} \left\{ \| \btheta\|_q^q (\sigma\rho)^{2-q}\right\}
		 	\,\lesssim \,
		 	\begin{cases}
		 		(\sigma\rho)^{2-1/a} \,|\theta_{(1)}|^{1/a} \sum\limits_{j=1}^r (j^{-a})^{1/a},&\;\;a \geq 1/2, \\
		 		|\theta_{(1)}| \,\sum\limits_{j=1}^r j^{-2a},&\;\;a < 1/2;
		 	\end{cases}
		 	\nonumber \\
		 	&\qquad
		 	\,\asymp\, 
		 	\begin{cases}
		 		|\theta_{(1)}|^{1/a}\,(\sigma\rho)^{2-1/a} \,\log(r),&\;\;a \geq 1/2, \\
		 		|\theta_{(1)}|\, r^{-2a+1}, &\;\;a < 1/2;
		 	\end{cases}
		\nonumber
	\end{align}
	Comparing this to the expressions for $\,\sum_{j=1}^r \min(\sigma\rho,\, |\theta_{(j)}|)^2\,$ above, we conclude the proof.
	
\subsection{Proof of Lemma~\ref{L:noise}}
	It is straightforward to verify that $\bxi$ is also a sub-Weibull random vector conditionally on $\XX$:
	\begin{equation}
            \begin{aligned}
            	\sup\limits_{\|\bw\|_2=1} \left\| \bw^\T \frac{\eL^{-1} \eU^\T \XX^\T \beps}{n} \right\|_{\psi_\alpha} \leq \frac{1}{\sqrt{n}} \sup\limits_{\|\bw\|_2=1} \| \bw^\T \beps\|_{\psi_\alpha} \leq \frac{\sigma}{\sqrt{n}},
                \nonumber
            \end{aligned}
        \end{equation}
        where the first inequality holds given $\XX$ since $\| \XX\eU\eL^{-1}\bw/\sqrt{n}\|_2 = 1$,
        and the last inequality is due to Assumption~\ref{A:Noise}. Then, taking $\be_1, \ldots, \be_r$ (the standard basis in $\R^r$) we have
	\begin{equation}
            \begin{aligned}
            	& \Prob\left[ |\be_j^{\top} \bxi| \geq t \,\big|\,\XX\right] \leq 2\,\exp\left(-\left(t\sqrt{n}/\sigma\right)^\alpha\right)\;\;\;\text{for }j \in[r].
                \nonumber
            \end{aligned}
        \end{equation}
    Applying the union bound and plugging in $t = \sigma\rho/2$, we get the desired.

\subsection{Proof of Lemma~\ref{L:psi_n}}
	The proof is pretty standard and can be found in numerous papers.
	Fix $l, l^\prime \in [d]$. We have
	\begin{equation}
            \begin{aligned}
            	\overline{\eta}_{ll^\prime} &= \frac{\bu_l^\T (\Se-\St) \bu_{l^\prime}}{\sqrt{\lambda_l\lambda_{l^\prime}}}
            	= \bu_l^\T (\St^{-1/2}\Se\St^{-1/2} - \Id_d) \bu_{l^\prime}\\
            	&= \frac{1}{n}\sum\limits_{i=1}^n \left\{ (\bu_l^\T \St^{-1/2} \bx_i)\cdot(\bu_{l^\prime}^\T \St^{-1/2} \bx_i) - \E\left[ (\bu_l^\T \St^{-1/2} \bx_i)\cdot(\bu_{l^\prime}^\T \St^{-1/2} \bx_i) \right]\right\}.
                \nonumber
            \end{aligned}
        \end{equation}
        Since by Assumption~\ref{A:X} $\St^{-1/2} \bx_i$ is sub-Gaussian for $i\in [n]$, then by Lemma 2.7.7 of \cite{Vershynin} $(\bu_l^\T \St^{-1/2} \bx_i)\cdot(\bu_{l^\prime}^\T \St^{-1/2} \bx_i)$ is sub-Exponential and by Exercise 2.7.10 of \cite{Vershynin} its centered version is also sub-Exponential for $i\in [n]$. Bernstein's inequality (e.g. Corollary 2.8.3 of \cite{Vershynin}) applied to this centered random variables implies
        \begin{equation}
            \begin{aligned}
            	\Prob\left[ |\overline{\eta}_{ll^\prime}| \geq t\right] \leq 2\exp\left( -cn\, \min\left(\frac{t^2}{C^2}, \frac{t}{C}\right)\right).
                \nonumber
            \end{aligned}
        \end{equation}
        By union bound,
        \begin{equation}
            \begin{aligned}
            	\Prob\left[ \max\limits_{l,l^\prime\in[d]}|\overline{\eta}_{ll^\prime}| \geq t\right] \leq 2d^2\exp\left( -c n\,\min\left(\frac{t^2}{C^2}, \frac{t}{C}\right)\right).
                \nonumber
            \end{aligned}
        \end{equation}
        Taking $t = C\sqrt{\log(2d^2/\delta)/n}$ for some other properly chosen $C$ and using Assumption~\ref{A:tech} to make sure $t^2/c^2 \leq t/c$, we conclude the proof.

\subsection{Proof of Lemma~\ref{L:relbounds}}
	The first part follows from Corollary 2 of \cite{Wahl}, and the second part follows from Lemma 4 of \cite{Wahl}. Conditions (2.1) and $|\eta_{ll^\prime}| \leq x$ for all $l,l^\prime\in[d]$ are satisfied since we work on $\Omega_3$ from Lemma~\ref{L:psi_n} and consider $j\in[k^*]$ with properly defined $k^*$.

\subsection{Proof of Lemma~\ref{L:opnorm}}
	We first apply inequalities
	\begin{align}
		\| \bLambda \bU^\T \eU_{\leq k^*} \eL_{\leq k^*}^{-1} \| &=
		\| \bLambda \bU^\T \eU_{\leq k^*} \bLambda_{\leq k^*}^{-1} \bLambda_{\leq k^*}\eL_{\leq k^*}^{-1} \|
		\leq
		\| \bLambda \bU^\T \eU_{\leq k^*} \bLambda_{\leq k^*}^{-1} \| \| \bLambda_{\leq k^*}\eL_{\leq k^*}^{-1} \|
		\nonumber
		\\&\leq
		\left( \left\| \bLambda \bU^\T \eU_{\leq k^*} \bLambda_{\leq k^*}^{-1} - \begin{bmatrix} \Id_{k^*} \\ \Oo_{(d-k^*) \times k^*} \end{bmatrix} \right\| + 1\right) \| \bLambda_{\leq k^*}\eL_{\leq k^*}^{-1} \|.
		\nonumber
	\end{align}
	The spectral norm of $\bLambda_{\leq k^*}\eL_{\leq k^*}^{-1}$ is easy to control:
	\begin{align}
		\| \bLambda_{\leq k^*}\eL_{\leq k^*}^{-1} \| = \left( \max\limits_{j\in[k^*]} \frac{\lambda_j}{\widehat{\lambda}_j}\right)^{1/2} \leq \frac{1}{(1-C\epsilon)^{1/2}} \leq 4,
		\nonumber
	\end{align}
	where the first inequality is due to Lemma~\ref{L:relbounds} on $\Omega_3$ and in the last inequality we assumed that $\epsilon$ is small enough by Assumption~\ref{A:tech}.
	
	Next, let us focus on $\bLambda \bU^\T \eU_{\leq k^*} \bLambda_{\leq k^*}^{-1} - [ \Id_{k^*}  \,\Oo_{k^*\times (d-k^*)} ]^\T$, which we denote by $\bH$ for shortness. Denote its columns as $\bh_1, \ldots, \bh_{k^*}$. We can bound $\ell_1$-norm of each column, using Lemma~\ref{L:relbounds}, as
	\begin{align}
		\| \bh_j \|_1 &= \sum\limits_{\substack{l=1\\l\neq j}}^d \left| \frac{\lambda_l^{1/2} \bu_l^\T \widehat{\bu}_j}{\lambda_j^{1/2}}\right| + |\bu_j^\T \widehat{\bu}_j - 1|
		=
		\sum\limits_{\substack{l=1\\l\neq j}}^d \left| \frac{\lambda_l^{1/2} \bu_l^\T \widehat{\bu}_j}{\lambda_j^{1/2}}\right| + \frac{1}{2}\,\| \widehat{\bu}_j - \bu_j^\T\|_2^2
		\nonumber\\&\leq
		\epsilon\sum\limits_{\substack{l=1\\l\neq j}}^d \frac{\lambda_l^{1/2} }{\lambda_j^{1/2}} \frac{\lambda_j^{1/2} \lambda_l^{1/2}}{|\lambda_j - \lambda_l|} + C\epsilon^2\sum\limits_{\substack{l=1\\l\neq j}}^d \frac{\lambda_j \lambda_l}{(\lambda_j - \lambda_l)^2}
	\nonumber
	\end{align}
	on $\Omega_3$.
	Applying \cite{Wahl}, inequalities (3.30), or \cite{Jirak}, Lemma 7.13, together with Assumption~\ref{A:cvxdecay}, we get
	\begin{align}
		\| \bh_j \|_1 \leq \epsilon \,j \log(j) + C\epsilon^2\,j^2.
	\nonumber
	\end{align}
	Since $j\leq k^*$ and $\epsilon k^* \leq C$, we have
	\begin{align}
		\| \bh_j \|_1 \leq C\epsilon \,j \log(k^*) \;\;\;\text{for all }\;j\in[k^*].
	\nonumber
	\end{align}
	Finally, we have for the Frobenius norm $\| \bH \|_{\Fr}$ on $\Omega_3$
	\begin{align}
		\| \bH \|_{\Fr}^2 = \sum\limits_{j=1}^{k^*} \| \bh_j\|_2^2 \leq C \sum\limits_{j=1}^{k^*} \epsilon^2 j^2 \log(k^*)^2
		= C\epsilon^2 {k^*}^{3} \log(k^*)^2 \leq C^\prime,
	\nonumber
	\end{align}
	where we used the definition of $k^*$.
	The inequality between the spectral and the Frobenius norms completes the proof.

\subsection{Proof of Lemma~\ref{L:eigen}}
	Fix arbitrary $j\in[k^*]$. We have the following chain of inequalities:
	\begin{equation}
            \begin{aligned}
            	|\widehat{\lambda}_j^{1/2}\widehat{\bu}_j^\T \bbeta-\lambda_j^{1/2}\bu_j^\T \bbeta|
            	&\leq \|\widehat{\lambda}_j^{1/2}\widehat{\bu}_j-\lambda_j^{1/2}\bu_j\|_2 \|\bbeta\|_2 \\
            	&\leq \|\widehat{\lambda}_j^{1/2}\widehat{\bu}_j -\lambda_j^{1/2}\widehat{\bu}_j + \lambda_j^{1/2}\widehat{\bu}_j -\lambda_j^{1/2}\bu_j\|_2 \|\bbeta\|_2
            	\\&\leq \left( \left|\widehat{\lambda}_j^{1/2} -\lambda_j^{1/2}\right| \,\|\widehat{\bu}_j\|_2 + \lambda_j^{1/2} \| \widehat{\bu}_j - \bu_j\|_2 \right) \|\bbeta\|_2
            	\\&\leq \left( |\widehat{\lambda}_j -\lambda_j|^{1/2} + \lambda_j^{1/2} \| \widehat{\bu}_j - \bu_j\|_2 \right) \|\bbeta\|_2.
                \nonumber
            \end{aligned}
        \end{equation}
        Applying Lemma~\ref{L:relbounds} we have on $\Omega_3$
        \begin{equation}
            \begin{aligned}
            	|\widehat{\lambda}_j^{1/2}\widehat{\bu}_j^\T \bbeta-\lambda_j^{1/2}\bu_j^\T \bbeta|
            	&\leq C\|\bbeta\|_2 \left(  \lambda_j^{1/2} \epsilon^{1/2} + \lambda_j^{1/2}\epsilon\sqrt{\sum\limits_{\substack{l=1\\l\neq j}}^d \frac{\lambda_j\lambda_l}{(\lambda_j-\lambda_l)^2}}\right)
            	\\&\leq C\|\bbeta\|_2 \left(\lambda_j^{1/2}  \epsilon^{1/2} + \lambda_j^{1/2}\epsilon \,j\right) ,
                \nonumber
            \end{aligned}
        \end{equation}
        where in the second inequality we used \cite{Wahl}, inequalities (3.30), or \cite{Jirak}, Lemma 7.13, together with Assumption~\ref{A:cvxdecay}. Taking maximum over $j\in[k^*]$ we obtain the claim (i), and raising to the square and summing over $j\in[k^*]$ we get the claim (ii).

\end{appendices}

\bibliographystyle{apalike}

\end{document}